\documentclass[12pt,reqno]{amsart}
\usepackage{amssymb}
\usepackage{amsxtra}
\usepackage{mathrsfs}
\usepackage[all]{xy}
\usepackage{cite}
\usepackage{paralist}
\usepackage[hypertex]{hyperref}
\numberwithin{equation}{section}
\newtheorem{theorem}{Theorem}[section]
\newtheorem{lemma}[theorem]{Lemma}
\newtheorem{prop}[theorem]{Proposition}
\newtheorem{corollary}[theorem]{Corollary}

\theoremstyle{definition}
\newtheorem{definition}[theorem]{Definition}
\theoremstyle{remark}
\newtheorem{example}[theorem]{Example}
\newtheorem{remark}[theorem]{Remark}

\DeclareMathOperator{\Hom}{Hom}
\DeclareMathOperator{\Ker}{Ker}
\renewcommand*{\Im}{\mathop{\mathrm{Im}}}

\DeclareMathOperator{\Ob}{Ob}
\DeclareMathOperator{\inv}{inv}
\DeclareMathOperator{\HCC}{HCC}
\DeclareMathOperator{\supp}{supp}
\DeclareMathOperator{\card}{card}
\newcommand*{\lmod}{\mbox{-}\!\mathop{\mathsf{mod}}}
\newcommand*{\lalg}{\mbox{-}\!\mathop{\mathsf{alg}}}
\newcommand*{\cHol}{\mathop{\mathit{Hol}}}
\newcommand*{\Ptens}{\mathop{\widehat\otimes}}
\newcommand*{\Tens}{\mathop{\otimes}}
\newcommand*{\Htens}{\mathop{\dot\otimes}}
\newcommand*{\ptens}[1]{\mathop{\widehat\otimes}_{#1}}

\newcommand*{\Pfree}{\mathop{\widehat *}}
\newcommand*{\id}{1}

\newcommand*{\reg}{\mathrm{reg}}
\newcommand*{\defo}{\mathrm{def}}
\newcommand*{\fdef}{\mathrm{fdef}}
\newcommand*{\wh}{\widehat}

\newcommand*{\la}{\langle}
\newcommand*{\ra}{\rangle}
\renewcommand*{\Re}{\mathop{\mathrm{Re}}}
\newcommand*{\CC}{\mathbb C}
\newcommand*{\N}{\mathbb N}
\newcommand*{\Z}{\mathbb Z}
\newcommand*{\R}{\mathbb R}
\newcommand*{\DD}{\mathbb D}
\newcommand*{\BB}{\mathbb B}

\newcommand*{\cO}{\mathscr O}
\newcommand*{\cA}{\mathscr A}
\newcommand*{\cC}{\mathscr C}
\newcommand*{\cF}{\mathscr F}
\newcommand*{\fF}{\mathfrak F}
\newcommand*{\fm}{\mathfrak m}
\newcommand*{\cN}{\mathscr N}
\newcommand*{\cB}{\mathscr B}
\newcommand*{\cL}{\mathscr L}
\newcommand*{\cU}{\mathscr U}
\newcommand*{\cT}{\mathscr T}
\newcommand*{\fA}{\mathfrak A}
\newcommand*{\ccH}{\mathcal H}
\newcommand*{\sT}{\mathsf T}
\newcommand*{\sU}{\mathsf U}
\newcommand*{\sE}{\mathsf E}
\newcommand*{\hI}{\tilde I}
\newcommand*{\hcO}{\tilde\cO}
\newcommand*{\hsE}{\tilde\sE}
\newcommand*{\rp}{\mathrm p}
\newcommand*{\rt}{\mathrm t}
\newcommand*{\rT}{\mathrm T}

\newcommand*{\rs}{\mathrm s}
\newcommand*{\rrm}{\mathrm m}
\newcommand*{\spn}{\mathrm{span}}
\newcommand*{\add}{\mathrm{add}}
\newcommand*{\mult}{\mathrm{mult}}
\newcommand*{\quot}{\mathrm{quot}}

\newcommand*{\AM}{\mathsf{AM}}
\newcommand*{\Bnd}{\mathsf{Bnd}}
\newcommand*{\AlgBnd}{\mathsf{AlgBnd}}
\newcommand*{\eps}{\varepsilon}
\newcommand*{\ol}{\overline}
\newcommand*{\dd}{\partial}
\newenvironment{mycompactenum}{\pltopsep=5pt\begin{compactenum}[\upshape (i)]}%
{\end{compactenum}}
\newcommand*{\xra}{\xrightarrow}

\begin{document}
\title[Nonformal deformations of the polydisk and of the ball]{Nonformal deformations\\[2pt]
of the algebras of holomorphic functions\\[2pt]
on the polydisk and on the ball in $\CC^n$}
\author{Alexei Yu. Pirkovskii}
\address{Faculty of Mathematics\\
HSE University\\
6 Usacheva, 119048 Moscow, Russia}
\email{aupirkovskii@hse.ru}
\date{}

\begin{abstract}
We construct Fr\'echet $\cO(\CC^\times)$-algebras
$\cO_\defo(\DD^n)$ and $\cO_\defo(\BB^n)$ which may be interpreted
as nonformal (or, more exactly, holomorphic) deformations of the algebras
$\cO(\DD^n)$ and $\cO(\BB^n)$
of holomorphic functions on the polydisk
$\DD^n\subset\CC^n$ and on the ball $\BB^n\subset\CC^n$, respectively.
The fibers of our algebras over $q\in\CC^\times$ are isomorphic to the
previously introduced ``quantum polydisk'' and ``quantum ball'' algebras,
$\cO_q(\DD^n)$ and $\cO_q(\BB^n)$. We show that
the algebras $\cO_\defo(\DD^n)$ and $\cO_\defo(\BB^n)$ yield continuous
Fr\'echet algebra bundles over $\CC^\times$ which are strict deformation quantizations
(in Rieffel's sense) of $\DD^n$ and $\BB^n$.
We also give a noncommutative power series interpretation of $\cO_\defo(\DD^n)$ and apply it to
showing that $\cO_\defo(\DD^n)$ is not topologically projective (and a fortiori
is not topologically free) over $\cO(\CC^\times)$.
Finally, we consider respective formal deformations of $\cO(\DD^n)$ and $\cO(\BB^n)$,
and we show that they can be obtained from the holomorphic deformations by
extension of scalars.
\end{abstract}

\maketitle

\section{Introduction}
\label{sect:intro}

Formal deformations of associative algebras are by now classical and
relatively well-studied objects. They were introduced by
Gerstenhaber \cite{Gerst_def_1}, who also obtained first deep results
on constructing formal deformations, obstructions to their existence,
their classification, and their relations
to Hochschild cohomology \cite{Gerst_def_2,Gerst_def_3,Gerst_def_4,Gerst_Schack_survey}.
An important milestone in the formal deformation theory was the
foundational paper \cite{BFFLS}, in which formal deformations
were put into the framework of deformation quantization.
This has resulted in a considerable interest in formal deformations.
The absolute highlight
of this theory is Kontsevich's celebrated theorem \cite{Konts},
which states that every Poisson manifold admits a deformation quantization.
We refer to \cite{Gerst_Schack_survey,Markl} for a thorough treatment
of formal deformation theory, and to \cite{Wald_book,Esposito} for
an introduction to deformation quantization.

By contrast, the theory of nonformal deformations is now at a much earlier
stage of development. In fact, there seems to be no unified approach to
nonformal deformations so far, and their theory still lacks general existence and
classification results. Roughly, a general feature of all existing approaches to nonformal
deformations, which distinguishes them from formal deformations, is that the role
of the ``base'' ring is now played by a certain algebra of functions
(continuous, or smooth, or holomorphic\dots) rather than by the algebra of formal power series.
This opens the possibility to ``specialize'' nonformal deformations of an algebra $A$
at nonzero values of the deformation parameter and thus to obtain actual multiplications on the
underlying vector space of $A$ (or, more exactly, on the underlying
vector space of a certain dense subalgebra of $A$, assuming that $A$ is equipped with
a suitable topology). Because of this, nonformal deformations are probably even
more natural than formal deformations, at least from the physical point of view.
Indeed, the deformation parameter is usually interpreted as Planck's constant,
which is of course a concrete nonzero physical quantity rather than a formal variable.
For a more detailed discussion of physical aspects of nonformal deformations,
and for further references, see
\cite{BRWald_Wick,BWald_disk,Wald_EMS,KRSWald_disk,Landsman_clq,Landsman_qth,Rf_qst}.

The first successful attempt to introduce and study nonformal deformations
is due to Rieffel \cite{Rf_Heis} (see also \cite{Rf_dq_oa,Rf_Lie,Rf_mem,Rf_qst,Rf_quest}),
who used a $C^*$-algebra approach to this notion. Rieffel was mostly interested
in the quantization aspect of deformations, so he coined the name ``strict deformation
quantization'' for his theory. On the technical side, Rieffel's deformations
are continuous fields
(or continuous bundles) of $C^*$-algebras whose fibers are completions of the
underlying space of a certain Poisson subalgebra $A\subset C^\infty(M)$
(where $M$ is a Poisson manifold) equipped with a family of ``deformed''
multiplications. In concrete examples, the deformed product of functions in $A$ is often given by an
oscillatory integral formula (see \cite{Rf_mem} for details).
Rieffel's approach was further developed by Landsman (see \cite{Landsman_clq} and
references therein) and applied to some problems of quantum physics.
Numerous generalizations of Rieffel's original constructions were obtained by
Bieliavsky et al. \cite{Biel_sympl,Biel_axb,Biel_solv,Biel_dq_hyper,Biel_dq_Heis,
Biel_dq_Kah,Biel_starexp_hpbd,BGNT_Kah,BGNT_aff},
by Kasprzak \cite{Kaspr_crs,Kaspr_coact,Kaspr_hmg,Kaspr_tens},
and by Neshveyev and Tuset \cite{NT_def,BGNT_Kah,BGNT_aff},
both in the $C^*$-algebraic and in the Fr\'echet algebraic settings.
See also a paper \cite{Lech_Wald_sdq} by Lechner and Waldmann, where some of Rieffel's
results are extended to nonmetrizable locally convex algebras and modules.

Another approach to nonformal deformations is based on a convergence analysis for
power series appearing in the formal deformation theory.
Perhaps the most natural choice here would be to take
the definition of a formal deformation and to replace everywhere the algebra of formal power series
by the algebra of holomorphic functions on a subset of the complex plane.
The resulting objects were introduced and studied by
Pflaum and Schottenloher \cite{Pfl_Schott} under the name of
``topologically free holomorphic deformations''.
Although this approach is quite natural and attractive, it turns out to be too restrictive
for our purposes.

The delicate point in the power series approach to nonformal deformations is that,
for many concrete examples in deformation quantization,
the formal power series defining the star product of two functions on
a Poisson manifold $M$ normally
diverges everywhere except $0$ (see \cite{Wald_EMS} for more on this).
Thus one is forced to work with some proper subalgebras of $C^\infty(M)$
(or of $\cO(M)$ in the complex analytic setting)
for which the series converges.
The first concrete example of this situation was investigated by
Omori, Maeda, Miyazaki, and Yoshioka \cite{OMMY_converg}, who proved
the convergence of the Moyal-Weyl star product on the algebra of those holomorphic
functions on $\CC^2$ which have exponential order $\le 2$ and minimal type.
See also \cite{OMMY_order} for further results in this direction.
A systematic study of convergent star products was undertaken
by Waldmann et al. \cite{BRWald_Wick,BWald_disk,ESWald_Gutt,Wald_Weyl,%
SWald_lim,KRSWald_disk,HRWald_Lie}; see \cite{Wald_EMS} for a recent survey.

Our approach to nonformal deformations is somewhat different. In some sense,
it is situated somewhere in between Rieffel's continuous deformations \cite{Rf_Heis}
and Pflaum-Schottenloher's free holomorphic deformations \cite{Pfl_Schott}.
Broadly speaking, the objects we would like to deform are
algebras of holomorphic functions on complex Stein manifolds.
Like in \cite{Pfl_Schott}, our ``base'' ring is the algebra $\cO(\Omega)$
of holomorphic functions on an open set $\Omega\subset\CC$, and our deformations
are Fr\'echet $\cO(\Omega)$-algebras.
However, in contrast to \cite{Pfl_Schott}
and to the examples considered by Waldmann's group (see the preceding paragraph),
we do not require that $R$ be topologically free over the base algebra.
The reason for that will be explained below. At the same time, we would like our
deformations to yield continuous bundles (in the sense of \cite{Gierz}) of Fr\'echet
algebras over $\Omega$ with fibers isomorphic to $R/\fm_z R$ ($z\in\Omega$),
where $\fm_z$ is the maximal ideal of $\cO(\Omega)$ corresponding to the evaluation at $z$.

In the present paper, we do not attempt to give a general definition of a holomorphic
deformation (although we have a candidate definition in mind). Instead, we
concentrate on two concrete examples, namely on the deformations of
the algebras of holomorphic functions on the open polydisk and on the open ball
in $\CC^n$. We nevertheless hope that our techniques can be applied
in a more general setting. Also, the constructions developed here provide
a more natural and conceptual approach to the algebras of holomorphic functions
on the quantum polydisk and on the quantum ball, which were
introduced and studied in our earlier papers
\cite{Pir_ncStein,Pir_ISQS21,Pir_HFG,Pir_qball_JNCG}.

To motivate our constructions, let us look at the following example,
which serves as an algebraic prototype of the deformations we are going to introduce.
Given $q\in\CC^\times$ (where $\CC^\times=\CC\setminus\{ 0\}$), consider the
algebra $\cO_q^\reg(\CC^n)$ of ``regular functions on the quantum affine space'' generated by
$n$ elements $x_1,\ldots ,x_n$ subject to the relations $x_i x_j=qx_j x_i$ for all $i<j$
(see, e.g., \cite{Br_Good}).
If $q=1$, then $\cO_q^\reg(\CC^n)$ is nothing but the
polynomial algebra $\CC[x_1,\ldots ,x_n]=\cO^\reg(\CC^n)$. Because of this,
it is common to interpret $\cO_q^\reg(\CC^n)$ as a ``deformation'' of $\cO^\reg(\CC^n)$
and to say that $\cO^\reg(\CC^n)$ is the ``limit'' of $\cO_q^\reg(\CC^n)$ as $q\to 1$.
Here the words ``deformation'' and ``limit'' are often understood heuristically.
Nevertheless, it is rather easy to give them a precise meaning.
Indeed, it is a standard fact that the monomials $x_1^{k_1}\cdots x_n^{k_n}$
$(k_1,\ldots ,k_n\ge 0)$ form a basis of $\cO_q^\reg(\CC^n)$,
so the underlying vector space of $\cO_q^\reg(\CC^n)$ can be identified with that
of $\cO^\reg(\CC^n)$. As a result, we have a family of associative
multiplications on the same underlying vector space. This situation can be axiomatized
as follows.
By a {\em Laurent deformation} of a $\CC$-algebra $A$
we mean a family $\{\star_q : q\in\CC^\times\}$
of associative multiplications on $A$ such that $\star_1$ is the initial
multiplication on $A$ and such that for every $a,b\in A$ the function
$q\in\CC^\times\mapsto a\star_q b\in A$ is an $A$-valued Laurent polynomial.
If we identify the underlying vector spaces of $\cO_q^\reg(\CC^n)$ and $\cO^\reg(\CC^n)$
as explained above, then
the resulting family of multiplications on $\cO^\reg(\CC^n)$ clearly becomes
a Laurent deformation of $\cO^\reg(\CC^n)$.

Equivalently, a Laurent deformation of $A$ can be defined as a
$\CC[t^{\pm 1}]$-algebra $R$ together with an algebra isomorphism
$R/(t-1)R\cong A$ such that $R$ is a free $\CC[t^{\pm 1}]$-module.
To see that the above definitions are indeed equivalent, observe that, if
$\{\star_q : q\in\CC^\times\}$ is a Laurent deformation of $A$ in the sense of the first definition,
then the map $A\times A\to A[t^{\pm 1}]$ taking $(a,b)$ to the Laurent polynomial
$(a\star b)(t)=a\star_t b$ uniquely extends to a $\CC[t^{\pm 1}]$-bilinear
multiplication $\star$ on $A[t^{\pm1}]$, so $R=(A[t^{\pm 1}],\star)$ is a Laurent
deformation of $A$ in the sense of the second definition. Conversely,
if $R$ is as in the second definition, then for each $q\in\CC^\times$
we have a vector space isomorphism $R/(t-q)R\cong A$, so we can let
$(A,\star_q)=R/(t-q)R$.

Our second definition of a Laurent deformation has a natural ``holomorphic'' analog in the case where
$A$ is a complete locally convex topological algebra.
Specifically, we can replace $\CC[t^{\pm 1}]$ by the algebra $\cO(\CC^\times)$ of holomorphic
functions on $\CC^\times$, and consider ``topologically free'' modules (in the sense of \cite{X1})
instead of free modules. This leads to the above-mentioned free holomorphic deformations
in the sense of \cite{Pfl_Schott}.

Let us now come back to our concrete example, i.e., to the quantum affine space.
Since the algebras $\cO_q^\reg(\CC^n)$ form a Laurent deformation of $\cO^\reg(\CC^n)$,
a natural question is then whether we can ``analytify'' this example by constructing
a free holomorphic deformation of the algebra $\cO(D)$ of holomorphic functions on
an open set $D\subset\CC^n$.
Surprisingly, it turns out that this is not possible in general.
Indeed, if such a deformation of $\cO(D)$ existed, then evaluating at every point $q\in\CC^\times$
would yield a family $\{\star_q : q\in\CC^\times\}$ of continuous associative multiplications
on $\cO(D)$ such that $x_i x_j = x_i\star_q x_j=qx_j\star_q x_i$ for all $i<j$.
However, as was observed in \cite[Remark 3.12]{Pir_qball_JNCG}, if
$D$ is either the open polydisk or the open ball in $\CC^n$, then
there is no such multiplication on $\cO(D)$ provided that $n\ge 2$ and $|q|<1$.
As a result, there is no chance
to construct a free holomorphic deformation of $\cO(D)$ in the sense of \cite{Pfl_Schott}.

Thus we see that a more flexible approach to holomorphic deformations is needed.
What we suggest here is a construction based on the following observation.
Let $R$ denote the Laurent deformation
of $\cO^\reg(\CC^n)$ introduced above. Thus $R$ is a $\CC[t^{\pm 1}]$-algebra
such that for each $q\in\CC^\times$ we have $R/(t-q)R\cong\cO_q^\reg(\CC^n)$ and such that
$R$ is free over $\CC[t^{\pm 1}]$. It is easy to see that
we have a $\CC[t^{\pm 1}]$-algebra isomorphism
\begin{equation}
\label{algdef_free_quot}
R\cong (\CC[t^{\pm 1}]\Tens F_n)/I,
\end{equation}
where $F_n=\CC\la\zeta_1,\ldots ,\zeta_n\ra$ is the free algebra and
$I$ is the two-sided ideal of $\CC[t^{\pm 1}]\Tens F_n$ generated by
$\zeta_i\zeta_j-t\zeta_j\zeta_i\; (i<j)$; cf. \cite{Good_semiclass}.
To construct an (inevitably nonfree) holomorphic deformation of $\cO(D)$, we suggest
to replace $\CC[t^{\pm 1}]$ by $\cO(\CC^\times)$
in \eqref{algdef_free_quot}, to replace $\Tens$ by $\Ptens$
(where $\Ptens$ is the completed projective tensor product),
and finally, to replace $F_n$ by a suitable ``algebra of holomorphic functions of several
free variables'' (whatever that means).

The nontrivial part of the above program is to find a suitable holomorphic
analog of $F_n$ (or, to put in another way, a suitable free analog of $\cO(D)$).
There are several approaches to the theory of holomorphic functions
of several free variables, they are closely related to each other, but are not formally equivalent.
Perhaps the best developed approach, which goes back to J.~L.~Taylor \cite{T2,T3},
deals with matrix-valued functions on $n$-tuples of square matrices of all sizes.
See \cite{KVV_book} and the second part of \cite{AMCCY} for a detailed account
of this theory and for further references. Another approach,
mostly due to Popescu (see \cite{Pop_VN,Pop_disk,%
Pop_holball,Pop_var,Pop_var_2,Pop_interpol,%
Pop_holball2,Pop_aut,Pop_ncdom,Pop_class,Pop_freebihol1,Pop_freebihol2,%
Pop_curv,Pop_Euler,Pop_ncpolydom_I,Pop_ncpolydom_II,Pop_ncvar_23}, to cite a few),
is motivated by the multivariable dilation theory and is based on
free versions of the disk algebra $A(\bar\DD)$ and of the Hardy algebra $H^\infty(\DD)$.
For a related work, see \cite{Dav_FSG,Dav_Topl,Dav_Pop,Dav_bihol,Dav_isoprobl,%
Ar_Lat_iso,Ar_Lat_class,Ar_Lat_iso2,SSS_18,SSS_20,MS_Hardy,MS_ncfun,MS_tensfunc,MS_fun_tens}.

Of course, the problem of defining a free analog of the algebra $\cO(D)$
for a general open set $D\subset\CC^n$ is already quite challenging and offers a variety of possibilities.
We concentrate here on two classical domains $D$, namely on the open polydisk $\DD^n_r\subset\CC^n$
and on the open ball $\BB^n_r\subset\CC^n$ of radius $r$.
In our view, these cases are already quite interesting and nontrivial.
What makes them especially convenient for our purposes is that reasonable free
analogs of the algebras $\cO(\DD^n_r)$ and $\cO(\BB^n_r)$ can be defined entirely in terms
of free power series. In fact, we have at least two natural candidates for a free power series
analog of $\cO(\DD^n_r)$, namely Taylor's free polydisk algebra $S_n(r)$ \cite{T2}
(denoted by $\cF^\rT(\DD^n_r)$ in this paper) and the ``universal'' free polydisk
algebra $\cF(\DD^n_r)$ introduced in \cite{Pir_HFG}. Although these algebras are not
isomorphic, we will see (see Theorem~\ref{thm:D_DT_iso}) that they lead to the same holomorphic
deformation of $\cO(\DD^n_r)$. To construct a holomorphic deformation
of $\cO(\BB^n_r)$, we use Popescu's
algebra $\cHol(\cB(\ccH)_r^n)$ of ``free holomorphic functions on the open
operatorial ball'' \cite{Pop_holball} for the same purpose.
For notational consistency, we denote Popescu's algebra by $\cF(\BB^n_r)$
in this paper.

The structure of the paper is as follows.
Section~\ref{sect:prelim} contains some preliminaries concerning the quantized algebras
$\cO_q(\DD^n_r)$ and $\cO_q(\BB^n_r)$ introduced in our earlier
papers \cite{Pir_ncStein,Pir_ISQS21,Pir_HFG,Pir_qball_JNCG}.
In Section~\ref{sect:free_poly} we interpret $\cO_q(\DD^n_r)$ as the quotient
of Taylor's free polydisk algebra $\cF^\rT(\DD^n_r)$ modulo the closed two-sided ideal
generated by $\zeta_i\zeta_j-q\zeta_j\zeta_i\; (i<j)$.
To this end, we prove that $\cF^\rT(\DD^n_r)$ has a remarkable
universal property formulated in terms of the joint spectral radius.
A similar interpretation of $\cO_q(\DD^n_r)$ as a quotient of $\cF(\DD^n_r)$
was obtained in \cite{Pir_HFG}, and Theorem~\ref{thm:q_poly_quot_free_poly} (iv) gives
a ``metric'' extension of this result.
To perform a similar construction for $\BB^n_r$, we discuss in Section~\ref{sect:free_ball}
the ``free ball'' algebra $\cF(\BB^n_r)$ introduced by G.~Popescu in \cite{Pop_holball}.
We give an alternative definition of $\cF(\BB^n_r)$, which seems to be more appropriate for our
purposes, show that Popescu's definition is equivalent to ours, and prove that
the quotient of $\cF(\BB^n_r)$ modulo the closed two-sided ideal
generated by $\zeta_i\zeta_j-q\zeta_j\zeta_i\; (i<j)$ is topologically isomorphic to $\cO_q(\BB^n_r)$.

The results of Sections~\ref{sect:free_poly} and~\ref{sect:free_ball} are then applied
in Section~\ref{sect:deforms} to constructing Fr\'echet $\cO(\CC^\times)$-algebras
which can be viewed as ``holomorphic deformations'' of $\cO(\DD^n_r)$ and $\cO(\BB^n_r)$.
Specifically, using \eqref{algdef_free_quot} as a motivation, for every
$F\in\{\cF(\DD^n_r),\cF^\rT(\DD^n_r),\cF(\BB^n_r)\}$ we consider the quotient
$(\cO(\CC^\times)\Ptens F)/I$, where $I$ is the closed two-sided ideal of
$\cO(\CC^\times)\Ptens F$ generated by $\zeta_i\zeta_j-t\zeta_j\zeta_i\; (i<j)$.
This yields three Fr\'echet $\cO(\CC^\times)$-algebras denoted by
$\cO_\defo(\DD^n_r)$, $\cO_\defo^\rT(\DD^n_r)$, and $\cO_\defo(\BB^n_r)$,
respectively. If we let $R=\cO_\defo(\DD^n_r)$ or $R=\cO_\defo^\rT(\DD^n_r)$, then
for every $q\in\CC^\times$ the fiber $R_q=R/\ol{(t-q)R}$ of $R$ over $q$
is isomorphic to $\cO_q(\DD^n_r)$.
Similarly, if we let $R=\cO_\defo(\BB^n_r)$, then
for every $q\in\CC^\times$ the fiber $R_q$ is isomorphic to $\cO_q(\BB^n_r)$.
Moreover, we show that
the $\cO(\CC^\times)$-algebras $\cO_\defo(\DD^n_r)$ and $\cO_\defo^\rT(\DD^n_r)$
not only have isomorphic fibers, but are in fact isomorphic,
in spite of the fact that $\cF(\DD^n_r)$ and $\cF^\rT(\DD^n_r)$
are not isomorphic (unless $n=1$ or $r=\infty$).

In Section~\ref{sect:nonproj} we prove that
$\cO_\defo(\DD^n_r)$ is not topologically projective (and hence is not topologically free)
over $\cO(\CC^\times)$. This fact illustrates an essential (and a bit surprising) difference
between $\cO_\defo(\DD^n_r)$ and its algebraic prototype \eqref{algdef_free_quot},
which is free over $\CC[t^{\pm 1}]$.
In Sections~\ref{sect:cont_bnd}
and~\ref{sect:Rief_quant} we show that
the Fr\'echet algebra bundles $\sE(\DD^n_r)$ and $\sE(\BB^n_r)$ associated to
$\cO_\defo(\DD^n_r)$ and $\cO_\defo(\BB^n_r)$ are continuous, and that they
form strict Fr\'echet deformation quantizations of $\cO(\DD^n_r)$ and $\cO(\BB^n_r)$
in the sense of Rieffel.

Section~\ref{sect:form_def} is devoted to finding a relationship between our
holomorphic deformations, i.e., the algebras $\cO_\defo(\DD^n_r)$ and $\cO_\defo(\BB^n_r)$,
and formal deformations. For every open set $U\subset\CC^n$, we construct a deformed
multiplication (star product) $\star$ on the topologically free $\CC[[h]]$-module
$\CC[[h]]\Ptens\cO(U)$ satisfying the relations of the quantum affine space
(see above), i.e., $x_j\star x_k=q x_k\star x_j$ for all $j<k$, where $q=e^{ih}$
and $x_1,\ldots ,x_n$ are the coordinate functions. The resulting algebra
$\cO_\fdef(U)=(\CC[[h]]\Ptens\cO(U),\star)$ is a formal deformation of $\cO(U)$
in the sense of \cite{BFGP}. The main result of Section~\ref{sect:form_def} states
that, if $U=\DD^n_r$ or $U=\BB^n_r$, then $\cO_\fdef(U)$ can be obtained from
$\cO_\defo(U)$ via extension of scalars, i.e., there is an isomorphism
$\cO_\fdef(U) \cong \CC[[h]]\ptens{\cO(\CC^\times)}\cO_\defo(U)$.
Finally, Appendix A contains some auxiliary facts on bundles of locally convex spaces and algebras.

\section{Preliminaries and notation}
\label{sect:prelim}

We shall work over the field $\CC$ of complex numbers. All algebras are assumed to
be associative and unital, and all algebra homomorphisms are assumed to be unital
(i.e., to preserve identity elements).
By a {\em Fr\'echet algebra} we mean a complete metrizable locally convex
algebra (i.e., a topological algebra whose underlying space
is a Fr\'echet space). A {\em locally $m$-convex algebra} \cite{Michael} is a topological
algebra $A$ whose topology can be defined by a family of submultiplicative
seminorms (i.e., seminorms $\|\cdot\|$ satisfying $\| ab\|\le \| a\| \| b\|$
for all $a,b\in A$). A complete locally $m$-convex algebra is called
an {\em Arens-Michael algebra} \cite{X2}.

Given a complex manifold $X$, we denote by $\cO(X)$ the algebra of holomorphic functions
on $X$. This is a Fr\'echet-Arens-Michael algebra for the topology of uniform convergence
on compact sets. The same is true in the more general setting where $X$ is a reduced
complex space \cite[\S V.6]{GR_Stein}.

Throughout we will use the following multi-index notation. Let $\Z_+=\N\cup\{ 0\}$ denote the set
of all nonnegative integers. For each $n\in\N$ and each $d\in \Z_+$, let $W_{n,d}=\{ 1,\ldots ,n\}^d$,
and let $W_n=\bigsqcup_{d\in\Z_+} W_{n,d}$. Thus a typical element of $W_n$
is a $d$-tuple $\alpha=(\alpha_1,\ldots ,\alpha_d)$ of arbitrary length
$d\in\Z_+$, where $\alpha_j\in\{ 1,\ldots ,n\}$ for all $j$. The only element
of $W_{n,0}$ will be denoted by $*$.
For each $\alpha\in W_{n,d}\subset W_n$, let $|\alpha|=d$.
Given an algebra $A$, an $n$-tuple $a=(a_1,\ldots ,a_n)\in A^n$, and
$\alpha=(\alpha_1,\ldots ,\alpha_d)\in W_n$, we let
$a_\alpha=a_{\alpha_1}\cdots a_{\alpha_d}\in A$ if $d>0$;
it is also convenient to set $a_*=1\in A$.

Let $F_n=\CC\la\zeta_1,\ldots ,\zeta_n\ra$ denote the free algebra on $n$ generators.
Recall that the set $\{ \zeta_\alpha : \alpha\in W_n\}$ is a vector space basis of $F_n$.
We let $\fF_n$ denote the algebra of all formal series $f=\sum_{\alpha\in W_n} c_\alpha\zeta_\alpha$
(where $c_\alpha\in\CC$) with the obvious multiplication extended from $F_n$.
In other words, $\fF_n=\varprojlim_{d\ge 0} F_n/I^d$, where $I$ is the ideal of
$F_n$ generated by $\zeta_1,\ldots ,\zeta_n$.

Given $k=(k_1,\ldots ,k_n)\in\Z_+^n$, we let $a^k=a_1^{k_1}\cdots a_n^{k_n}$.
If the $a_i$'s are invertible, then $a^k$ makes sense for all $k\in\Z^n$.
As usual, for each $k=(k_1,\ldots ,k_n)\in\Z^n$, we let $|k|=|k_1|+\cdots +|k_n|$.
Given $d\in\Z_+$, let
\[
(\Z_+^n)_d=\{ k\in \Z_+^n : |k|=d\}.
\]
For each $\alpha\in W_n$ and each $i\in\{1,\ldots ,n\}$, define
\[
\rp_i(\alpha)=\card\alpha^{-1}(i).
\]
Thus we have a ``projection''
\begin{equation}
\label{proj}
\rp\colon W_n\to\Z_+^n,\quad \rp(\alpha)=(\rp_1(\alpha),\ldots ,\rp_n(\alpha)).
\end{equation}
Observe that, for each $\alpha\in W_n$, we have $|\rp(\alpha)|=|\alpha|$.
It is also easy to see that
\begin{equation}
\label{card_p^{-1}}
\card\rp^{-1}(k)=\frac{|k|!}{k!} \qquad (k\in\Z_+^n).
\end{equation}

Let $q\in\CC^\times=\CC\setminus\{ 0\}$.
Recall from Section~\ref{sect:intro} that the algebra $\cO_q^\reg(\CC^n)$
{\em of regular functions on the quantum affine $n$-space} is generated by
$n$ elements $x_1,\ldots ,x_n$ subject to the relations $x_i x_j=qx_j x_i$ for all $i<j$
(see, e.g., \cite{Br_Good}). Observe that for each $\alpha\in W_n$ there exists
a unique $\rt(\alpha)\in\CC^\times$ such that
\begin{equation}
\label{t_alpha}
x_\alpha=\rt(\alpha)x^{\rp(\alpha)}.
\end{equation}
An explicit formula for $\rt(\alpha)$ will be given in Lemma~\ref{lemma:t(alpha)};
at the moment, let us only observe that $\rt(\alpha)$ is an integer power of $q$.

We will also use the standard notation related to $q$-numbers
(see, e.g., \cite{Kac_Ch,Kassel,Gasp_Rahm}).
Given $q\in\CC^\times$ and $k\in\N$, let
\[
[k]_q=1+q+\cdots +q^{k-1};\quad [k]_q!=[1]_q [2]_q \cdots [k]_q.
\]
It is also convenient to let $[0]_q!=1$. If $k=(k_1,\ldots ,k_n)\in\Z_+^n$, then
we let $[k]_q!=[k_1]_q!\cdots [k_n]_q!$.

Given $r\in (0,+\infty]$, we let $\DD^n_r$ (respectively, $\BB^n_r$) denote the open polydisk
(respectively, the open ball) of radius $r$ in $\CC^n$. Thus we have
\begin{align*}
\DD^n_r&=\Bigl\{ z=(z_1,\ldots ,z_n)\in\CC^n : \max_{1\le i\le n} |z_i|<r\Bigr\},\\
\BB^n_r&=\Bigl\{ z=(z_1,\ldots ,z_n)\in\CC^n : \sum_{i=1}^n |z_i|^2<r^2\Bigr\}.
\end{align*}
To motivate further definitions,
let us recall useful power series characterizations of the algebras
of holomorphic functions on $\DD^n_r$ and on $\BB^n_r$.
We have topological algebra isomorphisms
\begin{align}
\label{poly_power_rep}
\cO(\DD^n_r)&\cong
\biggl\{
f=\sum_{k\in\Z_+^n} c_k z^k :
\| f\|_{\DD,\rho}=\sum_{k\in\Z_+^n} |c_k| \rho^{|k|}<\infty
\;\forall \rho\in (0,r) \biggr\},\\
\label{ball_pow_rep}
\cO(\BB_r^n)&\cong
\biggl\{ f=\sum_{k\in\Z_+^n} c_k z^k :
\| f\|_{\BB,\rho}=\sum_{k\in\Z_+^n}
|c_k| \left(\frac{k!}{|k|!}\right)^{\frac{1}{2}} \rho^{|k|}<\infty
\; \forall \rho\in (0,r)\biggr\}.
\end{align}
The spaces on the right-hand side of \eqref{poly_power_rep} and \eqref{ball_pow_rep}
are subalgebras of the algebra of formal power series, $\CC[[z_1,\ldots ,z_n]]$,
and are Fr\'echet-Arens-Michael algebras under the families
$\{ \|\cdot\|_{\DD,\rho} : \rho\in (0,r)\}$
(respectively, $\{ \|\cdot\|_{\BB,\rho} : \rho\in (0,r)\}$)
of submultiplicative norms.
The isomorphisms~\eqref{poly_power_rep} and \eqref{ball_pow_rep}
take each holomorphic function on $\DD^n_r$
(respectively, on $\BB^n_r$) to its Taylor expansion at $0$.

Both isomorphisms \eqref{poly_power_rep} and \eqref{ball_pow_rep} follow from
a general result due to Aizenberg and Mityagin \cite[Theorem 4]{Aiz_Mit}.
Formally this result applies in the case where $r<\infty$ only, but a standard inverse
limit argument (cf. \cite[proof of Corollary 3.3]{Pir_qball_JNCG}) shows that the
isomorphisms hold for $r=\infty$ as well. Note also that \eqref{poly_power_rep}
follows from Aizenberg-Mityagin's theorem immediately, while deducing
\eqref{ball_pow_rep} from this theorem requires some extra work
\cite[Proposition 3.5]{Pir_qball_JNCG}. For \eqref{poly_power_rep}, see
also \cite[Theorem 3.1]{Rolewicz} and \cite[Example 27.27]{MV}.

Quantized analogs of the algebras $\cO(\DD^n_r)$ and $\cO(\BB^n_r)$
were introduced in \cite{Pir_HFG,Pir_qball_JNCG}
(see also short surveys \cite{Pir_ncStein,Pir_ISQS21}).
Given $q\in\CC^\times$, we define functions $u_q,w_q\colon\Z_+^n\to [0,+\infty)$ by
\begin{equation}
\label{w_q}
u_q(k)=|q|^{\sum_{i<j}k_i k_j},\qquad w_q(k)=\min\{ u_q(k),1\}.
\end{equation}
The algebra of {\em holomorphic functions on the quantum $n$-polydisk
of radius $r\in (0,+\infty]$} is defined by
\begin{equation*}
\cO_q(\DD^n_r)=
\biggl\{
f=\sum_{k\in\Z_+^n} c_k x^k :
\| f\|_{\DD,\rho}=\sum_{k\in\Z_+^n} |c_k| w_q(k) \rho^{|k|}<\infty
\;\forall \rho\in (0,r) \biggr\}.
\end{equation*}
As was observed in \cite{Pir_HFG}, $\cO_q(\DD^n_r)$ is a Fr\'echet algebra for
the topology generated by the norms $\|\cdot\|_{\DD,\rho}$ ($0<\rho<r$) and for
the multiplication uniquely determined by $x_i x_j=qx_j x_i$ ($i<j$).
Moreover, each norm $\|\cdot\|_{\DD,\rho}$ is submultiplicative
(cf. also \cite[Lemma 5.10]{Pir_qfree}), so that $\cO_q(\DD^n_r)$ is an Arens-Michael algebra.

\begin{remark}
\label{rem:OqCn}
If $r=\infty$, then the algebra $\cO_q(\DD^n_\infty)=\cO_q(\CC^n)$ is the
{\em Arens-Michael envelope} of $\cO_q^\reg(\CC^n)$, i.e., the completion of
$\cO_q^\reg(\CC^n)$ for the topology generated by all submultiplicative seminorms
on $\cO_q^\reg(\CC^n)$ \cite[Theorem 5.11 and Proposition 5.12]{Pir_qfree}.
Together with the above isomorphism
$\cO_1(\DD^n_r)\cong\cO(\DD^n_r)$ (see \eqref{poly_power_rep}),
this is the main motivation for our definition of $\cO_q(\DD^n_r)$.
\end{remark}

\begin{remark}
The algebra $\cO_q(\DD^n_r)$ can also be defined in a more general multiparameter setting,
see \cite[Subsection 7.4]{Pir_HFG}.
\end{remark}

The algebra of {\em holomorphic functions on the quantum $n$-ball
of radius $r\in (0,+\infty]$} is defined by
\begin{equation*}
\cO_q(\BB_r^n)=
\biggl\{ f=\sum_{k\in\Z_+^n} c_k x^k :
\| f\|_{\BB,\rho}=\sum_{k\in\Z_+^n}
|c_k| \left(\frac{[k]_{|q|^2}!}{\bigl[ |k|\bigr]_{|q|^2}!}\right)^{1/2}\!\!\!\!
u_q(k) \rho^{|k|}<\infty\; \forall \rho\in (0,r)\biggr\}
\end{equation*}
As was observed in \cite[Theorem 3.9]{Pir_qball_JNCG},
$\cO_q(\BB^n_r)$ is a Fr\'echet algebra for
the topology generated by the norms $\|\cdot\|_{\BB,\rho}$ ($0<\rho<r$) and for
the multiplication uniquely determined by $x_i x_j=qx_j x_i$ ($i<j$).
Moreover, each norm $\|\cdot\|_{\BB,\rho}$ is submultiplicative,
so that $\cO_q(\BB^n_r)$ is an Arens-Michael algebra.
Note that $\cO_1(\BB^n_r)\cong\cO(\BB^n_r)$ by \eqref{ball_pow_rep}.
Since the monomials
$x^k$ ($k\in\Z_+^n$) form a vector space basis of $\cO_q^\reg(\CC^n)$, we see that
$\cO_q^\reg(\CC^n)$ is a dense subalgebra of $\cO_q(\DD^n_r)$
and of $\cO_q(\BB^n_r)$.

A slightly shorter expression for the norms $\|\cdot\|_{\BB,\rho}$ is given by
\begin{equation}
\label{qballnorms_2}
\| f\|_{\BB,\rho}=\sum_{k\in\Z_+^n}
|c_k| \left(\frac{[k]_{|q|^{-2}}!}{\bigl[ |k|\bigr]_{|q|^{-2}}!}\right)^{1/2}\!\!\!\!
\rho^{|k|}
\end{equation}
(see \cite[Corollary 3.14]{Pir_qball_JNCG}).

For $r=\infty$, we have $\cO_q(\DD^n_\infty)=\cO_q(\BB^n_\infty)$,
both algebraically and topologically \cite[Corollary 5.3]{Pir_qball_JNCG}.
Thus our notation is consistent with the equalities $\DD^n_\infty=\BB^n_\infty=\CC^n$
(see also Remark~\ref{rem:OqCn}).
If $r<\infty$ and $|q|\ne 1$, then we still have
$\cO_q(\DD^n_r)=\cO_q(\BB^n_r)$ \cite[Theorem 5.2]{Pir_qball_JNCG}.
However, if $r<\infty$, $|q|=1$ and $n\ge 2$, then the algebras
$\cO_q(\DD^n_r)$ and $\cO_q(\BB^n_r)$ are not topologically isomorphic
\cite[Theorem 5.14]{Pir_qball_JNCG}.

\section{Quantum polydisk as a quotient of the free polydisk}
\label{sect:free_poly}

We begin this section by recalling some results from~\cite{Pir_HFG}.
Let $(A_i)_{i\in I}$ be a family of Arens-Michael algebras.
The {\em Arens-Michael free product} \cite{Cuntz_doc,Pir_HFG} of $(A_i)_{i\in I}$
is the coproduct of $(A_i)_{i\in I}$ in the category $\AM$ of Arens-Michael algebras, i.e.,
an Arens-Michael algebra $\Pfree_{i\in I} A_i$ together with a natural isomorphism
\begin{equation*}
\Hom_{\AM}(\Pfree_{i\in I} A_i,B)\cong\prod_{i\in I}\Hom_{\AM}(A_i,B)\qquad (B\in\Ob(\AM)).
\end{equation*}
The Arens-Michael free product always exists and can be constructed explicitly~\cite{Cuntz_doc,Pir_HFG}.
Clearly, the Arens-Michael free product is unique up to a unique
topological algebra isomorphism over the $A_i$'s.

Let $r>0$, and let $\DD_r=\DD^1_r$ denote the open disk of radius $r$.

\begin{definition}[\cite{Pir_HFG}]
\label{def:F_poly}
The {\em algebra of holomorphic functions on the free $n$-polydisk of
radius $r$} is
\begin{equation}
\label{F_poly}
\cF(\DD^n_r)=\cO(\DD_r)\Pfree\cdots\Pfree\cO(\DD_r).
\end{equation}
\end{definition}

To explain why $\cF(\DD^n_r)$ is indeed a natural free analog of $\cO(\DD^n_r)$,
observe that the coproduct in the category of {\em commutative} Arens-Michael algebras
is the completed projective tensor product, $\Ptens$. Recall also that for each pair
$X,Y$ of complex manifolds we have
a topological isomorphism $\cO(X)\Ptens\cO(Y)\cong\cO(X\times Y)$ \cite[II.3.3]{Groth}.
Hence replacing in \eqref{F_poly} the ``noncommutative coproduct''
by the ``commutative coproduct'' yields the algebra of holomorphic functions
on $\DD_r^n$.

\begin{remark}
\label{rem:FCn}
If $r=\infty$, then the algebra $\cF(\DD_\infty^n)=\cF(\CC^n)$ is the
Arens-Michael envelope (cf. Remark~\ref{rem:OqCn})
of the free algebra $F_n=\CC\la\zeta_1,\ldots ,\zeta_n\ra$
\cite[Proposition 4.5]{Pir_HFG}.
\end{remark}

\begin{remark}
The algebra $\cF(\DD^n_r)$ can easily be defined in the more general setting where
$r=(r_1,\ldots ,r_n)\in (0,+\infty]^n$ and $\DD^n_r$ is the polydisk of polyradius $r$
\cite[Subsection 7.2]{Pir_HFG}.
\end{remark}

The algebra $\cF(\DD_r^n)$ can also be described more explicitly as a certain subalgebra
of the algebra $\fF_n$ of free power series.
For each $i=1,\ldots ,n$, let $\zeta_i$ denote the canonical image of the
complex coordinate $z\in\cO(\DD_r)$
(i.e., of the inclusion mapping $z\colon\DD_r\hookrightarrow\CC$)
under the embedding of the $i$th factor
$\cO(\DD_r)$ into $\cF(\DD_r^n)$.
Given $d\ge 2$ and $\alpha=(\alpha_1,\ldots ,\alpha_d)\in W_n$, let
\[
\rs(\alpha)=\card\bigl\{ i \in \{ 1,\ldots ,d-1\} : \alpha_i\ne\alpha_{i+1} \bigr\}.
\]
If $|\alpha|\in\{ 0,1\}$, we let $\rs(\alpha)=|\alpha|-1$.
The next result is a special case of \cite[Proposition 7.8]{Pir_HFG}.

\begin{prop}
We have
\begin{equation}
\label{F_poly_expl}
\cF(\DD_r^n)=\Bigl\{ a=\sum_{\alpha\in W_n} c_\alpha\zeta_\alpha :
\| a\|_{\rho,\tau}=\sum_{\alpha\in W_n} |c_\alpha|\rho^{|\alpha|} \tau^{\rs(\alpha)+1}<\infty
\;\forall \rho\in (0,r),\; \forall \tau\ge 1\Bigr\}.
\end{equation}
The topology on $\cF(\DD_r^n)$ is given by the norms
$\|\cdot\|_{\rho,\tau}\; (0<\rho<r,\; \tau\ge 1)$,
and the multiplication is induced from $\fF_n$.
\end{prop}

The following universal property of $\cF(\DD^n_r)$ was proved in \cite[Proposition 7.7]{Pir_HFG}.
Given an algebra $A$ and an element $a\in A$, the spectrum of $a$ in $A$
will be denoted by $\sigma_A(a)$.

\begin{prop}
\label{prop:univ_F_poly}
Let $A$ be an Arens-Michael algebra, and let $a=(a_1,\ldots ,a_n)$ be an $n$-tuple in $A^n$
such that $\sigma_A(a_i)\subset\DD_r$ for all $i=1,\ldots ,n$. Then there exists a unique
continuous homomorphism $\gamma_a\colon\cF(\DD^n_r)\to A$
such that $\gamma_a(\zeta_i)=a_i$ for all $i=1,\ldots ,n$. Moreover,
the assignment $a\mapsto\gamma_a$ determines a natural isomorphism
\[
\Hom_\AM(\cF(\DD_r^n),A)\cong\{ a\in A^n : \sigma_A(a_i)\subset\DD_r\;\forall i=1,\ldots ,n\}
\qquad (A\in\Ob(\AM)).
\]
\end{prop}

Another algebra closely related to $\cF(\DD_n^r)$ was introduced by J.~L.~Taylor \cite{T2,T3}.
We will define it in a slightly more general context.
For a Banach space $E$, the {\em analytic tensor algebra} $\wh{T}(E)$
(\cite{Cuntz_doc}; cf. also \cite{Vogt1,Vogt2,Pir_qfree}) is given by
\[
\wh{T}(E)=\Bigl\{ a=\sum_{d=0}^\infty a_d : a_d\in E^{\wh{\otimes}d},\;
\| a\|_{\rho}=\sum_d \| a_d\| \rho^d<\infty\;\forall\rho>0\Bigr\},
\]
where $E^{\wh{\otimes}d}=E\Ptens\cdots \Ptens E$ is the $d$th completed
projective tensor power of $E$.
The topology on $\wh{T}(E)$ is given by the norms
$\|\cdot\|_{\rho}\; (\rho>0)$, and the multiplication
on $\wh{T}(E)$ is given by concatenation, like on the usual tensor algebra $T(E)$.
Each norm $\|\cdot\|_{\rho}$ is easily seen to be
submultiplicative, and so $\wh{T}(E)$ is an Arens-Michael algebra containing
$T(E)$ as a dense subalgebra. As was observed by J.~Cuntz \cite{Cuntz_doc},
$\wh{T}(E)$ has the universal property that, for every Arens-Michael algebra $A$,
each continuous linear map $E\to A$ uniquely extends to a continuous
homomorphism $\wh{T}(E)\to A$. In other words, there is a natural isomorphism
\begin{equation*}
\Hom_\AM(\wh{T}(E),A)\cong\cL(E,A)\qquad (A\in\Ob(\AM)),
\end{equation*}
where $\cL(E,A)$ is the space of all continuous linear maps from $E$ to $A$.
Note that $\wh{T}(E)$ was originally defined in the more general setting where $E$
is a complete locally convex space \cite{Cuntz_doc}, but this generality
is not needed here.

Fix now $r>0$, and let
\[
\wh{T}_r(E)=\Bigl\{ a=\sum_{d=0}^\infty a_d : a_d\in E^{\wh{\otimes}d},\;
\| a\|_{\rho}=\sum_d \| a_d\| \rho^d<\infty\;\forall\rho\in (0,r)\Bigr\}.
\]
It follows from the above discussion that $\wh{T}_r(E)$ is an Arens-Michael algebra containing
$T(E)$ as a dense subalgebra. Note that $\wh{T}_r(E)$ essentially depends on the fixed
norm on $E$ (in contrast to $\wh{T}(E)$, which depends only on the topology of $E$).

\begin{definition}
Let $\CC^n_1$ be the vector space $\CC^n$ endowed with the $\ell^1$-norm
$\| x\|=\sum_{i=1}^n |x_i|$  (where $x=(x_1,\ldots ,x_n)\in\CC^n$).
The algebra $\wh{T}_r(\CC^n_1)$ will be denoted by $\cF^\rT(\DD_r^n)$ and will be called
{\em Taylor's algebra of holomorphic functions on the free $n$-polydisk of
radius $r$}.
\end{definition}

Using the canonical isometric isomorphisms $\CC^m_1\Ptens\CC^n_1\cong\CC^{mn}_1$,
we see that
\begin{equation}
\label{F_poly_T}
\cF^\rT(\DD_r^n)=\Bigl\{ a=\sum_{\alpha\in W_n} c_\alpha\zeta_\alpha\in\fF_n :
\| a\|_\rho=\sum_{\alpha\in W_n} |c_\alpha|\rho^{|\alpha|}<\infty
\;\forall \rho\in (0,r)\Bigr\}.
\end{equation}
The algebra $\cF^\rT(\DD_r^n)$ was introduced by J.~L.~Taylor \cite{T2,T3}
and was denoted by $S(r)$ in \cite{T2} and by $\cF_n(r)$ in \cite{T3}.
Our notation hints that both $\cF(\DD_r^n)$ and $\cF^\rT(\DD_r^n)$ are natural
candidates for the algebra of holomorphic functions on the free polydisk;
the superscript ``$\rT$'' is for ``Taylor''.

\begin{prop}
\label{prop:FD_vs_FTD}
We have $\cF(\DD_r^n)\subset\cF^\rT(\DD_r^n)$, and the embedding
$\cF(\DD_r^n)\to\cF^\rT(\DD_r^n)$ is continuous.
For $r=\infty$, this embedding is a topological algebra isomorphism
(thus we have $\cF(\CC^n)=\cF^\rT(\CC^n)$, both algebraically
and topologically).
\end{prop}
\begin{proof}
By looking at \eqref{F_poly_expl} and~\eqref{F_poly_T}, and by using the obvious inequality
$\rs(\alpha)+1\le |\alpha|$, we see that
\begin{equation}
\label{2norms_Fn}
\|\cdot\|_\rho\le\|\cdot\|_{\rho,\tau}\le\|\cdot\|_{\rho\tau} \quad\text{on $\fF_n$}
\qquad (\rho>0,\;\tau\ge 1).
\end{equation}
The first inequality in \eqref{2norms_Fn} yields a continuous inclusion
of $\cF(\DD^n_r)$ into $\cF^\rT(\DD^n_r)$. On the other hand,
the second inequality in \eqref{2norms_Fn} implies that
the inclusion map is a topological isomorphism for $r=\infty$.
\end{proof}

\begin{remark}
It is easy to observe that $\cF(\DD_r^n)\ne\cF^\rT(\DD_r^n)$
unless $r=\infty$ or $n=1$; for instance,
the element $\sum_k r^{-2k} (\zeta_1\zeta_2)^k$
belongs to $\cF^\rT(\DD_r^n)$, but does not belong to $\cF(\DD_r^n)$.
Moreover, $\cF(\DD^n_r)$ is nuclear as a locally convex space \cite{Pir_HFG},
while $\cF^\rT(\DD^n_r)$ is not \cite{T2,Lum_PI}, so they are not isomorphic even as locally
convex spaces.
\end{remark}

Our next goal is to show that $\cF^\rT(\DD_r^n)$ has a remarkable
universal property similar in spirit to Proposition~\ref{prop:univ_F_poly}.
Before introducing this property, let us recall the following definition.
Let $A$ be an Arens-Michael algebra, and let $\{ \|\cdot\|_\lambda : \lambda\in\Lambda\}$
be a directed defining family of submultiplicative seminorms on $A$.
For an $n$-tuple $a=(a_1,\ldots ,a_n)\in A^n$, the {\em joint spectral radius} of
$a$ is defined by
\begin{equation}
\label{j-sprad}
r_\infty(a)=\sup_{\lambda\in\Lambda}%
\lim_{d\to\infty}\Bigl(\sup_{\alpha\in W_{n,d}} \| a_\alpha\|_\lambda\Bigr)^{1/d}.
\end{equation}

\begin{remark}
The joint spectral radius was introduced in \cite{RS} in the case where $A$ is a Banach algebra.
An extension to locally convex algebras (not necessarily locally $m$-convex) was discussed
in \cite{Solt2} under the assumption that the $a_i$'s pairwise commute.
More generally, one can consider the joint $\ell^p$-spectral radius, $r_p(a)$,
for every $p\in [1,+\infty]$
(see \cite[V.35 and Comments to Chapter V]{Muller_book} for the Banach algebra case,
and \cite[Section~5]{Pir_qball_JNCG} for the case of Arens-Michael algebras). Definition~\eqref{j-sprad}
corresponds to the case where $p=\infty$.
\end{remark}

By \cite[Proposition 5.7]{Pir_qball_JNCG}, the definition of $r_\infty(a)$ does not depend on
the choice of a directed defining family of submultiplicative seminorms on $A$.
If $\varphi\colon A\to B$ is a continuous homomorphism of Arens-Michael algebras,
then for every $a=(a_1,\ldots ,a_n)\in A^n$ we have $r_\infty(\varphi(a))\le r_\infty(a)$, where
$\varphi(a)=(\varphi(a_1),\ldots ,\varphi(a_n))$
\cite[Proposition 5.8]{Pir_qball_JNCG}.

\begin{definition}
Let $A$ be an Arens-Michael algebra, and let $r>0$.
We say that an $n$-tuple $a\in A^n$ is
{\em strictly spectrally $r$-contractive} if, for each Banach algebra $B$ and each continuous
homomorphism $\varphi\colon A\to B$, we have $r_\infty(\varphi(a))<r$.
\end{definition}

\begin{remark}
\label{rem:contr_image}
Observe that, if $A$ and $B$ are Arens-Michael algebras, $\psi\colon A\to B$ is a continuous
homomorphism, and $a\in A^n$ is strictly spectrally $r$-contractive, then so is $\psi(a)\in B^n$.
\end{remark}

An equivalent but more convenient definition is as follows.

\begin{prop}
\label{prop:spec_contr}
Let $A$ be an Arens-Michael algebra, and let $\{\|\cdot\|_\lambda : \lambda\in\Lambda\}$
be a directed defining family of submultiplicative seminorms on $A$. For each $\lambda\in\Lambda$,
let $A_\lambda$ denote the completion of $A$ with respect to $\|\cdot\|_\lambda$.
Given $a\in A^n$, let $a_\lambda$ denote the canonical image of $a$ in $A_\lambda^n$.
Then the following conditions are equivalent:
\begin{mycompactenum}
\item $a$ is strictly spectrally $r$-contractive;
\item $r_\infty(a_\lambda)<r$ for all $\lambda\in\Lambda$;
\item $\lim_{d\to\infty}\bigl(\sup_{\alpha\in W_{n,d}} \| a_\alpha\|_\lambda\bigr)^{1/d}<r$
for all $\lambda\in\Lambda$.
\end{mycompactenum}
\end{prop}
\begin{proof}
$\mathrm{(i)}\Longrightarrow\mathrm{(ii)}\iff\mathrm{(iii)}$. This is clear.

$\mathrm{(ii)}\Longrightarrow\mathrm{(i)}$.
Let $B$ be a Banach algebra, and let $\varphi\colon A\to B$ be a continuous homomorphism.
There exist $\lambda\in\Lambda$ and $C>0$ such that for all $a\in A$ we have
$\| \varphi(a)\| \le C \| a\|_\lambda$. Hence there exists a unique continuous homomorphism
$\psi\colon A_\lambda\to B$ such that $\varphi=\psi\tau_\lambda$, where
$\tau_\lambda\colon A\to A_\lambda$ is the canonical map. We have
\[
r_\infty(\varphi(a))=r_\infty(\psi(a_\lambda)) \le r_\infty(a_\lambda)<r,
\]
and so $a$ is strictly spectrally $r$-contractive.
\end{proof}

\begin{corollary}
\label{cor:Ban_spec_contr}
If $A$ is a Banach algebra, then $a\in A^n$ is strictly spectrally $r$-contractive
if and only if $r_\infty(a)<r$.
\end{corollary}

\begin{example}
\label{ex:zeta_spec_contr}
Let $\zeta=(\zeta_1,\ldots ,\zeta_n)\in\cF^\rT(\DD_r^n)^n$.
For each $\rho\in (0,r)$, each $d\in\Z_+$,
and each $\alpha\in W_{n,d}$, we have $\| \zeta_\alpha\|_\rho^{1/d}=\rho$.
Hence $\zeta\in\cF^\rT(\DD_r^n)^n$ is strictly spectrally
$r$-contractive, but is not strictly spectrally $r'$-contractive whenever $r'<r$.
The same assertion holds for the $n$-tuple $z=(z_1,\ldots ,z_n)\in\cO(\DD_r^n)^n$
of coordinate functions on $\DD_r^n$.
Note that such a phenomenon can never happen in a Banach algebra
(see Corollary~\ref{cor:Ban_spec_contr}).
\end{example}

\begin{example}
Let $\zeta=(\zeta_1,\ldots ,\zeta_n)\in\cF(\DD_r^n)^n$, where $n\ge 2$.
For each $\rho\in (0,r)$, each $\tau\ge 1$, each $d\in\Z_+$,
and each $\alpha\in W_{n,d}$, we have
$\| \zeta_\alpha\|_{\rho,\tau}^{1/d}=\rho\tau^{(s(\alpha)+1)/d}$.
In particular, for $\alpha=(1,2,\ldots ,1,2)\in W_{n,2d}$
we have $\|\zeta_\alpha\|_{\rho,\tau}^{1/2d}=\rho\tau$.
Hence $\zeta\in\cF(\DD_r^n)^n$ is not strictly spectrally
$R$-contractive for any $R>0$.
\end{example}

\begin{prop}
\label{prop:univ_F_poly_T}
Let $A$ be an Arens-Michael algebra, and let $a=(a_1,\ldots ,a_n)$ be a strictly
spectrally $r$-contractive $n$-tuple in $A^n$. Then there exists a unique
continuous homomorphism $\gamma_a\colon\cF^\rT(\DD^n_r)\to A$
such that $\gamma_a(\zeta_i)=a_i$ for all $i=1,\ldots ,n$. Moreover,
the assignment $a\mapsto\gamma_a$ determines a natural isomorphism
\begin{equation}
\label{univ_F_poly_T}
\Hom_\AM(\cF^\rT(\DD_r^n),A)\cong\{ a\in A^n : a\; \text{\upshape is str. spec. $r$-contr.}\}
\qquad (A\in\Ob(\AM)).
\end{equation}
\end{prop}
\begin{proof}
Let $\|\cdot\|$ be a continuous submultiplicative seminorm on $A$.
Using Proposition~\ref{prop:spec_contr}, we can choose $\rho>0$ such that
\[
\lim_{d\to\infty}\Bigl(\sup_{\alpha\in W_{n,d}} \| a_\alpha\|\Bigr)^{1/d}<\rho<r.
\]
Hence there exists $d_0\in \Z_+$ such that
\begin{equation}
\label{free_calc_cont_1}
\sup_{\alpha\in W_{n,d}} \| a_\alpha\| < \rho^d \qquad (d\ge d_0).
\end{equation}
Choose now $C\ge 1$ such that
$\| a_\alpha\|\le C\rho^{|\alpha|}$ whenever $|\alpha|<d_0$.
Together with~\eqref{free_calc_cont_1}, this yields
$\| a_\alpha\|\le C\rho^{|\alpha|}$ for all $\alpha\in W_n$.
Hence for each $f=\sum_\alpha c_\alpha\zeta_\alpha\in\cF^\rT(\DD_r^n)$ we obtain
\begin{equation*}
\sum_{\alpha\in W_n} |c_\alpha| \| a_\alpha\|
\le C \sum_{\alpha\in W_n} |c_\alpha| \rho^{|\alpha|}
= C \| f\|_\rho.
\end{equation*}
Therefore the series $\sum_\alpha c_\alpha a_\alpha$ absolutely converges in $A$, and the
mapping
\[
\gamma_a\colon\cF^\rT(\DD_r^n)\to A,\quad
\sum_{\alpha\in W_n} c_\alpha \zeta_\alpha \mapsto \sum_{\alpha\in W_n} c_\alpha a_\alpha
\]
is the required continuous homomorphism. The uniqueness of $\gamma_a$ is immediate
from the density of the free algebra $F_n$ in $\cF^\rT(\DD_n^r)$.

Conversely, using Example~\ref{ex:zeta_spec_contr} and Remark~\ref{rem:contr_image}, we see that,
for each Arens-Michael algebra $A$ and each continuous homomorphism
$\varphi\colon \cF^\rT(\DD_r^n)\to A$, the $n$-tuple $\varphi(\zeta)\in A^n$ is strictly
spectrally $r$-contractive. Thus~\eqref{univ_F_poly_T} is indeed a bijection, as required.
\end{proof}

\begin{remark}
Proposition \ref{prop:univ_F_poly_T} can easily be extended to the algebra $\wh{T}_r(E)$
for any Banach space $E$. Now the set of strictly spectrally $r$-contractive $n$-tuples on
the right-hand side of~\eqref{univ_F_poly_T} should be replaced by
the set of all $\psi\in\cL(E,A)$ such that the image of the unit ball of $E$ under $\psi$
is a strictly spectrally $r$-contractive set in $A$.
Related results can be found in~\cite{Dos_Slodk}.
\end{remark}

\begin{theorem}
\label{thm:q_poly_quot_free_poly}
Let $q\in\CC^\times$, $n\in\N$, and $r\in (0,+\infty]$.
\begin{mycompactenum}
\item There exists a surjective continuous homomorphism
\begin{equation}
\label{pi}
\pi \colon \cF(\DD_r^n)\to\cO_q(\DD_r^n), \qquad
\zeta_i\mapsto x_i \quad (i=1,\ldots ,n).
\end{equation}
\item $\Ker\pi$ coincides with
the closed two-sided ideal of $\cF(\DD_r^n)$
generated by the elements
$\zeta_i\zeta_j-q\zeta_j\zeta_i\; (i,j=1,\ldots ,n,\; i<j)$.
\item $\Ker\pi$ is a complemented subspace of $\cF(\DD_r^n)$.
\item Under the identification $\cO_q(\DD_r^n)\cong\cF(\DD^n_r)/\Ker\pi$,
the norm $\|\cdot\|_{\DD,\rho}$ on $\cO_q(\DD_r^n)$ is equal to the quotient of
the norm $\|\cdot\|_{\rho,\tau}$ on $\cF(\DD^n_r)\; (\rho\in (0,r),\; \tau\ge 1)$.
\end{mycompactenum}
\end{theorem}

Parts (i)--(iii) of the above theorem were proved in~\cite[Theorem 7.13]{Pir_HFG}
in the more general multiparameter case. The proof of the ``quantitative''
part (iv) will be given below, together
with the proof of Theorem~\ref{thm:q_poly_quot_free^T_poly}.

\begin{theorem}
\label{thm:q_poly_quot_free^T_poly}
Let $q\in\CC^\times$, $n\in\N$, and $r\in (0,+\infty]$.
\begin{mycompactenum}
\item There exists a surjective continuous homomorphism
\begin{equation}
\label{pi^T}
\pi^\rT \colon \cF^\rT(\DD_r^n)\to\cO_q(\DD_r^n), \qquad
\zeta_i\mapsto x_i \quad (i=1,\ldots ,n).
\end{equation}
\item $\Ker\pi^\rT$ coincides with
the closed two-sided ideal of $\cF^\rT(\DD_r^n)$
generated by the elements
$\zeta_i\zeta_j-q\zeta_j\zeta_i\; (i,j=1,\ldots ,n,\; i<j)$.
\item $\Ker\pi^\rT$ is a complemented subspace of $\cF^\rT(\DD_r^n)$.
\item Under the identification $\cO_q(\DD_r^n)\cong\cF^\rT(\DD^n_r)/\Ker\pi^\rT$,
the norm $\|\cdot\|_{\DD,\rho}$ on $\cO_q(\DD_r^n)$ is equal to the quotient of
the norm $\|\cdot\|_\rho$ on $\cF^\rT(\DD^n_r)\; (\rho\in (0,r))$.
\end{mycompactenum}
\end{theorem}

To prove Theorem~\ref{thm:q_poly_quot_free^T_poly}, we need some notation from \cite{Pir_qfree}.
For each $d\in\Z_+$, let the symmetric group $S_d$ act on $W_{n,d}$ via
$\sigma(\alpha)=\alpha\sigma^{-1}\; (\alpha\in W_{n,d},\; \sigma\in S_d)$.
Clearly, for each $\alpha\in W_{n,d}$ and $\sigma\in S_d$ there exists a unique
$\lambda(\sigma,\alpha)\in\CC^\times$ such that
\begin{equation*}
x_\alpha=\lambda(\sigma,\alpha)x_{\sigma(\alpha)}.
\end{equation*}
Given $k=(k_1,\ldots ,k_n)\in\Z_+^n$ with $|k|=d$, let
\[
\delta(k)=(\underbrace{1,\ldots ,1}_{k_1},\ldots , \underbrace{n,\ldots ,n}_{k_n})\in W_{n,d}.
\]
By \cite[Proposition 5.12]{Pir_qfree}, we have
\begin{equation}
\label{w_q_2}
w_q(k)=\min\{ |\lambda(\sigma,\delta(k))| : \sigma\in S_d\}.
\end{equation}
Clearly, for each $\alpha\in \rp^{-1}(k)$ (where $\rp$ is given by~\eqref{proj})
there exists $\sigma_\alpha\in S_d$ such that
$\alpha=\sigma_\alpha(\delta(k))$. Therefore $x^k=\lambda(\sigma_\alpha,\delta(k))x_\alpha$.
Comparing with~\eqref{t_alpha}, we see that
\begin{equation}
\label{t_alpha_2}
\lambda(\sigma_\alpha,\delta(k))=\rt(\alpha)^{-1}.
\end{equation}
Observe also that, if $\alpha\in W_{n,d}$ and $\sigma_1,\sigma_2\in S_d$ are such that
$\sigma_1(\alpha)=\sigma_2(\alpha)$, then $\lambda(\sigma_1,\alpha)=\lambda(\sigma_2,\alpha)$.
In other words, $\lambda(\sigma,\alpha)$ depends only on $\alpha$
and $\sigma(\alpha)$. Since the orbit of $\delta(k)$ under the action of $S_d$ is exactly $\rp^{-1}(k)$,
we conclude from~\eqref{w_q_2} and~\eqref{t_alpha_2} that
\begin{equation}
\label{w_q_3}
w_q(k)=\min\{ |\rt(\alpha)|^{-1} : \alpha\in \rp^{-1}(k)\}.
\end{equation}

\begin{proof}[Proof of Theorem~\upshape{\ref{thm:q_poly_quot_free^T_poly}}]
Let $x=(x_1,\ldots ,x_n)\in\cO_q(\DD^n_r)^n$. For each $\rho\in (0,r)$, each $d\in\Z_+$,
and each $\alpha\in W_{n,d}$, we clearly have $\| x_\alpha\|_{\DD,\rho}\le\rho^d$.
Hence $x$ is strictly spectrally $r$-contractive, and Proposition~\ref{prop:univ_F_poly_T}
yields the required homomorphism~\eqref{pi^T}.
The homomorphism $\pi$ given by~\eqref{pi} is then the composition
of $\pi^\rT$ with the canonical embedding
\[
\nu\colon\cF(\DD^n_r)\to\cF^\rT(\DD^n_r),\qquad \zeta_i\mapsto\zeta_i\quad (i=1,\ldots ,n).
\]
By Theorem~\ref{thm:q_poly_quot_free_poly}, $\pi$ is onto, and hence so is $\pi^\rT$.
This proves (i).

Since $\Ker\pi$ is a complemented subspace of $\cF(\DD^n_r)$,
there exists a continuous linear map $\varkappa\colon\cO_q(\DD^n_r)\to\cF(\DD^n_r)$
such that $\pi\varkappa=\id$. Letting
$\varkappa^\rT=\nu\varkappa\colon\cO_q(\DD^n_r)\to\cF^\rT(\DD^n_r)$,
we see that $\pi^\rT\varkappa^\rT=\id$, whence $\Ker\pi^\rT$ is a complemented subspace
of $\cF^\rT(\DD^n_r)$. This proves (iii).

Using the density of $\Im\nu$ in $\cF^\rT(\DD^n_r)$, we obtain
\[
\Ker\pi^\rT
=\Im(\id-\varkappa^\rT\pi^\rT)
=\ol{\Im\bigl((\id-\varkappa^\rT\pi^\rT)\nu\bigr)}
=\ol{\Im\bigl(\nu(\id-\varkappa\pi)\bigr)}
=\ol{\nu(\Ker\pi)}.
\]
Now (ii) follows from Theorem~\ref{thm:q_poly_quot_free_poly} (ii).

Since \eqref{pi} and \eqref{pi^T} are surjective, the Open Mapping Theorem yields topological
isomorphisms
\[
\cO_q(\DD_r^n)\cong\cF(\DD^n_r)/\Ker\pi\cong \cF^\rT(\DD^n_r)/\Ker\pi^\rT.
\]
To prove parts (iv) of Theorems~\ref{thm:q_poly_quot_free_poly}
and~\ref{thm:q_poly_quot_free^T_poly}, we have to show that
\begin{equation}
\label{three_norms}
\| \cdot\|_{\DD,\rho}=\|\cdot\|_\rho^\wedge=\| \cdot\|_{\rho,\tau}^\wedge,
\end{equation}
where $\|\cdot\|_\rho^\wedge$ and $\| \cdot\|_{\rho,\tau}^\wedge$ are the quotient
norms of $\|\cdot\|_\rho$ and $\| \cdot\|_{\rho,\tau}$, respectively.
Let $f=\sum_{\alpha\in W_n} c_\alpha \zeta_\alpha\in\cF^\rT(\DD^n_r)$.
We have
\[
\pi^\rT(f)=\sum_{\alpha\in W_n} c_\alpha x_\alpha
=\sum_{k\in\Z_+^n}\biggl(\sum_{\alpha\in \rp^{-1}(k)} c_\alpha \rt(\alpha)\biggr) x^k.
\]
Together with~\eqref{w_q_3}, this yields
\begin{equation}
\label{Drho<=rho}
\begin{split}
\| \pi^\rT(f)\|_{\DD,\rho}
&\le \sum_{k\in\Z_+^n}\biggl(\sum_{\alpha\in \rp^{-1}(k)} |c_\alpha \rt(\alpha)|\biggr)
\| x^k\|_{\DD,\rho}\\
&\le \sum_{k\in\Z_+^n} \biggl( \max_{\alpha\in \rp^{-1}(k)} |\rt(\alpha)|
\sum_{\alpha\in \rp^{-1}(k)} |c_\alpha| \biggr) w_q(k) \rho^{|k|}\\
&= \sum_{k\in\Z_+^n} \biggl(\sum_{\alpha\in \rp^{-1}(k)} |c_\alpha|\biggr) \rho^{|k|}
= \| f\|_\rho.
\end{split}
\end{equation}
If now $f\in\cF(\DD^n_r)$ and $\tau\ge 1$, then
\begin{equation}
\label{Drho<=rho_tau}
\| \pi(f)\|_{\DD,\rho}=\| \pi^\rT(\nu(f))\|_{\DD,\rho}\le \|\nu(f)\|_\rho=\| f\|_{\rho,1}\le\| f\|_{\rho,\tau}.
\end{equation}
From~\eqref{Drho<=rho} and \eqref{Drho<=rho_tau}, we conclude that
$\|\cdot\|_{\DD,\rho}\le\| \cdot\|_\rho^\wedge$ and
$\|\cdot\|_{\DD,\rho}\le\|\cdot\|_{\rho,\tau}^\wedge$.
On the other hand, both $\|\cdot\|_\rho^\wedge$ and $\| \cdot\|_{\rho,\tau}^\wedge$ are
submultiplicative norms on $\cO_q(\DD^n_r)$, and we have
$\| x_i\|_\rho^\wedge\le \|\zeta_i\|_\rho=\rho$ and
$\| x_i\|_{\rho,\tau}^\wedge\le \|\zeta_i\|_{\rho,\tau}=\rho$.
By the maximality property of $\|\cdot\|_{\DD,\rho}$
\cite[Lemma 5.10]{Pir_qfree}, it follows that
$\|\cdot\|_\rho^\wedge\le\|\cdot\|_{\DD,\rho}$ and
$\|\cdot\|_{\rho,\tau}^\wedge\le\|\cdot\|_{\DD,\rho}$.
Together with the above estimates, this gives~\eqref{three_norms}
and completes the proof
of (iv), both for Theorem~\ref{thm:q_poly_quot_free_poly}
and Theorem~\ref{thm:q_poly_quot_free^T_poly}.
\end{proof}

\section{Quantum ball as a quotient of the free ball}
\label{sect:free_ball}
The goal of this section is to prove a quantum ball analog of Theorems~\ref{thm:q_poly_quot_free_poly}
and \ref{thm:q_poly_quot_free^T_poly}.
Towards this goal, it will be convenient to introduce a
``hilbertian'' version of the algebra $\wh{T}_r(E)$. In what follows,
given Hilbert spaces $H_1$ and $H_2$, their Hilbert tensor product will
be denoted by $H_1\Htens H_2$. The Hilbert tensor product of $n$ copies
of a Hilbert space $H$ will be denoted by $H^{\dot\otimes n}$.

Given $r>0$ and a Hilbert space $H$, let
\[
\dot T_r(H)=\Bigl\{ a=\sum_{d=0}^\infty a_d : a_d\in H^{\dot\otimes d},\;
\| a\|_{\rho}^\bullet=\sum_d \| a_d\| \rho^d<\infty\;\forall\rho\in (0,r)\Bigr\}.
\]
Clearly, $\dot T_r(H)$ is a Fr\'echet space for the topology
determined by the norms $\|\cdot\|_\rho^\bullet\; (\rho>0)$.
Similarly to the case of $\wh{T}_r(E)$, it is easy to check that
each norm $\|\cdot\|_\rho^\bullet$ is submultiplicative on the
tensor algebra $T(H)\subset \dot T_r(H)$. Therefore there exists
a unique continuous multiplication on $\dot T_r(H)$ extending that of $T(H)$,
and $\dot T_r(H)$ becomes an Arens-Michael algebra containing
$T(H)$ as a dense subalgebra.

\begin{definition}
\label{def:F_ball}
Let $\CC^n_2$ be the vector space $\CC^n$ endowed with the inner product
$\la x,y\ra=\sum_{i=1}^n x_i\bar y_i$.
The algebra $\dot T_r(\CC^n_2)$ will be denoted by $\cF(\BB_r^n)$ and will be called
{\em the algebra of holomorphic functions on the free $n$-ball of
radius $r$}.
\end{definition}

Using the canonical isometric isomorphisms $\CC^m_2\Htens\CC^n_2\cong\CC^{mn}_2$,
we see that
\begin{equation}
\label{F_ball}
\cF(\BB_r^n)=\biggl\{ a=\sum_{\alpha\in W_n} c_\alpha\zeta_\alpha\in\fF_n :
\| a\|_\rho^\bullet=\sum_{d=0}^\infty\Bigl(\sum_{|\alpha|=d}%
|c_\alpha|^2\Bigr)^{1/2} \rho^d<\infty
\;\forall \rho\in (0,r)\biggr\}.
\end{equation}

Let us first compare $\cF(\BB^n_r)$ with $\cF^\rT(\DD^n_r)$ and observe,
in particular, that Definition~\ref{def:F_ball} is consistent with
our convention that $\BB_\infty^n=\DD^n_\infty=\CC^n$
(cf. Proposition~\ref{prop:FD_vs_FTD}).

\begin{prop}
\label{prop:FTD_vs_FB}
We have $\cF^\rT(\DD_r^n)\subset\cF(\BB_r^n)$, and the embedding
$\cF^\rT(\DD_r^n)\to\cF(\BB_r^n)$ is continuous.
For $r=\infty$, this embedding is a topological algebra isomorphism.
\end{prop}
\begin{proof}
Given $a=\sum_\alpha c_\alpha \zeta_\alpha\in\fF_n$ and $\rho>0$, we have
\[
\| a\|_\rho^\bullet\le \sum_d \Bigl(\sum_{|\alpha|=d} |c_\alpha|\Bigr) \rho^d=\| a\|_\rho.
\]
This yields a continuous inclusion
of $\cF^\rT(\DD^n_r)$ into $\cF(\BB^n_r)$. On the other hand,
by using the Cauchy-Bunyakowsky-Schwarz inequality, we obtain
\[
\| a\|_\rho\le \sum_d\Bigl(\sum_{|\alpha|=d} |c_\alpha|^2\Bigr)^{1/2} n^{d/2} \rho^d
= \| a\|_{\rho\sqrt{n}}^\bullet.
\]
This shows that the inclusion map is a topological isomorphism for $r=\infty$.
\end{proof}

Our next goal is to show that $\cF(\BB^n_r)$ coincides with the
algebra $\cHol(\cB(\ccH)_r^n)$ of ``free holomorphic functions on the open
operatorial ball'' introduced by G.~Popescu \cite{Pop_holball}.
To this end, let us recall some results from \cite{Pop_holball}.

Let $H$ be a Hilbert space, let $\cB(H)$ denote the algebra of bounded linear operators on $H$,
and let $T=(T_1,\ldots ,T_n)$ be an $n$-tuple in $\cB(H)^n$.
Following \cite{Pop_holball}, we identify $T$ with the ``row'' operator
acting from the Hilbert direct sum $H^n=H\oplus\cdots\oplus H$ to $H$.
Thus we have $\| T\|=\|\sum_{i=1}^n T_i T_i^*\|^{1/2}$.

Given a free formal series
$f=\sum_{\alpha\in W_n} c_\alpha\zeta_\alpha\in\fF_n$,
the {\em radius of convergence} $R(f)\in [0,+\infty]$
is given by
\[
\frac{1}{R(f)}=\limsup_{d\to\infty}%
\Bigl(\sum_{|\alpha|=d} |c_\alpha|^2\Bigr)^{\frac{1}{2d}}.
\]
By \cite[Theorem 1.1]{Pop_holball}, for each $T\in\cB(H)^n$ such that
$\| T\|<R(f)$, the series
\begin{equation}
\label{Pop_series}
\sum_{d=0}^\infty \Bigl(\sum_{|\alpha|=d} c_\alpha T_\alpha \Bigr)
\end{equation}
converges in $\cB(H)$ and, moreover, $\sum_d \| \sum_{|\alpha|=d} c_\alpha T_\alpha\|<\infty$.
On the other hand, if $H$ is infinite-dimensional,
then for each $R'>R(f)$ there exists $T\in\cB(H)^n$ with $\| T\|=R'$
such that the series~\eqref{Pop_series} diverges.
The collection of all $f\in\fF_n$ such that $R(f)\ge r$ is denoted
by $\cHol(\cB(\ccH)_r^n)$. By \cite[Theorem 1.4]{Pop_holball}, $\cHol(\cB(\ccH)_r^n)$
is a subalgebra of $\fF_n$. For each $f\in \cHol(\cB(\ccH)_r^n)$, each Hilbert space $H$,
and each $T\in\cB(H)^n$ with $\| T\|<r$, the sum of the series~\eqref{Pop_series}
is denoted by $f(T)$. The map
\[
\gamma_T\colon \cHol(\cB(\ccH)_r^n) \to \cB(H),\qquad f\mapsto f(T),
\]
is an algebra homomorphism.

Fix an infinite-dimensional Hilbert space $\ccH$, and,
for each $\rho\in (0,r)$, define a seminorm $\|\cdot\|_\rho^P$ on $\cHol(\cB(\ccH)_r^n)$
by
\[
\| f\|_\rho^P=\sup\{ \| f(T)\| : T\in\cB(\ccH)^n,\; \| T\|\le\rho\}.
\]
By \cite[Theorem 5.6]{Pop_holball}, $\cHol(\cB(\ccH)_r^n)$ is a Fr\'echet space for
the topology determined by the family $\{ \|\cdot\|_\rho^P : \rho\in (0,r)\}$ of seminorms.

The following result is implicitly contained in \cite{Pop_holball}. For the reader's convenience,
we give a proof here.

\begin{prop}
\label{prop:Pop=my}
For each $r\in (0,+\infty]$, $\cHol(\cB(\ccH)_r^n)=\cF(\BB^n_r)$ as topological algebras.
\end{prop}
\begin{proof}
For each $f=\sum_\alpha c_\alpha\zeta_\alpha\in\fF_n$ and each $\rho>0$, let
\[
\| f\|_\rho^{(\infty)}=\sup_{d\in\Z_+}\Bigl(\sum_{|\alpha|=d}%
|c_\alpha|^2\Bigr)^{1/2} \rho^d \in [0,+\infty].
\]
We clearly have $\| f\|_\rho^{(\infty)} \le \| f\|_\rho^\bullet$. On the other hand, if $\tau>\rho$,
then
\[
\| f\|_\rho^\bullet
= \sum_{d=0}^\infty\Bigl(\sum_{|\alpha|=d} |c_\alpha|^2\Bigr)^{1/2} \rho^d
\le \sum_{d=0}^\infty \left(\frac{\rho}{\tau}\right)^d \| f\|_\tau^{(\infty)}
= \left(\frac{\tau}{\tau-\rho}\right) \| f\|_\tau^{(\infty)}.
\]
Hence
\[
\cF(\BB_r^n)=\biggl\{ f=\sum_{\alpha\in W_n} c_\alpha\zeta_\alpha :
\| f\|_\rho^{(\infty)}<\infty \;\forall \rho\in (0,r)\biggr\},
\]
and the families $\{ \|\cdot\|_\rho^\bullet : 0<\rho<r\}$ and
$\{ \|\cdot\|_\rho^{(\infty)} : 0<\rho<r\}$ of norms on $\cF(\BB_r^n)$ are equivalent.

By \cite[Corollary 1.2]{Pop_holball}, for each
$f=\sum_\alpha c_\alpha\zeta_\alpha\in\fF_n$ we have
\begin{equation*}
R(f)=\sup\biggl\{ \rho\ge 0 : \text{ the sequence }
\Bigl\{\Bigl(\sum_{|\alpha|=d} |c_\alpha|^2\Bigr)^{1/2}\rho^d\Bigr\}_{d\in\Z_+}
\text{ is bounded}\biggr\}.
\end{equation*}
Thus $\cHol(\cB(\ccH)_r^n)=\cF(\BB^n_r)$ as algebras. To complete the proof, it suffices to show
that for each $\rho\in (0,r)$ we have
\begin{equation}
\label{Pop_my_equiv}
\|\cdot\|_\rho^{(\infty)} \le \|\cdot\|_\rho^P \le \| \cdot\|_\rho^\bullet
\end{equation}
on $\cF(\BB_r^n)$.

Fix $\rho\in (0,r)$ and $T\in\cB(\ccH)^n$ such that $\| T\|\le\rho$.
For each $f=\sum_\alpha c_\alpha\zeta_\alpha\in\cF(\BB^n_r)$, we have
\[
\begin{split}
\| f(T)\|
&= \Bigl\| \sum_\alpha c_\alpha T_\alpha\Bigr\|
\le \sum_{d=0}^\infty \Bigl\| \sum_{|\alpha|=d} c_\alpha T_\alpha\Bigr\|
\le \sum_{d=0}^\infty\Bigl(\sum_{|\alpha|=d} |c_\alpha|^2\Bigr)^{1/2}
\Bigl\| \sum_{|\alpha|=d} T_\alpha T_\alpha^*\Bigr\|^{1/2}\\
&\le\sum_{d=0}^\infty \Bigl(\sum_{|\alpha|=d} |c_\alpha|^2\Bigr)^{1/2}
\Bigl\| \sum_{i=1}^n T_i T_i^*\Bigr\|^{d/2}
\le \sum_{d=0}^\infty \Bigl(\sum_{|\alpha|=d} |c_\alpha|^2\Bigr)^{1/2} \rho^d
=\| f\|_\rho^\bullet.
\end{split}
\]
Thus $\| f\|_\rho^P\le \| f\|_\rho^\bullet$. Let now $S_1,\ldots ,S_n$
be the left creation operators on the full Fock space
\begin{equation}
\label{Fock}
H=\mathop{\dot\bigoplus}\limits_{d\in\Z_+} (\CC^n_2)^{\dot\otimes d}
\end{equation}
(where $\dot\oplus$ stands for the Hilbert direct sum). Recall that $S_i x=e_i\otimes x$
for each $x\in (\CC^n_2)^{\dot\otimes d}$ and each $d\in\Z_+$, where $e_1,\ldots ,e_n$
is the standard basis of $\CC^n$. Let $e_0\in H$ be the ``vacuum vector'', i.e.,
any element of the $0$th direct summand $\CC$ in~\eqref{Fock} with $\| e_0\|=1$.
Then $\{ S_\alpha e_0 : \alpha\in W_n\}$ is an orthonormal basis of $H$.
Let now $T_i=\rho S_i\; (i=1,\ldots ,n)$, and let $T=(T_1,\ldots ,T_n)$.
We have $\| T\|=\rho\|\sum_{i=1}^n S_i S_i^*\|^{1/2}=\rho$, and
\[
\begin{split}
\| f(T) \|
&= \Bigl\| \sum_\alpha c_\alpha \rho^{|\alpha|} S_\alpha\Bigr\|
\ge  \Bigl\| \sum_\alpha c_\alpha \rho^{|\alpha|} S_\alpha e_0\Bigr\|
= \Bigl( \sum_\alpha |c_\alpha|^2 \rho^{2|\alpha|}\Bigr)^{1/2}\\
&=\Bigl(\sum_d\Bigl(\sum_{|\alpha|=d} |c_\alpha|^2\Bigr)\rho^{2d}\Bigr)^{1/2}
\ge\sup_d\Bigl(\sum_{|\alpha|=d} |c_\alpha|^2\Bigr)^{1/2}\rho^d
=\| f\|_\rho^{(\infty)}.
\end{split}
\]
This yields \eqref{Pop_my_equiv} and completes the proof.
\end{proof}

\begin{remark}
It is immediate from \eqref{Pop_my_equiv} that each seminorm
$\|\cdot\|_\rho^P$ on $\cHol(\cB(\ccH)_r^n)$ is actually a norm.
\end{remark}

\begin{remark}
Note that the proof of Proposition~\ref{prop:Pop=my} does not rely on the
completeness of $\cHol(\cB(\ccH)_r^n)$. Since $\cF(\BB^n_r)$ is obviously complete,
we see that Proposition~\ref{prop:Pop=my} readily implies \cite[Theorem 5.6]{Pop_holball}.
\end{remark}

For future use, it will be convenient to modify the norms $\|\cdot\|_\rho^\bullet$ on
$\cF(\BB^n_r)$ as follows.
Observe that we have a $\Z_+$-grading
$F_n=\bigoplus_{d\in\Z_+} (F_n)_d$, where
\[
(F_n)_d=\spn\{\zeta_\alpha : |\alpha|=d\}=(\CC^n_2)^{\dot\otimes d}.
\]
The norm $\|\cdot\|_\rho^\bullet$ on $F_n$ is then given by
$\| f\|_\rho^\bullet=\sum_{d\in\Z_+} \| f_d\| \rho^d$, where $f_d\in (F_n)_d$ is the $d$th
homogeneous component of $f$ and $\|\cdot\|$ is the hilbertian norm
on $(F_n)_d$.
Consider now a finer $\Z_+^n$-grading
$F_n=\bigoplus_{k\in\Z_+^n} (F_n)_k$, where
\[
(F_n)_k=\spn\{\zeta_\alpha : \alpha\in\rp^{-1}(k)\}.
\]
Observe that for each $d\in\Z_+$ we have $(F_n)_d=\bigoplus_{k\in (\Z_+^n)_d}(F_n)_k$.
For each $\rho>0$, define a new norm $\|\cdot\|_\rho^\circ$ on $F_n$ by
\[
\| f\|_\rho^\circ=\sum_{k\in\Z_+^n} \| f_k\| \rho^{|k|},
\]
where $f_k\in (F_n)_k$ is the $k$th
homogeneous component of $f$ and $\|\cdot\|$ is the hilbertian norm
on $(F_n)_k$ inherited from $(F_n)_{|k|}$. Explicitly, for $f=\sum_\alpha c_\alpha\zeta_\alpha\in F_n$
we have
\[
\| f\|_\rho^\circ=\sum_{k\in\Z_+^n}
\biggl(\sum_{\alpha\in \rp^{-1}(k)} |c_\alpha|^2\biggr)^{1/2} \rho^{|k|}.
\]

\begin{lemma}
\label{lemma:equiv_norms_free_ball}
For each $\rho>0$, the norm $\|\cdot\|_\rho^\circ$ is submultiplicative, and the families
\begin{equation}
\label{norms_free_ball}
\{ \|\cdot\|_\rho^\bullet : \rho\in (0,r)\}\quad\text{and}\quad\{ \|\cdot\|_\rho^\circ : \rho\in (0,r)\}
\end{equation}
of norms are equivalent on $F_n$.
\end{lemma}
\begin{proof}
Let $r,s\in\Z_+^n$, $a\in (F_n)_r\subset (\CC^n_2)^{\dot\otimes |r|}$,
and $b\in (F_n)_s\subset (\CC^n_2)^{\dot\otimes |s|}$. We clearly have
$\| ab\|=\| a\otimes b\|=\| a\| \| b\|$.
Hence for each $f,g\in F_n$ and each $\rho>0$ we obtain
\[
\begin{split}
\| fg\|_\rho^\circ
&=\sum_{k\in\Z_+^n} \| (fg)_k\|\rho^{|k|}
=\sum_{k\in\Z_+^n} \biggl\| \sum_{r+s=k} f_r g_s\biggr\| \rho^{|k|}\\
&\le \sum_{k\in\Z_+^n} \sum_{r+s=k} \| f_r\| \| g_s\| \rho^{|k|}
=\sum_{r,s\in\Z_+^n} \| f_r\| \| g_s\| \rho^{|r|} \rho^{|s|}
=\| f\|_\rho^\circ \| g\|_\rho^\circ.
\end{split}
\]
Thus $\|\cdot\|_\rho^\circ$ is submultiplicative.

Let now $f\in F_n$, and let $\rho\in (0,r)$.
By using the decomposition $(F_n)_d=\bigoplus_{k\in (\Z_+^n)_d} (F_n)_k$, we obtain
\begin{equation}
\label{bullet_prec_circ}
\begin{split}
\| f\|_\rho^\bullet
&=\sum_{d\in\Z_+} \| f_d\| \rho^d
=\sum_{d\in\Z_+} \biggl\| \sum_{k\in (\Z_+^n)_d} f_k\biggr\| \rho^d\\
&\le \sum_{d\in\Z_+} \sum_{k\in (\Z_+^n)_d} \| f_k\| \rho^d
=\sum_{k\in\Z_+^n} \| f_k\| \rho^{|k|}=\| f\|_\rho^\circ.
\end{split}
\end{equation}
Conversely, the orthogonality of the $f_k$'s and the Cauchy-Bunyakowsky-Schwarz inequality
yield the estimate
\begin{equation}
\label{circ_prec_bullet}
\| f\|_\rho^\circ
=\sum_{d\in\Z_+} \sum_{k\in (\Z_+^n)_d} \| f_k\| \rho^d
\le \sum_{d\in\Z_+} \biggl\| \sum_{k\in (\Z_+^n)_d} f_k\biggr\|
\left(\card (\Z_+^n)_d\right)^{1/2}\rho^d.
\end{equation}
We have
\begin{equation*}
\card (\Z_+^n)_d=\binom{d+n-1}{n-1}
\le (d+1)(d+2)\cdots (d+n-1)\le (d+n-1)^{n-1}.
\end{equation*}
Hence for each $\rho_1\in (\rho,r)$ we obtain
\[
C=\sup_{d\in\Z_+} \left(\card (\Z_+^n)_d\right)^{1/2}(\rho/\rho_1)^d <\infty.
\]
Together with \eqref{circ_prec_bullet}, this implies that
\begin{equation}
\label{circ_prec_bullet2}
\| f\|_\rho^\circ\le C \sum_{d\in\Z_+} \biggl\| \sum_{k\in (\Z_+^n)_d} f_k\biggr\| \rho_1^d
=C \sum_{d\in\Z_+} \| f_d\|\rho_1^d=C\| f\|_{\rho_1}^\bullet.
\end{equation}
Now \eqref{bullet_prec_circ} and \eqref{circ_prec_bullet2} imply that
the families \eqref{norms_free_ball} of norms are equivalent.
\end{proof}

\begin{prop}
We have
\begin{equation*}
\cF(\BB_r^n)=\biggl\{ a=\sum_{\alpha\in W_n} c_\alpha\zeta_\alpha :
\| a\|_\rho^\circ=\sum_{k\in\Z_+^n}\Bigl(\sum_{\alpha\in \rp^{-1}(k)}%
|c_\alpha|^2\Bigr)^{1/2} \rho^{|k|}<\infty
\;\forall \rho\in (0,r)\biggr\}.
\end{equation*}
Moreover, each norm $\|\cdot\|_\rho^\circ$ on $\cF(\BB^n_r)$ is submultiplicative.
\end{prop}
\begin{proof}
Immediate from Lemma~\ref{lemma:equiv_norms_free_ball}.
\end{proof}

Let us introduce some notation.
Given $\alpha=(\alpha_1,\ldots ,\alpha_d)\in W_n$, let
\[
\rrm(\alpha)=\bigl| \{ (i,j) : 1\le i<j\le d,\; \alpha_i>\alpha_j\}\bigr|.
\]

\begin{lemma}
\label{lemma:t(alpha)}
For each $\alpha\in W_n$, we have $\rt(\alpha)=q^{-\rrm(\alpha)}$.
In other words, $x_\alpha=q^{-\rrm(\alpha)} x^{\rp(\alpha)}$ in $\cO_q^\reg(\CC^n)$.
\end{lemma}
\begin{proof}
We use induction on $\rrm(\alpha)$. If $\rrm(\alpha)=0$, then $\alpha=\delta(\rp(\alpha))$,
whence $x_\alpha=x^{\rp(\alpha)}$ and $\rt(\alpha)=1$. Suppose now that $\rrm(\alpha)=r>0$,
and assume that $\rt(\beta)=q^{-\rrm(\beta)}$ for all $\beta\in W_n$ such that $\rrm(\beta)<r$.
Let $s=\min\{ i\ge 2: \alpha_i<\alpha_{i-1}\}$, and let $\beta=(s\, s-1)(\alpha)$.
It is elementary to check that $\rrm(\beta)=r-1$. By the induction hypothesis, we have
$x_\beta=q^{1-r}x^k$, where $k=\rp(\beta)=\rp(\alpha)$. Therefore
$x_\alpha=q^{-1}x_\beta=q^{-r}x^k$, i.e., $\rt(\alpha)=q^{-r}$, as required.
\end{proof}

\begin{lemma}
\label{lemma:x^k_norm}
For each $k\in\Z_+^n$, we have
\[
\| x^k\|_{\BB,1}=\biggl(\sum_{\alpha\in \rp^{-1}(k)} |q|^{-2\rrm(\alpha)}\biggr)^{-1/2}.
\]
\end{lemma}
\begin{proof}
Given $s\in\Z_+$, let
\[
\inv(k,s)=\card\{ \alpha\in \rp^{-1}(k) : \rrm(\alpha)=s\}.
\]
By \cite[Theorem 3.6]{Andrews}, we have
\[
\frac{\bigl[ |k|\bigr]_q!}{[k]_q!}
=\sum_{s\ge 0} \inv(k,s)q^s
=\sum_{\alpha\in \rp^{-1}(k)} q^{\rrm(\alpha)}.
\]
Together with \eqref{qballnorms_2}, this implies that
\[
\| x^k\|_{\BB,1}
=\left(\frac{[k]_{|q|^{-2}}!}{\bigl[ |k|\bigr]_{|q|^{-2}}!}\right)^{1/2}
=\biggl(\sum_{\alpha\in \rp^{-1}(k)} |q|^{-2\rrm(\alpha)}\biggr)^{-1/2}. \qedhere
\]
\end{proof}

We are now ready to prove the main result of this section.

\begin{theorem}
\label{thm:q_ball_quot_free_ball}
Let $q\in\CC^\times$, $n\in\N$, and $r\in (0,+\infty]$.
\begin{mycompactenum}
\item There exists a surjective continuous homomorphism
\begin{equation*}
\pi \colon \cF(\BB_r^n)\to\cO_q(\BB_r^n), \qquad
\zeta_i\mapsto x_i \quad (i=1,\ldots ,n).
\end{equation*}
\item $\Ker\pi$ coincides with
the closed two-sided ideal of $\cF(\BB_r^n)$
generated by the elements
$\zeta_i\zeta_j-q\zeta_j\zeta_i\; (i,j=1,\ldots ,n,\; i<j)$.
\item $\Ker\pi$ is a complemented subspace of $\cF(\BB_r^n)$.
\item Under the identification $\cO_q(\BB_r^n)\cong\cF(\BB^n_r)/\Ker\pi$,
the norm $\|\cdot\|_{\BB,\rho}$ on $\cO_q(\BB_r^n)$ is equal to the quotient of
the norm $\|\cdot\|_\rho^\circ$ on $\cF(\BB^n_r)\; (\rho\in (0,r))$.
\end{mycompactenum}
\end{theorem}
\begin{proof}
Let $f=\sum_{\alpha\in W_n} c_\alpha \zeta_\alpha\in\cF(\BB^n_r)$.
We claim that the family
\[
\biggl\{\sum_{\alpha\in \rp^{-1}(k)} c_\alpha x_\alpha : k\in\Z_+^n\biggr\}
\]
is absolutely summable in $\cO_q(\BB^n_r)$. Indeed, using Lemma~\ref{lemma:t(alpha)},
the Cauchy-Bunyakowsky-Schwarz inequality, and Lemma~\ref{lemma:x^k_norm},
for each $\rho\in (0,r)$ we obtain
\begin{align}
\sum_{k\in\Z_+^n}\biggl\| \sum_{\alpha\in \rp^{-1}(k)} c_\alpha x_\alpha\biggr\|_{\BB,\rho}
&=\sum_{k\in\Z_+^n} \biggl| \sum_{\alpha\in \rp^{-1}(k)} c_\alpha q^{-\rrm(\alpha)}\biggr|
\| x^k\|_{\BB,\rho}\notag\\
&\le \sum_{k\in\Z_+^n} \biggl(\sum_{\alpha\in \rp^{-1}(k)} |c_\alpha|^2\biggr)^{1/2}
\biggl(\sum_{\alpha\in \rp^{-1}(k)} |q^{-2\rrm(\alpha)}|\biggr)^{1/2}
\| x^k\|_{\BB,1}\rho^{|k|}\notag\\
\label{Brho<=rho}
&= \sum_{k\in\Z_+^n} \biggl(\sum_{\alpha\in \rp^{-1}(k)} |c_\alpha|^2\biggr)^{1/2} \rho^{|k|}
= \| f\|_\rho^\circ.
\end{align}
Hence there exists a continuous linear map
\[
\pi\colon\cF(\BB^n_r)\to\cO_q(\BB^n_r),\quad \sum_{\alpha\in W_n} c_\alpha\zeta_\alpha
\mapsto \sum_{k\in\Z_+^n} \biggl(\sum_{\alpha\in \rp^{-1}(k)} c_\alpha x_\alpha\biggr).
\]
Clearly, $\pi$ is an algebra homomorphism. Moreover, \eqref{Brho<=rho} implies that
\begin{equation}
\label{Brho<=rho2}
\| \pi(f)\|_{\BB,\rho} \le \| f\|_\rho^\circ \qquad (f\in\cF(\BB^n_r)).
\end{equation}

Let us now construct a continuous linear map
$\varkappa\colon\cO_q(\BB^n_r)\to\cF(\BB^n_r)$ such that $\pi\varkappa=\id$.
To this end, observe that for each $k\in\Z_+^n$ we have
\begin{equation}
\label{preimage_opt}
x^k=\pi\biggl(\sum_{\alpha\in \rp^{-1}(k)} c_\alpha q^{\rrm(\alpha)}\zeta_\alpha\biggr)
\end{equation}
as soon as
\begin{equation}
\label{sum=1}
\sum_{\alpha\in \rp^{-1}(k)} c_\alpha=1.
\end{equation}
We have
\begin{equation}
\label{minimize}
\biggl\| \sum_{\alpha\in \rp^{-1}(k)} c_\alpha q^{\rrm(\alpha)}\zeta_\alpha\biggr\|_\rho^\circ
=\biggl(\sum_{\alpha\in \rp^{-1}(k)} |c_\alpha|^2 |q|^{2\rrm(\alpha)}\biggr)^{1/2}\rho^{|k|}.
\end{equation}
Our strategy is to minimize \eqref{minimize} under the condition~\eqref{sum=1}.
First observe that replacing $c_\alpha$ by $\Re c_\alpha$ preserves~\eqref{sum=1}
and does not increase~\eqref{minimize}. Thus we may assume that $c_\alpha\in\R$.
An elementary computation involving Lagrange multipliers shows that the minimum
of~\eqref{minimize} under the condition~\eqref{sum=1} is attained at
\begin{equation*}
c_\alpha^0=|q|^{-2\rrm(\alpha)}\biggl(\sum_{\beta\in \rp^{-1}(k)} |q|^{-2\rrm(\beta)}\biggr)^{-1}
\end{equation*}
and is equal to
\begin{equation}
\label{preimage_opt_norm}
\biggl(\sum_{\beta\in \rp^{-1}(k)} |q|^{-2\rrm(\beta)}\biggr)^{-1/2} \rho^{|k|}=\| x^k\|_{\BB,\rho}
\end{equation}
(see Lemma~\ref{lemma:x^k_norm}). Let now
\[
a_k=\sum_{\alpha\in \rp^{-1}(k)} c^0_\alpha q^{\rrm(\alpha)} \zeta_\alpha.
\]
By \eqref{preimage_opt} and \eqref{preimage_opt_norm}, we have
\begin{equation}
\label{preimage_opt_2}
\pi(a_k)=x^k,\quad \| a_k\|_\rho^\circ=\| x^k\|_{\BB,\rho}.
\end{equation}
Let us now define
\[
\varkappa\colon\cO_q(\BB^n_r)\to\cF(\BB^n_r),\quad
\sum_{k\in\Z_+^n} c_k x^k \mapsto \sum_{k\in\Z_+^n} c_k a_k.
\]
By \eqref{preimage_opt_2}, we have
\[
\sum_{k\in\Z_+^n} |c_k| \| a_k\|_\rho^\circ
=\sum_{k\in\Z_+^n} |c_k| \| x^k\|_{\BB,\rho}
=\biggl\| \sum_{k\in\Z_+^n} c_k x^k \biggr\|_{\BB,\rho},
\]
whence $\varkappa$ is indeed a continuous linear map from $\cO_q(\BB^n_r)$ to
$\cF(\BB^n_r)$. Moreover,
\begin{equation}
\label{varkappa_ball_est}
\| \varkappa(f)\|_\rho^\circ \le \| f\|_{\BB,\rho} \qquad (f\in\cO_q(\BB^n_r)).
\end{equation}
By \eqref{preimage_opt_2}, we also have $\pi\varkappa=\id$. This proves (i) and (iii).

Let now $I\subset\cF(\BB^n_r)$ denote the closed two-sided ideal generated by
$\zeta_i\zeta_j-q\zeta_j\zeta_i\; (i<j)$. Clearly, $I\subset\Ker\pi$, and hence
$\pi$ induces a continuous homomorphism
\[
\bar\pi\colon\cF(\BB^n_r)/I\to\cO_q(\BB^n_r),\quad \bar\zeta_i\mapsto x_i \quad (i=1,\ldots ,n),
\]
where we let $\bar f=f+I\in\cF(\BB^n_r)/I$ for each $f\in\cF(\BB^n_r)$.
Let $\bar\varkappa\colon \cO_q(\BB^n_r)\to\cF(\BB^n_r)/I$ denote the composite of $\varkappa$
with the quotient map $\cF(\BB^n_r)\to\cF(\BB^n_r)/I$. It is immediate from $\pi\varkappa=\id$
that $\bar\pi\bar\varkappa=\id$. On the other hand, we obviously have
$\bar\varkappa(x_i)=\bar\zeta_i$ for each $i$, and it follows from the definition of $I$ that $\bar\varkappa$
is an algebra homomorphism. Since the elements $\bar\zeta_1,\ldots ,\bar\zeta_n$ generate a dense
subalgebra of $\cF(\BB^n_r)/I$, we conclude that
$\Im\bar\varkappa$ is dense in $\cF(\BB^n_r)/I$.
Together with $\bar\pi\bar\varkappa=\id$, this implies that $\bar\pi$ and $\bar\varkappa$
are topological isomorphisms. Therefore $I=\Ker\pi$, which proves (ii).

To prove (iv), let us identify $\cF(\BB^n_r)/I$ with $\cO_q(\BB^n_r)$ via $\bar\pi$,
and let $\|\cdot\|_\rho^\wedge$ denote the quotient norm of $\|\cdot\|_\rho^\circ\; (\rho\in (0,r))$.
By~\eqref{Brho<=rho2}, we have $\|\cdot\|_{\BB,\rho}\le\|\cdot\|_\rho^\wedge$,
while~\eqref{varkappa_ball_est} together with $\pi\varkappa=\id$ yields the opposite
estimate. This completes the proof.
\end{proof}

\section{Constructing the deformations}
\label{sect:deforms}

In this section we construct our main objects, i.e., Fr\'echet $\cO(\CC^\times)$-algebras
(together with the associated Fr\'echet algebra bundles over $\CC^\times$)
whose fibers over $q\in\CC^\times$ are isomorphic to
$\cO_q(\DD^n_r)$ and $\cO_q(\BB^n_r)$.
For basic facts on locally convex bundles and for related notation we refer to Appendix A.

If $K$ is a commutative Fr\'echet algebra, then by a {\em Fr\'echet $K$-algebra}
we mean a complete metrizable locally convex $K$-algebra $A$ such that
the action $K\times A\to A,\; (k,a)\mapsto ka$, is continuous.
Given a reduced Stein space $X$ and a Fr\'echet algebra $A$, we let $\cO(X,A)$
denote the Fr\'echet algebra of all holomorphic $A$-valued functions on $X$.
Clearly, $\cO(X,A)$ is a Fr\'echet $\cO(X)$-algebra in a canonical way.
By \cite[II.3.3]{Groth} (see also \cite[II.4.14]{X1}),
we have a topological isomorphism
\begin{equation}
\label{vec_hol}
\cO(X)\Ptens A\to\cO(X,A), \qquad f\otimes a\mapsto (x\mapsto f(x)a).
\end{equation}
From now on, we identify $\cO(X)$ and $A$ with subalgebras of $\cO(X,A)$
via the embeddings $f\mapsto f\otimes 1$ and $a\mapsto 1\otimes a$, respectively.

Let $z\in\cO(\CC^\times)$ denote the complex coordinate
(i.e., the inclusion map $\CC^\times\hookrightarrow\CC$).
Fix $r\in (0,+\infty]$, and let $I_\DD$, $I_\DD^\rT$, and $I_\BB$ denote the closed two-sided
ideals of $\cO(\CC^\times,\cF(\DD^n_r))$,  $\cO(\CC^\times,\cF^\rT(\DD^n_r))$,
and $\cO(\CC^\times,\cF(\BB^n_r))$, respectively, generated
by the elements $\zeta_i \zeta_j-z \zeta_j \zeta_i$\; $(i<j)$. Consider the
Fr\'echet $\cO(\CC^\times)$-algebras
\begin{align}
\label{def_DD_BB}
\cO_\defo(\DD^n_r)&=\cO(\CC^\times,\cF(\DD^n_r))/I_\DD,\notag\\
\cO_\defo^\rT(\DD^n_r)&=\cO(\CC^\times,\cF^\rT(\DD^n_r))/I_\DD^\rT,\\
\cO_\defo(\BB^n_r)&=\cO(\CC^\times,\cF(\BB^n_r))/I_\BB.\notag
\end{align}
We will use the following simplified notation for the respective Fr\'echet algebra bundles
(see Theorem~\ref{thm:Kalg-Bnd}):
\[
\sE(\DD^n_r)=\sE(\cO_\defo(\DD^n_r)),\quad
\sE^\rT(\DD^n_r)=\sE(\cO_\defo^\rT(\DD^n_r)),\quad
\sE(\BB^n_r)=\sE(\cO_\defo(\BB^n_r)).
\]
Our first goal is to show that the fibers of $\cO_\defo(\DD^n_r)$ and $\cO_\defo^\rT(\DD^n_r)$
over $q\in\CC^\times$ are isomorphic to $\cO_q(\DD^n_r)$, while the
fiber of $\cO_\defo(\BB^n_r)$ over $q\in\CC^\times$ is isomorphic to $\cO_q(\BB^n_r)$.

\begin{lemma}
\label{lemma:quot-quot}
Let $\fA$ be a Fr\'echet algebra, and let $I,J\subset\fA$ be closed two-sided ideals.
Denote by $q_I\colon\fA\to\fA/I$ and $q_J\colon\fA\to\fA/J$ the quotient maps,
and let $I_0=\ol{q_J(I)}$, $J_0=\ol{q_I(J)}$. Then there exist topological algebra
isomorphisms
\begin{equation}
\label{quot-quot}
\fA/(\ol{I+J})\cong(\fA/I)/J_0\cong(\fA/J)/I_0
\end{equation}
induced by the identity map on $\fA$.
\end{lemma}
\begin{proof}
Elementary.
\end{proof}

\begin{lemma}
\label{lemma:F_I_bnd}
Let $F$ be a Fr\'echet algebra, $X$ be a reduced Stein space, $I\subset\cO(X,F)$ be a closed
two-sided ideal, and $A=\cO(X,F)/I$. For each $x\in X$, let $\eps_x^F\colon\cO(X,F)\to F$
denote the evaluation map at $x$,
and let $I_x=\ol{\eps_x^F(I)}$. Then there exists a Fr\'echet algebra isomorphism
$A_x\cong F/I_x$ making the diagram
\begin{equation}
\label{quot-quot-diag}
\xymatrix@C-20pt@R-5pt{
& \cO(X,F) \ar[dl]_(.6){\quot} \ar[dr]^(.6){\eps_x^F} \\
A \ar[d]_{\quot} && F \ar[d]^{\quot} \\
A_x \ar[rr]^{\sim} && F/I_x
}
\end{equation}
commute (here $\quot$ are the respective quotient maps).
\end{lemma}
\begin{proof}
Given $x\in X$, let $\fm_x=\{ f\in\cO(X) : f(x)=0\}$. We claim that $\eps_x^F$ is onto, and that
$\Ker\eps_x^F=\ol{\fm_x\cO(X,F)}$. Indeed, the exact sequence
\[
0\to \fm_x\to\cO(X)\xra{\eps_x}\CC\to 0
\]
is obviously $\CC$-split (i.e., it splits in the category of Fr\'echet spaces).
Applying the functor $(-)\Ptens F$
and using the canonical isomorphism $\cO(X)\Ptens F\cong\cO(X,F)$, we obtain
the $\CC$-spilt sequence
\begin{equation*}
0\to \fm_x\Ptens F\to\cO(X,F)\xra{\eps_x^F} F\to 0.
\end{equation*}
This implies that $\Ker\eps_x^F\subset\ol{\fm_x\cO(X,F)}$.
The reverse inclusion is obvious. Thus $\Ker\eps_x^F=\ol{\fm_x\cO(X,F)}$, as claimed.

Let now $\fA=\cO(X,F)$ and $J=\ol{\fm_x\fA}$.
It follows from the above discussion that
we can identify $F$ with $\fA/J$. Under this identification, $\eps_x^F$ becomes the
quotient map $q_J\colon\fA\to\fA/J$. Observe also that
$\ol{\fm_x A}=\ol{q_I(J)}$, where $q_I\colon\fA\to \fA/I=A$ is the quotient map.
Now Lemma~\ref{lemma:quot-quot} implies that
\[
F/I_x\cong (\fA/J)/\ol{q_J(I)} \cong (\fA/I)/\ol{q_I(J)}=A/\ol{\fm_x A}=A_x.
\]
The commutativity of \eqref{quot-quot-diag} is clear from the construction.
\end{proof}

\begin{theorem}
For each $q\in\CC^\times$, we have topological algebra isomorphisms
\begin{gather}
\label{poly_bnd_iso}
\cO_\defo(\DD^n_r)_q\cong \sE(\DD^n_r)_q\cong\cO_q(\DD^n_r),\\
\label{poly^T_bnd_iso}
\cO_\defo^\rT(\DD^n_r)_q\cong \sE^\rT(\DD^n_r)_q\cong\cO_q(\DD^n_r),\\
\label{ball_bnd_iso}
\cO_\defo(\BB^n_r)_q\cong \sE(\BB^n_r)_q\cong\cO_q(\BB^n_r).
\end{gather}
\end{theorem}
\begin{proof}
Applying Lemma~\ref{lemma:F_I_bnd} with
$F=\cF(\DD^n_r)$ and $I=I_\DD$, we see that
for each $q\in\CC^\times$ there exists a topological algebra isomorphism
$\cO_\defo(\DD^n_r)_q\cong F/I_q$, where $I_q=\ol{\eps_q^F(I)}$.
Observe that $I_q$ is precisely the closed two-sided ideal of $F$ generated by the elements
$\zeta_i \zeta_j-q \zeta_j \zeta_i\; (i<j)$. Now~\eqref{poly_bnd_iso} follows from
Theorem~\ref{thm:q_poly_quot_free_poly}.
Similarly, letting $F=\cF^\rT(\DD^n_r)$ and $I=I_\DD^\rT$, and using
Theorem~\ref{thm:q_poly_quot_free^T_poly} instead of Theorem~\ref{thm:q_poly_quot_free_poly},
we obtain~\eqref{poly^T_bnd_iso}. Finally,
letting $F=\cF(\BB^n_r)$ and $I=I_\BB$, and using
Theorem~\ref{thm:q_ball_quot_free_ball} instead of Theorem~\ref{thm:q_poly_quot_free_poly},
we obtain~\eqref{ball_bnd_iso}.
\end{proof}

Our next goal is to show that the Fr\'echet
$\cO(\CC^\times)$-algebras $\cO_\defo(\DD^n_r)$ and $\cO_\defo^\rT(\DD^n_r)$
not only have isomorphic fibers (see~\eqref{poly_bnd_iso} and~\eqref{poly^T_bnd_iso}),
but are in fact isomorphic.
Towards this goal, we need some notation and several lemmas.
For brevity, we denote the elements $\zeta_i+I_\DD$ and $z+I_\DD$ of $\cO_\defo(\DD^n_r)$
by $x_i$ and $z$, respectively. The same convention applies to $\cO_\defo^\rT(\DD^n_r)$
and $\cO_\defo(\BB^n_r)$.
Recall from Propositions~\ref{prop:FD_vs_FTD} and~\ref{prop:FTD_vs_FB} that
we have continuous inclusions
\[
\cF(\DD^n_r)\subset\cF^\rT(\DD^n_r)\subset\cF(\BB^n_r).
\]
These inclusions induce Fr\'echet $\cO(\CC^\times)$-algebra homomorphisms
\begin{equation}
\label{D-DT-B}
\cO_\defo(\DD^n_r)\xra{i_{\DD\DD}} \cO_\defo^\rT(\DD^n_r)
\xra{i_{\DD\BB}}\cO_\defo(\BB^n_r)
\end{equation}
taking each $x_i$ to $x_i$ and $z$ to $z$. We aim to prove that the first map in \eqref{D-DT-B}
is a topological isomorphism.

Let $\cO^\reg_\defo(\CC^n)$ denote the algebra generated by $n+2$ elements
$x_1,\ldots ,x_n,z,z^{-1}$ subject to the relations
\begin{alignat*}{2}
x_i x_j &= zx_j x_i && \quad(1\le i<j\le n),\\
x_i z &= z x_i && \quad (1\le i\le n);\\
zz^{-1} &= z^{-1} z=1.
\end{alignat*}
Observe that $\cO^\reg_\defo(\CC^n)$ is exactly the algebra $R$ given by \eqref{algdef_free_quot}.
Clearly, there are algebra homomorphisms from $\cO^\reg_\defo(\CC^n)$ to each of the algebras
$\cO_\defo(\DD^n_r)$, $\cO_\defo^\rT(\DD^n_r)$, and $\cO_\defo(\BB^n_r)$,
taking each $x_i$ to $x_i$ and $z$ to $z$. Together with~\eqref{D-DT-B}, these
homomorphisms fit into the commutative diagram
\begin{equation}
\label{D-DT-B2}
\xymatrix{
\cO_\defo(\DD^n_r) \ar[r]^{i_{\DD\DD}} & \cO_\defo^\rT(\DD^n_r) \ar[r]^{i_{\DD\BB}}
& \cO_\defo(\BB^n_r)\\
& \cO_\defo^\reg(\CC^n) \ar[ul]^{i_\DD} \ar[u]^{i_\DD^\rT} \ar[ur]_{i_\BB}
}
\end{equation}

\begin{lemma}
\label{lemma:O_reg_def_dense}
\begin{compactenum}
\item[{\upshape (i)}]
The set $\{ x^k z^p : k\in\Z_+^n,\; p\in\Z\}$ is a vector space basis of
$\cO^\reg_\defo(\CC^n)$.
\item[{\upshape (ii)}]
The homomorphisms $i_\DD$, $i_\DD^\rT$, and $i_\BB$ are injective and have dense ranges.
\end{compactenum}
\end{lemma}
\begin{proof}
Obviously, $\cO^\reg_\defo(\CC^n)$ is spanned by $\{ x^k z^p : k\in\Z_+^n,\; p\in\Z\}$.
Let $a=\sum_{k,p} c_{kp} x^k z^p\in \cO^\reg_\defo(\CC^n)$,
and suppose that $i_\BB(a)=0$. Then for each $q\in\CC^\times$ we have
$i_\BB(a)_q=0$. Identifying $\cO_\defo(\BB^n_r)_q$ with $\cO_q(\BB^n_r)$
(see~\eqref{ball_bnd_iso}), we see that $\sum_k (\sum_p c_{kp} q^p) x^k=0$
in $\cO_q(\BB^n_r)$. Hence $\sum_p c_{kp} q^p=0$ for each $k$. Since this holds for
every $q\in\CC^\times$, we conclude that $c_{kp}=0$ for all $k$ and $p$.
This implies that the set $\{ x^k z^p : k\in\Z_+^n,\; p\in\Z\}$ is linearly independent
in $\cO^\reg_\defo(\CC^n)$, and that $\Ker i_\BB=0$.
By looking at \eqref{D-DT-B2},
we see that $\Ker i_\DD^\rT=\Ker i_\DD=0$. The rest is clear.
\end{proof}

\begin{lemma}
\label{lemma:alpha_opt}
For each $k\in\Z_+^n$ and each integer $m\in [0,\sum_{i<j} k_i k_j]$ there exists $\alpha\in W_n$
such that $\rp(\alpha)=k$, $\rs(\alpha)\le n+2$, and $x^k=x_\alpha z^m$ in $\cO^\reg_\defo(\CC^n)$.
\end{lemma}
\begin{proof}
Given $d\in\Z_+$, let $Z_d=\{\zeta_\alpha : \alpha\in W_{n,d}\}\subset F_n$. Clearly,
$\alpha\mapsto\zeta_\alpha$ is a $1$-$1$ correspondence between $W_{n,d}$ and $Z_d$.
Hence we have an action of $S_d$ on $Z_d$ given by $\sigma(\zeta_\alpha)=\zeta_{\sigma(\alpha)}$
($\alpha\in W_{n,d},\; \sigma\in S_d$).

We now present an explicit procedure of constructing $\alpha$ out of $k$.
Let $d=|k|$, $\beta=\delta(k)$, and $w_0=\zeta^k=\zeta_\beta$.
Let also $T=\{ (i\;\; i+1) : 1\le i\le d-1\}\subset S_d$.
By interchanging the first letter $\beta_1$ of $w_0$
with the subsequent letters, we obtain $d-1$ elements of $T$
such that the action of their composition on $w_0$ yields
\begin{equation}
\label{iter_1}
\zeta_{\beta_2}\cdots\zeta_{\beta_d}\zeta_{\beta_1}.
\end{equation}
Next, by interchanging the first letter $\beta_2$ of \eqref{iter_1}
with the subsequent letters (except for the last one), we obtain $d-2$ elements of $T$
such that the action of their composition on \eqref{iter_1} yields
\begin{equation*}
\zeta_{\beta_3}\cdots\zeta_{\beta_d}\zeta_{\beta_2}\zeta_{\beta_1}.
\end{equation*}
Continuing this procedure, after finitely many steps we obtain $\sigma_1,\ldots ,\sigma_r\in T$
such that
\[
(\sigma_r\cdots\sigma_1)(w_0)=\zeta_{\beta_d}\cdots\zeta_{\beta_1}
=\zeta_n^{k_n}\cdots\zeta_1^{k_1}.
\]
Let $w_i=(\sigma_i\cdots\sigma_1)(w_0)\; (i=1,\ldots ,r)$, and let
$\pi\colon F_n\to\cO^\reg_\defo(\CC^n)$ denote the homomorphism taking each $\zeta_i$ to $x_i$.
It follows from the above construction that for each $i=0,\ldots ,r-1$ we have either
$\pi(w_i)=\pi(w_{i+1})z$ or $\pi(w_i)=\pi(w_{i+1})$. We also have
\[
\pi(w_0)=x^k=x_n^{k_n}\cdots x_1^{k_1} z^{\sum_{i<j} k_i k_j}
=\pi(w_r) z^{\sum_{i<j} k_i k_j}.
\]
Hence there exists $t\in\{ 0,\ldots ,r\}$ such that $\pi(w_0)=\pi(w_t)z^m$.
In other words, if $\alpha\in W_{n,d}$ is such that $w_t=\zeta_\alpha$, then we have
$x^k=x_\alpha z^m$, as required. Since $\rp^{-1}(k)$ is invariant under the action of $S_d$,
we have $\rp(\alpha)=k$.

To complete the proof, we have to show that $\rs(\alpha)\le n+2$.
It follows from the construction that $w_t$ is of the form
\begin{gather}
\label{perm_form}
\zeta_r^{k_r-\ell}\zeta_{r+1}^{k_{r+1}}\cdots\zeta_{s-1}^{k_{s-1}}
\zeta_s^p\zeta_r\zeta_s^{k_s-p}\zeta_{s+1}^{k_{s+1}}\cdots\zeta_n^{k_n}
\zeta_r^{\ell-1}\zeta_{r-1}^{k_{r-1}}\cdots\zeta_1^{k_1}\\
(1\le r\le n-1,\quad 1\le\ell\le k_r,\quad r+1\le s\le n,\quad 0\le p\le k_s)\notag
\end{gather}
(here $r$ is the number of the letter that is currently moving from left to right,
$\ell-1$ is the number of instances of $\zeta_r$ that are already moved to their final
destination, $s$ shows the position where the moving letter is located,
and $p$ shows how many copies of $\zeta_s$ are already interchanged with the moving letter).
An elementary computation shows that for each word $\zeta_\alpha$ of the form~\eqref{perm_form}
we have $\rs(\alpha)\le n+2$.
\end{proof}

\begin{lemma}
\label{lemma:m<p}
For each $n\in\Z_+$ and each $\alpha\in W_n$,
we have $\rrm(\alpha)\le\sum_{i<j} \rp_i(\alpha)\rp_j(\alpha)$.
\end{lemma}
\begin{proof}
We use induction on $n$. For $n=0$, there is nothing to prove.
Suppose now that $n\ge 1$, and that the assertion holds for all words in $W_{n-1}$.
Given $r\in\Z_+$, let $\bar n_r=(n,\ldots ,n)\in W_{n,r}$ ($r$ copies of $n$).
If $\alpha=\bar n_{|\alpha|}$, then there is nothing to prove.
Assume that $\alpha\ne \bar n_{|\alpha|}$, and represent $\alpha$ as
\[
\alpha=\bar n_{r_1}\beta_1 \bar n_{r_2}\beta_2\ldots \bar n_{r_k}\beta_k \bar n_{r_{k+1}},
\]
where $r_1,\ldots ,r_{k+1}\ge 0$,
$\beta_i\in W_{n-1}$, $|\beta_i|>0$ for all $i=1,\ldots ,k$.
Let $\beta=\beta_1\ldots\beta_k\in W_{n-1}$.
We have
\begin{equation}
\label{m_alpha_beta}
\rrm(\alpha)=\rrm(\beta)+r_1(|\beta_1|+\cdots +|\beta_k|)
+r_2(|\beta_2|+\cdots +|\beta_k|)
+\cdots +r_k|\beta_k|.
\end{equation}
On the other hand, $\rp_i(\alpha)=\rp_i(\beta)$
for all $i\le n-1$, and $\rp_n(\alpha)=r_1+\cdots +r_{k+1}$.
Together with~\eqref{m_alpha_beta} and the induction hypothesis, this yields
\[
\begin{split}
\rrm(\alpha)
&\le\rrm(\beta)+(r_1+\cdots +r_k)|\beta|
\le\sum_{i<j<n} \rp_i(\beta)\rp_j(\beta)+\rp_n(\alpha)|\beta|\\
&=\sum_{i<j<n} \rp_i(\alpha)\rp_j(\alpha)+\rp_n(\alpha)\sum_{i=1}^{n-1} \rp_i(\alpha)
=\sum_{i<j} \rp_i(\alpha)\rp_j(\alpha). \qedhere
\end{split}
\]
\end{proof}

Given $f\in\cO(\CC^\times)$ and $t\ge 1$, let
\[
\| f\|_t=\sum_{n\in\Z} |c_n(f)| t^{|n|},
\]
where $c_n(f)$ is the $n$th Laurent coefficient of $f$ at $0$. We will use the simple fact
that the family $\{\|\cdot\|_t : t\ge 1\}$ of norms generates the standard (i.e., compact-open)
topology on $\cO(\CC^\times)$.

\begin{lemma}
\label{lemma:x_alpha_sprad}
Given $\rho\in (0,r)$ and $\tau,t\ge 1$, let $\|\cdot\|_{\rho,\tau,t}$ denote the
quotient seminorm on $\cO_\defo(\DD^n_r)$ associated to the projective tensor norm
$\|\cdot\|_t\otimes_\pi \|\cdot\|_{\rho,\tau}$ on $\cO(\CC^\times)\Ptens\cF(\DD^n_r)$.
We have
\begin{equation}
\label{x_alpha_sprad}
\lim_{d\to\infty}\Bigl(\sup_{\alpha\in W_{n,d}} \| x_\alpha\|_{\rho,\tau,t}\Bigr)^{1/d}\le\rho.
\end{equation}
\end{lemma}
\begin{proof}
Let $\alpha\in W_{n,d}$. Repeating the argument of Lemma~\ref{lemma:t(alpha)},
we see that $x_\alpha=x^{\rp(\alpha)} z^{-\rrm(\alpha)}$ in $\cO_\defo^\reg(\CC^n)$
(and hence in $\cO_\defo(\DD^n_r)$). By Lemma~\ref{lemma:m<p}, the $n$-tuple
$k=\rp(\alpha)$ and the integer $m=\rrm(\alpha)$ satisfy the conditions of Lemma~\ref{lemma:alpha_opt}.
Hence there exists $\beta\in W_{n,d}$ such that $x^{\rp(\alpha)}=x_\beta z^{\rrm(\alpha)}$
and $\rs(\beta)\le n+2$. Thus we have $x_\alpha=x_\beta=1\otimes\zeta_\beta+I_\DD$, whence
\begin{equation}
\label{x_alpha_est}
\| x_\alpha\|_{\rho,\tau,t}\le
\|\zeta_\beta\|_{\rho,\tau}=\rho^d\tau^{s(\beta)+1}\le\rho^d\tau^{n+3}.
\end{equation}
Raising \eqref{x_alpha_est} to the power $1/d$, taking the supremum over $\alpha\in W_{n,d}$,
and letting $d\to\infty$, we obtain~\eqref{x_alpha_sprad}, as required.
\end{proof}

\begin{theorem}
\label{thm:D_DT_iso}
For each $n\in\N$ and each $r\in (0,+\infty]$,
$i_{\DD\DD}\colon\cO_\defo(\DD^n_r)\to\cO_\defo^\rT(\DD^n_r)$
is a topological isomorphism.
\end{theorem}
\begin{proof}
Let $\varphi_1\colon\cO(\CC^\times)\to\cO_\defo(\DD^n_r)$ denote the homomorphism given
by $\varphi_1(f)=f\otimes 1+I_\DD$.
By Lemma~\ref{lemma:x_alpha_sprad}, the $n$-tuple $(x_1,\ldots ,x_n)\in\cO_\defo(\DD^n_r)^n$
is strictly spectrally $r$-contractive. Applying Proposition~\ref{prop:univ_F_poly_T},
we obtain a continuous homomorphism
\[
\varphi_2\colon\cF^\rT(\DD^n_r)\to\cO_\defo(\DD^n_r),\qquad\zeta_i\mapsto x_i
\quad (i=1,\ldots ,n).
\]
Since the images of $\varphi_1$ and $\varphi_2$ commute, the map
\[
\varphi\colon\cO(\CC^\times)\Ptens\cF^\rT(\DD^n_r)\to\cO_\defo(\DD^n_r), \quad
\varphi(f\otimes a)=\varphi_1(f)\varphi_2(a),
\]
is an algebra homomorphism. By construction, $\varphi(\zeta_i\zeta_j-z\zeta_j\zeta_i)=0$
for all $i<j$. Hence $\varphi$ vanishes on $I_\DD^\rT$ and induces a homomorphism
\[
\psi\colon\cO_\defo^\rT(\DD^n_r)\to\cO_\defo(\DD^n_r).
\]
Clearly, $\psi$ takes each $x_i$ to $x_i$ and $z$ to $z$.
Since the canonical images of $\cO^\reg_\defo(\CC^n)$ are dense in $\cO_\defo(\DD^n_r)$
and in $\cO_\defo^\rT(\DD^n_r)$, we conclude that $\psi i_{\DD\DD}$ and
$i_{\DD\DD}\psi$ are the identity maps on $\cO_\defo(\DD^n_r)$
and $\cO_\defo^\rT(\DD^n_r)$, respectively. This completes the proof.
\end{proof}

\begin{corollary}
The Fr\'echet algebra bundles $\sE(\DD^n_r)$ and $\sE^\rT(\DD^n_r)$ are isomorphic.
\end{corollary}

From now on we identify $\cO_\defo^\rT(\DD^n_r)$ with $\cO_\defo(\DD^n_r)$
and $\sE^\rT(\DD^n_r)$ with $\sE(\DD^n_r)$ via $i_{\DD\DD}$.

\section{The nonprojectivity of $\cO_\defo(\DD^n_r)$}
\label{sect:nonproj}

In this section we show that $\cO_\defo(\DD^n_r)$ is not topologically projective
(and, as a consequence, is not topologically free) over $\cO(\CC^\times)$.
This is a new phenomenon as compared to the formal deformation theory
and to the free holomorphic deformations considered in \cite{Pfl_Schott}
(see Section~\ref{sect:intro} for more details).
To prove the nonprojectivity of $\cO_\defo(\DD^n_r)$,
we first give a power series representation of $\cO_\defo(\DD^n_r)$,
which may be of independent interest.

Define $\omega\colon\Z_+^n\times\Z\to\R$ by
\[
\omega(k,p)=
\begin{cases}
p & \text{if $p\ge 0$};\\
0 & \text{if $p<0$ and $p+\sum\limits_{i<j} k_i k_j\ge 0$};\\
p+\sum\limits_{i<j} k_i k_j & \text{if $p+\sum\limits_{i<j} k_i k_j<0$}.
\end{cases}
\]
Observe that
\begin{equation}
\label{omega_mod}
|\omega(k,p)|=\min\Bigl\{ |\lambda| : \lambda\in
\Bigl[p,p+\mathop{\textstyle\sum}\limits_{i<j} k_i k_j\Bigr]\Bigr\}.
\end{equation}
\begin{lemma}
\label{lemma:omega_ineq}
For all $k,\ell\in\Z_+^n,\; p,q\in\Z$ we have
\[
|\omega(k+\ell,p+q-\mathop{\textstyle\sum}\limits_{i>j} k_i \ell_j)|\le |\omega(k,p)|+|\omega(\ell,q)|.
\]
\end{lemma}
\begin{proof}
By \eqref{omega_mod}, we have
\begin{align}
|\omega(k,p)|+|\omega(\ell,q)|
&=\min\Bigl\{ |\lambda|+|\mu| : \lambda\in \Bigl[p,p+\mathop{\textstyle\sum}\limits_{i<j} k_i k_j\Bigr],\;
\mu\in \Bigl[q,q+\mathop{\textstyle\sum}\limits_{i<j} \ell_i \ell_j\Bigr]\Bigr\}\notag\\
&\ge\min\Bigl\{ |\lambda+\mu| : \lambda\in \Bigl[p,p+\mathop{\textstyle\sum}\limits_{i<j} k_i k_j\Bigr],\;
\mu\in \Bigl[q,q+\mathop{\textstyle\sum}\limits_{i<j} \ell_i \ell_j\Bigr]\Bigr\}\notag\\
\label{min1}
&=\min\Bigl\{ |\lambda| : \lambda\in \Bigl[p+q,p+q+\mathop{\textstyle\sum}\limits_{i<j} k_i k_j
+\mathop{\textstyle\sum}\limits_{i<j} \ell_i \ell_j\Bigr]\Bigr\}.
\end{align}
On the other hand,
\begin{align}
|&\omega(k+\ell,p+q-\mathop{\textstyle\sum}\limits_{i>j} k_i \ell_j)|\notag\\
&=\min\Bigl\{ |\lambda| : \lambda\in \Bigl[p+q-\mathop{\textstyle\sum}\limits_{i>j} k_i \ell_j,
p+q-\mathop{\textstyle\sum}\limits_{i>j} k_i \ell_j
+\mathop{\textstyle\sum}\limits_{i<j}(k_i+\ell_i)(k_j+\ell_j)\Bigr]\Bigr\}\notag\\
\label{min2}
&=\min\Bigl\{ |\lambda| : \lambda\in \Bigl[p+q-\mathop{\textstyle\sum}\limits_{i>j} k_i \ell_j,
p+q+\mathop{\textstyle\sum}\limits_{i<j} k_i k_j
+\mathop{\textstyle\sum}\limits_{i<j} \ell_i \ell_j
+\mathop{\textstyle\sum}\limits_{i<j} k_i \ell_j\Bigr]\Bigr\}.
\end{align}
To complete the proof, it remains to observe that the interval over which the minimum is taken
in \eqref{min1} is contained in the respective interval in \eqref{min2}.
\end{proof}

Given $n\in\N$ and $r\in (0,+\infty]$, let
\[
D_{n,r}=\biggl\{ a=\sum_{k\in\Z_+^n,\; p\in\Z} c_{kp} x^k z^p :
\| a\|_{\rho,\tau}=\sum_{k,p} |c_{kp}|\rho^{|k|}\tau^{|\omega(k,p)|}<\infty
\;\forall\rho\in (0,r),\;\forall\tau\ge 1\biggr\}.
\]
Clearly, $D_{n,r}$ is a Fr\'echet space for the topology determined by
the family $\{\|\cdot\|_{\rho,\tau} : \rho\in (0,r),\;\tau\ge 1\}$ of norms.
By Lemma~\ref{lemma:O_reg_def_dense} (i), we may identify $\cO_\defo^\reg(\CC^n)$
with a dense vector subspace of $D_{n,r}$.

\begin{prop}
The multiplication on $\cO_\defo^\reg(\CC^n)$ uniquely extends to
a continuous multiplication on $D_{n,r}$. Moreover, each norm
$\|\cdot\|_{\rho,\tau}\; (\rho\in (0,r),\;\tau\ge 1)$ is submultiplicative on $D_{n,r}$.
Thus $D_{n,r}$ is an Arens-Michael algebra.
\end{prop}
\begin{proof}
Since $D_{n,r}$ is the completion of $\cO_\defo^\reg(\CC^n)$, it suffices to show that
every norm $\|\cdot\|_{\rho,\tau}$ ($\rho\in (0,r),\;\tau\ge 1$) is submultiplicative on
$\cO_\defo^\reg(\CC^n)$.
Let $a=x^k z^p$ and $b=x^\ell z^q\in\cO_\defo^\reg(\CC^n)$.
By Lemma~\ref{lemma:omega_ineq}, we have
\[
\begin{split}
\| ab\|_{\rho,\tau}
=\| x^{k+\ell} z^{p+q-\sum_{i>j} k_i\ell_j}\|_{\rho,\tau}
&=\rho^{|k+\ell|}\tau^{|\omega(k+\ell,p+q-\sum_{i>j} k_i\ell_j)|}\\
&\le\rho^{|k|}\rho^{|\ell|}\tau^{|\omega(k,p)|}\tau^{|\omega(\ell,q)|}
=\| a\|_{\rho,\tau} \| b\|_{\rho,\tau}.
\end{split}
\]
Since the norm of every element of $\cO_\defo^\reg(\CC^n)$ is the sum of the norms
of the respective monomials, we conclude that the inequality
$\| ab\|_{\rho,\tau}\le\| a\|_{\rho,\tau} \| b\|_{\rho,\tau}$ holds for all
$a,b\in \cO_\defo^\reg(\CC^n)$.
\end{proof}

\begin{theorem}
\label{thm:O_def_D_powrep}
There exists a topological algebra isomorphism $\cO_\defo(\DD^n_r)\to D_{n,r}$
uniquely determined by $z\mapsto z$, $x_i\mapsto x_i\; (i=1,\ldots ,n)$.
\end{theorem}
\begin{proof}
Since $D_{n,r}$ is an Arens-Michael algebra, and since $z\in D_{n,r}$ is invertible,
there exists a unique continuous homomorphism $\varphi_1\colon\cO(\CC^\times)\to D_{n,r}$
such that $\varphi_1(w)=z$ (where $w$ is the coordinate on $\CC$).
Observe also that for each $i=1,\ldots ,n$ we have a continuous homomorphism from
$\cO(\DD_r)$ to $D_{n,r}$ uniquely determined by $w\mapsto x_i$ (cf. \eqref{poly_power_rep}).
By the universal property of $\cF(\DD^n_r)$ (see \eqref{F_poly}),
there exists a unique continuous homomorphism $\varphi_2\colon\cF(\DD^n_r)\to D_{n,r}$
such that $\varphi_2(\zeta_i)=x_i$ ($i=1,\ldots ,n$).
Since the images of $\varphi_1$ and $\varphi_2$ commute, the map
\[
\tilde\varphi\colon\cO(\CC^\times)\Ptens\cF(\DD^n_r)\to D_{n,r},\quad
f\otimes a\mapsto\varphi_1(f)\varphi_2(a),
\]
is an algebra homomorphism. By construction, $\tilde\varphi(I_\DD)=0$. Hence $\tilde\varphi$
induces a continuous homomorphism
\[
\varphi\colon\cO_\defo(\DD^n_r)\to D_{n,r},\quad z\mapsto z,
\quad x_i\mapsto x_i\quad (i=1,\ldots ,n).
\]
We claim that $\varphi$ is a topological isomorphism. To see this, observe first that
for each $k\in\Z_+^n$ and each $p\in\Z$ we have $\omega(k,p)-p\in [0,\sum_{i<j} k_i k_j]$.
By Lemma~\ref{lemma:alpha_opt}, there exists $\alpha(k,p)\in W_n$ such that
\begin{equation}
\label{alpha_k_p}
\rp(\alpha(k,p))=k,\quad \rs(\alpha(k,p))\le n+2,\quad x^k=x_{\alpha(k,p)}z^{\omega(k,p)-p}.
\end{equation}
We now define
\[
\tilde\psi\colon D_{n,r}\to\cO(\CC^\times)\Ptens\cF(\DD^n_r),\quad
\sum_{k,p} c_{kp} x^k z^p\mapsto
\sum_{k,p} c_{kp} z^{\omega(k,p)}\otimes\zeta_{\alpha(k,p)}.
\]
To see that $\tilde\psi$ is a continuous linear map from $D_{n,r}$ to
$\cO(\CC^\times)\Ptens\cF(\DD^n_r)$, take $t,\tau\ge 1$, $\rho\in (0,r)$, and let
$\|\cdot\|_{t;\rho,\tau}$ denote the projective tensor norm $\|\cdot\|_t\otimes_\pi\|\cdot\|_{\rho,\tau}$
on $\cO(\CC^\times)\Ptens\cF(\DD^n_r)$ (cf. Lemma~\ref{lemma:x_alpha_sprad}).
Using the first two formulas in~\eqref{alpha_k_p}, we see that
\[
\sum_{k,p} |c_{kp}| \| z^{\omega(k,p)}\otimes\zeta_{\alpha(k,p)}\|_{t;\rho,\tau}
\le\sum_{k,p} |c_{kp}| t^{|\omega(k,p)|}\rho^{|k|}\tau^{n+3}
=\tau^{n+3} \Bigl\|\sum_{k,p} c_{kp} x^k z^p\Bigr\|_{\rho,t}.
\]
This implies that $\tilde\psi$ indeed takes $D_{n,r}$ to $\cO(\CC^\times)\Ptens\cF(\DD^n_r)$
and is continuous. Let now $\psi\colon D_{n,r}\to\cO_{\defo}(\DD^n_r)$ be the composition
of $\tilde\psi$ and the quotient map of $\cO(\CC^\times)\Ptens\cF(\DD^n_r)$ onto
$\cO_{\defo}(\DD^n_r)$. By the third equality in~\eqref{alpha_k_p},
we have
\[
\psi(x^k z^p)=x_{\alpha(k,p)}z^{\omega(k,p)}=x^k z^p
\qquad (k\in\Z_+^n,\; p\in\Z).
\]
Since $\varphi(x^k z^p)=x^k z^p$ as well, and since the monomials $x^k z^p$ span
dense vector subspaces of $D_{n,r}$ and $\cO_\defo(\DD^n_r)$, we conclude that
$\varphi\psi$ and $\psi\varphi$ are the identity maps. This completes the proof.
\end{proof}

Recall some definitions and facts from \cite{X1} (see also \cite{X2}).
Let $A$ be a Fr\'echet algebra.
By a {\em left Fr\'echet $A$-module} we mean a left $A$-module $X$ together with
a Fr\'echet space topology such that the action
$A\times X\to X$ is continuous. Morphisms of Fr\'echet $A$-modules are assumed
to be continuous. A morphism
$\sigma\colon X\to Y$ of left Fr\'echet $A$-modules is an {\em admissible epimorphism}
if there exists a continuous linear map $\varkappa\colon Y\to X$ such that $\sigma\varkappa=\id_Y$.
A left Fr\'echet $A$-module $P$ is {\em topologically projective} if for each admissible epimorphism
$X\to Y$ of left Fr\'echet $A$-modules the induced map
$\Hom_A(P,X)\to\Hom_A(P,Y)$ is onto. If $E$ is a Fr\'echet space, then the projective tensor
product $A\Ptens E$ is a left Fr\'echet $A$-module in a natural way.
A left Fr\'echet $A$-module $F$ is {\em topologically free} if $F\cong A\Ptens E$ for some Fr\'echet space $E$.
Since $\Hom_A(A\Ptens E,-)\cong\Hom_{\CC}(E,-)$, it follows that each topologically free Fr\'echet
$A$-module is topologically projective. Given a left Fr\'echet $A$-module $P$, let $\mu_P\colon A\Ptens P\to P$
denote the $A$-module morphism uniquely determined by $a\otimes x\mapsto ax$.
By \cite[Chap.~III, Theorem~1.27]{X1}, $P$ is topologically projective if and only if
there exists an $A$-module morphism
$\nu\colon P\to A\Ptens P$ such that $\mu_P\nu=\id_P$.

\begin{theorem}
\label{thm:nonproj}
For each $n\ge 2$ and each $r\in (0,+\infty]$,
the Fr\'echet $\cO(\CC^\times)$-module $\cO_\defo(\DD^n_r)$ is not topologically projective
(and hence is not topologically free).
\end{theorem}
\begin{proof}
We identify $\cO_\defo(\DD^n_r)$ with $D_{n,r}$ via the isomorphism constructed in
Theorem~\ref{thm:O_def_D_powrep}. Assume, towards a contradiction, that $\cO_\defo(\DD^n_r)$
is topologically projective over $\cO(\CC^\times)$. Then there exists a Fr\'echet $\cO(\CC^\times)$-module
morphism $\nu\colon\cO_\defo(\DD^n_r)\to\cO(\CC^\times)\Ptens\cO_\defo(\DD^n_r)$
such that $\mu\nu=\id$ (where $\mu=\mu_{\cO_\defo(\DD^n_r)}$).
Recall the standard fact that the projective tensor product of two K\"othe sequence
spaces is again a K\"othe sequence space (cf. \cite[41.7]{Kothe_II}).
Hence we have a topological isomorphism
\[
\begin{split}
&\cO(\CC^\times)\Ptens\cO_\defo(\DD^n_r)\\
&\cong
\biggl\{ v=\sum_{\substack{k\in\Z_+^n\\ r,p\in\Z}} c_{rkp} z^r \otimes x^k z^p :
\| v\|_{t,\rho,\tau}=\sum_{r,k,p} |c_{rkp}| t^{|r|} \rho^{|k|} \tau^{|\omega(k,p)|}<\infty
\;\forall \rho\in (0,r),\;\forall t,\tau\ge 1\biggr\}.
\end{split}
\]
For each $m\in\N$, let
\[
\nu(x_1^{2m} x_2^{m})=\sum_{r,k,p} c_{rkp}^{(m)} z^r \otimes x^k z^p.
\]
Since $\mu\nu=\id$, we see that $c_{rkp}^{(m)}=0$ unless $p+r=0$. Hence
\begin{equation}
\label{nu_pow}
\nu(x_1^{2m} x_2^{m})=\sum_{k,p} c_{kp}^{(m)} z^{-p} \otimes x^k z^p,
\end{equation}
where $c_{kp}^{(m)}=c_{-pkp}^{(m)}$. This implies that
\[
x_1^{2m} x_2^{m}=\mu\Bigl(\sum_{k,p} c_{kp}^{(m)} z^{-p} \otimes x^k z^p\Bigr)
=\sum_{k,p} c_{kp}^{(m)} x^k.
\]
Letting $\bar m=(2m,m,0,\ldots ,0)\in\Z_+^n$, we conclude that
\begin{equation}
\label{sum_c_kp}
\sum_p c_{kp}^{(m)}=
\begin{cases}
1 & \text{if $k=\bar m$};\\
0 & \text{if $k\ne\bar m$.}
\end{cases}
\end{equation}
Fix $\rho\in (0,r)$ and choose $\rho_1\in (0,r)$, $\tau_1\ge 1$, and $C>0$ such that
\begin{equation}
\label{coretr_estim}
\| \nu(u)\|_{2,\rho,1}\le C\| u\|_{\rho_1,\tau_1}
\qquad (u\in\cO_\defo(\DD^n_r)).
\end{equation}
Letting $u_{sm}=z^{-s} x_1^{2m} x_2^{m}$, where $0\le s\le 2m^2$,
we obtain from \eqref{coretr_estim}
\begin{equation}
\label{nu_le}
\| \nu(u_{sm})\|_{2,\rho,1}\le C\| u_{sm}\|_{\rho_1,\tau_1}=C\rho_1^{3m},
\end{equation}
because $\omega(\bar m,-s)=0$. On the other hand, \eqref{nu_pow} implies that
\begin{align}
\| \nu(u_{sm})\|_{2,\rho,1}
=\| z^{-s}\nu(x_1^{2m} x_2^{m})\|_{2,\rho,1}
&=\Bigl\|\sum_{k,p} c_{kp}^{(m)} z^{-s-p}\otimes x^k z^p\Bigr\|_{2,\rho,1}\notag\\
\label{nu_ge}
&=\sum_{k,p} |c_{kp}^{(m)}| 2^{|s+p|}\rho^{|k|}
\ge \sum_p |c_{\bar m p}^{(m)}| 2^{|s+p|}\rho^{3m}.
\end{align}
Combining \eqref{nu_le} and \eqref{nu_ge}, we see that
\begin{equation}
\label{cmp_est}
\sum_p |c_{\bar m p}^{(m)}| 2^{|s+p|}\rho^{3m} \le C\rho_1^{3m}
\qquad (m\in\N,\; 0\le s\le 2m^2).
\end{equation}
Since $\sum_{p\in\Z} |c_{\bar mp}^{(m)}|\ge 1$ by \eqref{sum_c_kp}, we have
\begin{align}
\label{ge-m^2}
\text{either }&\sum_{p\ge -m^2} |c_{\bar mp}^{(m)}|\ge 1/2\\
\label{le-m^2}
\text{or } &\sum_{p< -m^2} |c_{\bar mp}^{(m)}|\ge 1/2.
\end{align}
If \eqref{ge-m^2} holds, then, letting $s=2m^2$ in \eqref{cmp_est}, we obtain
\[
C\rho_1^{3m} \ge
\sum_{p\ge -m^2} |c_{\bar m p}^{(m)}| 2^{|2m^2+p|}\rho^{3m}
\ge \frac{1}{2}\, 2^{m^2} \rho^{3m}.
\]
On the other hand, if \eqref{le-m^2} holds, then, letting $s=0$ in \eqref{cmp_est}, we see that
\[
C\rho_1^{3m} \ge
\sum_{p< -m^2} |c_{\bar m p}^{(m)}| 2^{|p|}\rho^{3m}
\ge \frac{1}{2}\, 2^{m^2} \rho^{3m}.
\]
Thus for each $m\in\N$ we have $2^{m^2-1}\le C(\rho_1/\rho)^{3m}$, which is impossible.
The resulting contradiction completes the proof.
\end{proof}

\begin{remark}
In \cite[Example 3.2]{Pfl_Schott}, the authors construct the algebra $\cO_\defo(\CC^n)$ and claim
(essentially without proof) that it is topologically free over $\cO(\CC^\times)$.
Theorem~\ref{thm:nonproj} shows that this is not the case.
\end{remark}

\begin{remark}
We conjecture that a result similar to Theorem~\ref{thm:nonproj} holds for $\cO_\defo(\BB^n_r)$
as well.
\end{remark}

\section{The continuity of $\sE(\DD^n_r)$ and $\sE(\BB^n_r)$}
\label{sect:cont_bnd}

In operator algebra theory and in related fields of mathematics one usually works
with {\em continuous} Banach bundles, i.e., with those bundles of Banach spaces
whose norm is a continuous function on the total space
(see, e.g., \cite{Dupre,Fell_Mackey,FD}).
A similar notion makes sense in the more general context
of locally convex bundles (cf. Definition~\ref{def:cont_bnd}).
Thus a natural question is whether or not the bundles $\sE(\DD^n_r)$ and $\sE(\BB^n_r)$
constructed in Section~\ref{sect:deforms} are continuous.
The answer would obviously be yes if the algebras $\cO_\defo(\DD^n_r)$
and $\cO_\defo(\BB^n_r)$ were topologically free over $\cO(\CC^\times)$.
However, we know from Theorem~\ref{thm:nonproj} that this is not the case,
at least for $\cO_\defo(\DD^n_r)$.
In this section our goal is to show that the Fr\'echet algebra
bundles $\sE(\DD^n_r)$ and $\sE(\BB^n_r)$
are nevertheless continuous. This will be deduced from the following general result.

\begin{theorem}
\label{thm:cont_bnd}
Let $X$ be a reduced Stein space,  $F$ be a Fr\'echet algebra, $I\subset\cO(X,F)$ be a closed
two-sided ideal, and $A=\cO(X,F)/I$. For every $x\in X$, we identify $A_x$ with $F/I_x$
by Lemma~{\upshape\ref{lemma:F_I_bnd}}. Suppose that there exist a dense subalgebra
$A_0\subset A$ and a directed defining family $\cN_F=\{\|\cdot\|_\lambda : \lambda\in\Lambda\}$
of seminorms on $F$ such that for each $a\in A_0$ and each $\lambda\in\Lambda$
the function $X\to\R,\; x\mapsto \| a_x\|_{\lambda,x}$, is continuous
(where $\|\cdot\|_{\lambda,x}$ is the quotient seminorm of $\|\cdot\|_\lambda$ on $F/I_x$).
Then the bundle $\sE(A)$ is continuous.
\end{theorem}

To prove Theorem~\ref{thm:cont_bnd}, we need two lemmas.

\begin{lemma}
\label{lemma:quot-quot-semi}
Under the conditions of Lemma~{\upshape\ref{lemma:quot-quot}}, suppose that
$\|\cdot\|$ is a continuous seminorm on $\fA$. Let $\|\cdot\|_I$ be the quotient seminorm
of $\|\cdot\|$ on $\fA/I$, and let $\|\cdot\|_{I,J}$ be the quotient seminorm of $\|\cdot\|_I$
on $(\fA/I)/J_0$. Then the isomorphisms~\eqref{quot-quot} are isometric with respect to
$\|\cdot\|_{\ol{I+J}}$, $\|\cdot\|_{I,J}$, and $\|\cdot\|_{J,I}$, respectively.
\end{lemma}
\begin{proof}
Elementary.
\end{proof}

\begin{lemma}
Let $K\subset X$ be a holomorphically convex compact set. For each $f\in\cO(X)$,
we let $\| f\|_K=\sup_{x\in K} |f(x)|$. Given $\lambda\in\Lambda$,
let $\|\cdot\|_{K,\lambda}^\pi$ denote the projective tensor seminorm
$\|\cdot\|_K\otimes_\pi \|\cdot\|_\lambda$ on $\cO(X,F)$, and let $\|\cdot\|_{K,\lambda}$
denote the quotient seminorm of $\|\cdot\|_{K,\lambda}^\pi$ on $A$.
Finally, given $x\in X$, let
$\|\cdot\|_{K,\lambda,x}$ denote the quotient seminorm of $\|\cdot\|_{K,\lambda}$
on $A_x$. Then
\begin{equation}
\label{semi_on_A_x}
\|\cdot\|_{K,\lambda,x}=
\begin{cases}
\|\cdot\|_{\lambda,x} & \text{\emph{if} } x\in K;\\
0 & \text{\emph{if} } x\notin K.
\end{cases}
\end{equation}
\end{lemma}
\begin{proof}
As in Lemma~\ref{lemma:F_I_bnd}, we identify $F$ with $\cO(X,F)/\Ker\eps_x^F$.
Let $\|\cdot\|_{K,\lambda}^{(x)}$ denote the quotient seminorm of $\|\cdot\|_{K,\lambda}^\pi$ on $F$.
By applying Lemma~\ref{lemma:quot-quot-semi} to $\fA=\cO(X,F)$ and $J=\ol{\fm_x\fA}$
(see~\eqref{quot-quot-diag}), we conclude that $\|\cdot\|_{K,\lambda,x}$ equals the
quotient seminorm of $\|\cdot\|_{K,\lambda}^{(x)}$ on $F/I_x=A_x$.
To complete the proof, it remains to compare the seminorms $\|\cdot\|_{K,\lambda}^{(x)}$
and $\|\cdot\|_\lambda$ on $F$, i.e., to show that
\begin{equation}
\label{semi_on_F}
\|\cdot\|_{K,\lambda}^{(x)}=
\begin{cases}
\|\cdot\|_{\lambda} & \text{if } x\in K,\\
0 & \text{if } x\notin K.
\end{cases}
\end{equation}
Observe that $\eps_x^F=\eps_x\otimes\id_F$, where $\eps_x=\eps_x^{\CC}\colon\cO(X)\to\CC$
is the evaluation map. If $x\in K$, then for every $f\in\cO(X)$ we clearly have
$|\eps_x(f)|\le\| f\|_K$. This implies that $\|\eps_x^F(u)\|_\lambda\le \| u\|_{K,\lambda}^\pi$
for each $u\in\cO(X,F)$. Hence $\|\cdot\|_\lambda\le\|\cdot\|_{K,\lambda}^{(x)}$.
On the other hand, for each $v\in F$ we have $v=\eps_x^F(1\otimes v)$, and
$\| 1\otimes v\|_{K,\lambda}^\pi=\| v\|_\lambda$. Therefore
$\|\cdot\|_\lambda=\|\cdot\|_{K,\lambda}^{(x)}$ whenever $x\in K$.

Now assume that $x\notin K$. Since $K$ is holomorphically convex, there exists
$f\in\cO(X)$ such that $f(x)=1$ and $\| f\|_K<1$. For each $v\in F$ and each $n\in\N$ we have
$v=f^n(x)v=\eps_x^F(f^n\otimes v)$, whence
\[
\| v\|_{K,\lambda}^{(x)}\le\| f^n\otimes v\|_{K,\lambda}^\pi\le\| f\|_K^n \| v\|_\lambda\to 0
\quad (n\to\infty).
\]
Thus $\| \cdot\|_{K,\lambda}^{(x)}=0$ whenever $x\notin K$. This implies~\eqref{semi_on_F}
and completes the proof.
\end{proof}

\begin{proof}[Proof of Theorem~{\upshape\ref{thm:cont_bnd}}]
By construction of $\sE(A)$ and by Remark~\ref{rem:mod_bnd_equiv_fam},
the locally convex uniform vector structure on $\sE(A)$
is given by the family
\[
\cN_A=\{\|\cdot\|_{K,\lambda} : K\in\HCC(X),\; \lambda\in\Lambda\},
\]
where $\HCC(X)$ denotes the collection of all holomorphically convex compact subsets of $X$,
and $\|\cdot\|_{K,\lambda}$ is the seminorm on $\sE(A)$ whose restriction to each fiber $A_x$
is $\|\cdot\|_{K,\lambda,x}$.
Let $\Gamma=\{\tilde a : a\in A_0\}\subset\Gamma(X,E)$. Since $A_0$ is dense in $A$,
it follows that the set $\{ s_x : s\in\Gamma\}$ is dense in $A_x$ for each $x\in X$.
Unfortunately, we cannot directly apply Proposition~\ref{prop:cont_semi} to $\Gamma$
and $\|\cdot\|_{K,\lambda}$ because of~\eqref{semi_on_A_x}, which implies that
the function $x\mapsto\| s_x\|_{K,\lambda}=\| s_x\|_{K,\lambda,x}$
is continuous on $K$, but, in general,
not on the whole of $X$. Thus we have to modify $\cN_A$ as follows.
Given $K\in\HCC(X)$, choose a continuous, compactly supported function
$h_K\colon X\to [0,1]$ such that $h_K(x)=1$ for all $x\in K$, and let $K'$ denote the
holomorphically convex hull of $\supp h_K$. Define a new seminorm $\|\cdot\|'_{K,\lambda}$
on $\sE(A)$ by
\[
\| u\|'_{K,\lambda}=h_K(p(u))\| u\|_{K',\lambda}\qquad (u\in\sE(A)),
\]
and let $\cN'_A=\{\|\cdot\|'_{K,\lambda} : K\in\HCC(X),\; \lambda\in\Lambda\}$.
Clearly, $\|\cdot\|'_{K,\lambda}$ is upper semicontinuous (being the product of two
nonnegative, upper semicontinuous functions).
Taking into account~\eqref{semi_on_A_x}, we see that
\[
\|\cdot\|_{K,\lambda} \le \|\cdot\|'_{K,\lambda} \le \|\cdot\|_{K',\lambda}.
\]
This implies that $\cN'_A\sim\cN_A$, whence $\cU_{\cN_A}=\cU_{\cN'_A}$.
Moreover, $\cN'_A$ is admissible by Lemma~\ref{lemma:eqv_fam_1}~(ii).
Let now $a\in A_0$ and $x\in X$. By~\eqref{semi_on_A_x} and by the choice of $h_K$, we see that
\[
\|\tilde a_x\|'_{K,\lambda}
=\| a_x\|'_{K,\lambda}
=h_K(x)\| a_x\|_{K',\lambda}
=h_K(x)\| a_x\|_{\lambda,x},
\]
which is a continuous function on $X$ by assumption.
Now the result follows from Proposition~\ref{prop:cont_semi} applied to $\Gamma$
and $\cN'_A$.
\end{proof}

\begin{corollary}
The bundles $\sE(\DD^n_r)$ and $\sE(\BB^n_r)$ are continuous.
\end{corollary}
\begin{proof}
For convenience, let us denote the norm $\|\cdot\|_{\DD,\rho}$
on $\cO_q(\DD^n_r)$ by $\|\cdot\|_{\DD,q,\rho}$. Similarly, we write $\|\cdot\|_{\BB,q,\rho}$
for the norm $\|\cdot\|_{\BB,\rho}$ on $\cO_q(\BB^n_r)$.
Let $F=\cF^\rT(\DD^n_r)$, and let $\{\|\cdot\|_\rho : \rho\in (0,r)\}$ be the
standard defining family of seminorms on $F$, where $\|\cdot\|_\rho$ is given by~\eqref{F_poly_T}.
Let also $I=I_{\DD}^\rT$ and
$A=\cO(\CC^\times,F)/I=\cO_\defo^\rT(\DD^n_r)\cong\cO_\defo(\DD^n_r)$
(see Theorem~\ref{thm:D_DT_iso}). As in Lemma~\ref{lemma:F_I_bnd}, given
$q\in\CC^\times$, let $I_q=\ol{\eps_q^F(I)}\subset F$.
By Theorem~\ref{thm:q_poly_quot_free^T_poly}, we can identify $F/I_q$ with $\cO_q(\DD^n_r)$.
Moreover, the quotient seminorm $\|\cdot\|_{\rho,q}$ of $\|\cdot\|_\rho$ on $F/I_q$ becomes
$\|\cdot\|_{\DD,q,\rho}$ under this identification.

Let now $A_0=\cO^\reg_\defo(\CC^n)$. By Lemma~\ref{lemma:O_reg_def_dense},
$A_0$ is a dense subalgebra of $A$. Given $a=\sum_{k,p} c_{kp} x^k z^p\in A_0$,
we have
\[
\| a_q\|_{\rho,q}
=\| a_q\|_{\DD,q,\rho}
=\sum_k \biggl| \sum_p c_{kp} q^p\biggr| w_q(k) \rho^{|k|}.
\]
This implies that the function $q\mapsto\| a_q\|_{\rho,q}$ is continuous on $\CC^\times$.
By applying Theorem~\ref{thm:cont_bnd}, we conclude that $\sE(\DD^n_r)$ is continuous.

A similar argument applies to $\sE(\BB^n_r)$. Specifically, we have to replace
$\cF^\rT(\DD^n_r)$ by $\cF(\BB^n_r)$, $\|\cdot\|_\rho$ by $\|\cdot\|_\rho^\circ$,
$I_{\DD}^\rT$ by $I_{\BB}$, and to apply Theorem~\ref{thm:q_ball_quot_free_ball}
instead of Theorem~\ref{thm:q_poly_quot_free^T_poly}.
The continuity of $q\mapsto\| a_q\|^\circ_{\rho,q}$ now follows from
\[
\| a_q\|^\circ_{\rho,q}
=\| a_q\|_{\BB,q,\rho}
=\sum_k \biggl| \sum_p c_{kp} q^p\biggr|
\left(\frac{[k]_{|q|^2}!}{\bigl[ |k|\bigr]_{|q|^2}!}\right)^{1/2}\!\!\!\!
u_q(k) \rho^{|k|}. \qedhere
\]
\end{proof}

\section{Relations to Rieffel's quantization}
\label{sect:Rief_quant}

The strict ($C^*$-algebraic) version of deformation quantization was introduced by
Rieffel \cite{Rf_Heis,Rf_dq_oa,Rf_Lie,Rf_mem,Rf_qst,Rf_quest}
(see Section~\ref{sect:intro} for more details).
In this section we show that the bundles $\sE(\DD^n_r)$ and
$\sE(\BB^n_r)$ fit into Rieffel's framework adapted to the Fr\'echet algebra setting.

Towards this goal, it will be convenient to modify the bundles $\sE(\DD^n_r)$ and
$\sE(\BB^n_r)$ by replacing the deformation parameter $q\in\CC^\times$
with $h\in\CC$, where $q=\exp(ih)$. Specifically, let $h\in\cO(\CC)$ denote the complex
coordinate (i.e., the identity map of $\CC$), and let
$\hI_\DD$, $\hI_\DD^\rT$, and $\hI_\BB$ denote the closed two-sided
ideals of $\cO(\CC,\cF(\DD^n_r))$,  $\cO(\CC,\cF^\rT(\DD^n_r))$,
and $\cO(\CC,\cF(\BB^n_r))$, respectively, generated
by the elements $\zeta_j \zeta_k-e^{ih} \zeta_k \zeta_j\; (j<k)$.
By analogy with~\eqref{def_DD_BB}, consider the
Fr\'echet $\cO(\CC)$-algebras
\begin{align*}
\hcO_\defo(\DD^n_r)&=\cO(\CC,\cF(\DD^n_r))/\hI_\DD,\\
\hcO_\defo^\rT(\DD^n_r)&=\cO(\CC,\cF^\rT(\DD^n_r))/\hI_\DD^\rT,\\
\hcO_\defo(\BB^n_r)&=\cO(\CC,\cF(\BB^n_r))/\hI_\BB.
\end{align*}
We will use the following simplified notation for the respective Fr\'echet algebra bundles:
\[
\hsE(\DD^n_r)=\sE(\hcO_\defo(\DD^n_r)),\quad
\hsE^\rT(\DD^n_r)=\sE(\hcO_\defo^\rT(\DD^n_r)),\quad
\hsE(\BB^n_r)=\sE(\hcO_\defo(\BB^n_r)).
\]
Exactly as in Sections~\ref{sect:deforms} and \ref{sect:cont_bnd}, we see that
the fibers of $\hsE(\DD^n_r)$ and $\hsE(\BB^n_r)$ over $h\in\CC$
are isomorphic to $\cO_{\exp(ih)}(\DD^n_r)$ and $\cO_{\exp(ih)}(\BB^n_r)$, respectively,
that $\hsE(\DD^n_r)$ and $\hsE^\rT(\DD^n_r)$ are isomorphic, and that
$\hsE(\DD^n_r)$ and $\hsE(\BB^n_r)$ are continuous.

\begin{remark}
Alternatively, we can define the bundles $\hsE(\DD^n_r)$ and $\hsE(\BB^n_r)$
to be the pullbacks of $\sE(\DD^n_r)$ and $\sE(\BB^n_r)$ under the
exponential map $e\colon\CC\to\CC^\times,\; h\mapsto\exp(ih)$.
Specifically, suppose that $X$ and $Y$ are topological spaces, $f\colon X\to Y$ is
a continuous map, and $(E,p,\cU)$ is a locally convex bundle over $Y$.
The pullback $f^* E=E\times_Y X$ is a prebundle of topological vector spaces
over $X$ in a canonical way; the projection $\tilde p\colon f^* E\to X$
is given by $\tilde p(v,x)=x$ (cf. \cite[2.5]{Husemoller}). Define $\tilde f\colon f^* E\to E$
by $\tilde f(v,x)=v$, and let $\tilde f^{-1}(\cU)$ denote the locally convex
uniform vector structure on $f^* E$ with base $\{\tilde f^{-1}(U) : U\in\cU\}$.
It is easy to show that $(f^* E,\tilde p,\tilde f^{-1}(\cU))$ is a locally
convex bundle. Moreover, if $\cN=\{\|\cdot\|_\lambda : \lambda\in\Lambda\}$
is an admissible family of seminorms for $E$, then the collection
$\tilde\cN=\{\|\cdot\|_\lambda^f : \lambda\in\Lambda\}$
is an admissible family of seminorms for $f^* E$, where $\|\cdot\|_\lambda^f$
is given by $\| (v,x)\|_\lambda^f=\| v\|_\lambda$. Clearly, this implies that
if $E$ is continuous, then so is $f^* E$. Finally, it can be shown that
$\hsE(\DD^n_r)\cong e^*\sE(\DD^n_r)$ and $\hsE(\BB^n_r)\cong e^*\sE(\BB^n_r)$.
We will not use these results below, so we omit the details.
\end{remark}

Recall (see, e.g., \cite{Laur_Geng}) that a {\em Poisson algebra} is a commutative algebra $\cA$
together with a bilinear operation $\{\cdot , \cdot\}\colon \cA\times \cA\to \cA$
making $\cA$ into a Lie algebra and such that for each $a\in\cA$
the map $\{ a,\cdot\}$ is a derivation of the associative algebra $\cA$.

The following definition is a straightforward modification of Rieffel's quantization
to the Fr\'echet algebra case.

\begin{definition}
Let $\cA$ be a Poisson algebra, and let $X$ be an open connected subset of $\CC$
containing $0$. A {\em strict Fr\'echet deformation quantization} of $\cA$ is
the following data:
\begin{mycompactenum}
\item[$\mathrm{(DQ1)}$]
A continuous Fr\'echet algebra bundle $(A,p,\cU)$ over $X$;
\item[$\mathrm{(DQ2)}$]
A family of dense subalgebras $\{\cA_h\subset A_h : h\in X\}$;
\item[$\mathrm{(DQ3)}$]
A family of vector space isomorphisms
\[
i_h\colon\cA\to\cA_h,\quad a\mapsto a_h\qquad (h\in X)
\]
such that $i_0$ is an algebra isomorphism.
\end{mycompactenum}
Moreover, we require that for each $a,b\in\cA$
\begin{equation}
\label{comm_Pois}
\frac{a_h b_h-b_h a_h}{h}-i\{ a,b\}_h\to 0_0\quad (h\to 0).
\end{equation}
\end{definition}

\begin{remark}
\label{rem:comm_Pois}
If $\{\|\cdot\|_\lambda: \lambda\in\Lambda\}$ is an admissible family of seminorms on $A$,
then~\eqref{comm_Pois} is equivalent to
\[
\left\|\frac{a_h b_h-b_h a_h}{h}-i\{ a,b\}_h\right\|_\lambda\to 0\quad (h\to 0),
\]
for all $\lambda\in\Lambda$. This is immediate from the fact that
the family
\[
\bigl\{ \sT(V,0,\lambda,\eps) : V\subset X\text{ is an open neighborhood of }0,\;
\lambda\in\Lambda,\; \eps>0\bigr\}
\]
is a local base at $0_0$ (see Lemma~\ref{lemma:bas_unif} (ii)).
\end{remark}

\begin{theorem}
\label{thm:str_def}
Consider the Poisson algebra $\cA=\CC[x_1,\ldots ,x_n]$ with bracket
\[
\{ f,g\}=\sum_{i<j} x_i x_j \left(\frac{\dd f}{\dd x_i}\frac{\dd g}{\dd x_j}
-\frac{\dd f}{\dd x_j}\frac{\dd g}{\dd x_i}\right)\qquad (f,g\in\cA).
\]
Let $(A,p,\cU)$ be a continuous Fr\'echet algebra bundle over an open connected
subset $X\subset\CC$ containing $0$. Assume that for each $h\in X$ the fiber $A_h$
contains $\cA_h=\cO_{\exp(ih)}^\reg(\CC^n)$ as a dense subalgebra, and that
for each $j=1,\ldots ,n$ the constant section $x_j$ of $A$ is continuous
(for example, we can let $A=\hsE(\DD^n_r)$ or $A=\hsE(\BB^n_r)$).
Let $i_h\colon\cA\to\cA_h$
be the vector space isomorphism given by $x^k\mapsto x^k\; (k\in\Z_+^n)$.
Then $\bigl(A,\;\{\cA_h,\; i_h\}_{h\in X}\bigr)$ is a strict Fr\'echet deformation quantization of $\cA$.
\end{theorem}
\begin{proof}
The only thing that needs to be proved is the compatibility relation~\eqref{comm_Pois}.
Let
\begin{equation}
\label{sigma}
\sigma\colon\Z_+^n\times\Z_+^n\to\Z,\quad \sigma(k,\ell)=\sum_{i<j} k_i \ell_j.
\end{equation}
An easy computation shows that for each
$f=\sum_k a_k x^k$ and $g=\sum_k b_k x^k\in\cA$ we have
\[
f_h g_h=\sum_m\left(\sum_{k+\ell=m} a_k b_\ell e^{-ih\sigma(\ell,k)}\right) (x^m)_h
\]
and
\[
\{ f,g\}=
\sum_m\left(\sum_{k+\ell=m} a_k b_\ell \bigl(\sigma(k,\ell)-\sigma(\ell,k)\bigr)\right) x^m.
\]
Hence
\[
\frac{f_h g_h-g_h f_h}{h}-i\{ f,g\}_h
=\sum_m\left(\sum_{k+\ell=m} a_k b_\ell \varphi_{k\ell}(h)\right) (x^m)_h,
\]
where
\[
\varphi_{k\ell}(h)=\frac{e^{-ih\sigma(\ell,k)}-e^{-ih\sigma(k,\ell)}}{h}
-i\bigl(\sigma(k,\ell)-\sigma(\ell,k)\bigr).
\]
For each $k,\ell\in\Z_+^n$ we clearly have $\varphi_{k\ell}(h)\to 0$ as $h\to 0$.

Suppose now that $\{\|\cdot\|_\lambda : \lambda\in\Lambda\}$ is an admissible family
of continuous seminorms on $A$. For each $m\in\Z_+^n$ and each $\lambda\in\Lambda$, the function
$h\mapsto\| (x^m)_h\|_\lambda$ is continuous and hence is locally bounded.
Therefore
\[
\left\| \frac{f_h g_h-g_h f_h}{h}-i\{ f,g\}_h\right\|_\lambda
\le \sum_m\sum_{k+\ell=m} |a_k b_\ell \varphi_{k\ell}(h)| \| (x^m)_h\|_\lambda\to 0
\quad (h\to 0).
\]
By Remark \ref{rem:comm_Pois}, this implies \eqref{comm_Pois} and completes the proof.
\end{proof}

\section{Relations to formal deformations}
\label{sect:form_def}

In this section we establish a relationship between our ``holomorphic deformations''
and the more traditional concept of a formal deformation.
Specifically, we aim to show that the algebras obtained from
$\cO_\defo(\DD^n_r)$ and $\cO_\defo(\BB^n_r)$ via the extension of scalars
functor $\CC[[h]]\ptens{\cO(\CC^\times)}(-)$ are formal deformations
of $\cO(\DD^n_r)$ and $\cO(\BB^n_r)$, respectively.
The nontrivial point here is to prove that the resulting $\CC[[h]]$-algebras
are topologically free over $\CC[[h]]$.
The difficulty comes from the fact that $\cO_\defo(\DD^n_r)$ (and presumably $\cO_\defo(\BB^n_r)$)
are not topologically free over $\cO(\CC^\times)$ (see Section~\ref{sect:nonproj}).

In what follows, given a Fr\'echet algebra $K$ and a Fr\'echet space $E$, we identify
$E$ with a subspace of $K\Ptens E$ via the map $x\mapsto 1\Tens x$.
The proof of the following lemma is elementary and is therefore omitted.

\begin{lemma}
\label{lemma:K-bilin}
Let $K$ be a commutative Fr\'echet algebra, let $E$ be a Fr\'echet space, and let $M$ be a Fr\'echet
$K$-module.
Then each continuous bilinear map $E\times E\to M$ uniquely extends to
a continuous $K$-bilinear map $(K\Ptens E)\times (K\Ptens E)\to M$.
\end{lemma}

Let $A$ be a Fr\'echet algebra. Following \cite{BFGP}, we say that a Fr\'echet $\CC[[h]]$-algebra
$\tilde A$ is a {\em formal Fr\'echet deformation} of $A$ if
(1) $\tilde A$ is topologically free as a Fr\'echet $\CC[[h]]$-module,
and (2) $\tilde A/h\tilde A\cong A$ as Fr\'echet algebras.

Given a Fr\'echet space $E$, let $\CC[[h;E]]$ denote the vector space of all formal power
series with coefficients in $E$. The isomorphism $\CC[[h;E]]\cong E^{\Z_+}$,
$\sum_j x_j h^j \mapsto (x_j)$, makes $\CC[[h;E]]$ into a Fr\'echet space.
Since the completed projective tensor product commutes with products
(see \cite[\S41.6]{Kothe_II} or \cite[II.5.19]{X1}),
we have a Fr\'echet space isomorphism
\[
\CC[[h]]\Ptens E\cong\CC[[h;E]],\quad \Bigl(\sum_j c_j h^j\Bigr)\otimes x\mapsto \sum_j c_j x h^j.
\]
Thus a formal Fr\'echet deformation of $A$ amounts to a continuous
$\CC[[h]]$-bilinear multiplication $\star$ on the topologically free Fr\'echet $\CC[[h]]$-module
$\tilde A=\CC[[h]]\Ptens A\cong\CC[[h;A]]$ such that we have $a\star b=ab \mod h\tilde A$
for all $a,b\in A$.

Our plan in this section will be as follows. First we construct a
formal deformation $\cO_\fdef(U)$ of $\cO(U)$, where $U$ is an open subset of $\CC^n$,
and next we show that $\cO_\fdef(U)\cong\CC[[h]]\ptens{\cO(\CC^\times)}\cO_\defo(U)$
provided that either $U=\DD^n_r$ or $U=\BB^n_r$.

\begin{prop}
\label{prop:fdef}
Let $n\in\N$, and let $\mathcal U$ denote the collection of all open subsets of $\CC^n$.
There exists a unique family $\{\star_U : U\in\mathcal U\}$
of continuous $\CC[[h]]$-bilinear multiplications on $\CC[[h]]\Ptens\cO(U)$
such that
\begin{mycompactenum}
\item
for all $j,k=1,\ldots ,n$ we have
\begin{equation}
\label{relat_fdef}
x_j \star_U x_k=
\begin{cases}
x_j x_k & \text{if } j\le k,\\
e^{-ih} x_j x_k & \text{if } j>k,
\end{cases}
\end{equation}
where $x_1,\ldots ,x_n$ are the coordinates on $\CC^n$
(in particular, $x_j\star_U x_k=e^{ih} x_k\star_U x_j$ for $j<k$);
\item
the correspondence $U\mapsto\cO_\fdef(U)=(\CC[[h]]\Ptens\cO(U),\star_U)$
is a Fr\'echet algebra sheaf on $\CC^n$.
\end{mycompactenum}
Moreover, for each open set $U\subset\CC^n$, the $\CC[[h]]$-algebra
$\cO_\fdef(U)$ is a formal Fr\'echet
deformation of $\cO(U)$.
\end{prop}
\begin{proof}
Given an open set $U\subset\CC^n$, consider the derivations $D_j=x_j (\dd/\dd x_j)$
of $\cO(U)$ ($j=1,\ldots ,n$), and let
\[
P=\sum_{j>k} D_j\Tens D_k\colon \cO(U)\Ptens\cO(U) \to \cO(U)\Ptens\cO(U).
\]
Letting $\mu\colon\cO(U)\Ptens\cO(U)\to\cO(U)$ denote the multiplication map,
we define a continuous bilinear map
\begin{gather*}
\cO(U)\times\cO(U) \to \CC[[h,\cO(U)]], \quad (f,g)\mapsto f\star_U g,\\
f\star_U g=(\mu\circ e^{-ihP})(f\Tens g)
=\sum_{m=0}^\infty \frac{(-ih)^m}{m!} (\mu\circ P^m)(f\Tens g).
\end{gather*}
By Lemma~\ref{lemma:K-bilin}, this map uniquely extends to a continuous
$\CC[[h]]$-bilinear map $\tilde A\times\tilde A\xra{\star_U}\tilde A$, where
$\tilde A=\CC[[h;\cO(U)]]$. Since the derivations $D_j$ commute, it follows from
\cite[Theorem 2.1]{GZ} (cf. also \cite[Theorem 8]{Gerst_def_3}) that $\star_U$ is associative.
Thus $\cO_\fdef(U)=(\CC[[h]]\Ptens\cO(U),\star_U)$ is a Fr\'echet $\CC[[h]]$-algebra.
We clearly have $f\star_U g=fg \mod h\tilde A$ for all $f,g\in\cO(U)$, so
$\cO_\fdef(U)$ is a formal Fr\'echet deformation of $\cO(U)$.
Relations~\eqref{relat_fdef} are readily verified.
Identifying $\CC[[h]]\Ptens\cO(U)$ with the space $\cO(U,\CC[[h]])$ of
$\CC[[h]]$-valued holomorphic functions on $U$ via \eqref{vec_hol},
we see that the correspondence $U\mapsto \cO_\fdef(U)$ is a sheaf of Fr\'echet
spaces on $\CC^n$. Since the multiplications $\star_U$ are obviously compatible with the
restriction maps, we see that $\cO_\fdef$ is a Fr\'echet algebra sheaf.

If $U$ is a polydisk, then the polynomials in $x_1,\ldots ,x_n$ are dense in $\cO(U)$,
so the multiplication $\star_U$ is uniquely determined by~\eqref{relat_fdef}.
For an arbitrary open set $U$, cover $U$ by a family $(U_i)$ of polydisks, and observe
that the natural map $\cO_\fdef(U) \to \prod_i \cO_\fdef(U_i)$ is an injective algebra
homomorphism. Thus $\star_U$ is uniquely determined by the family $\{ \star_{U_i}\}$.
This completes the proof.
\end{proof}

\begin{remark}
In the terminology of \cite{KS_dqmod}, the sheaf $\cO_\fdef$ is a star-algebra on $\CC^n$.
\end{remark}

We will need the following simple construction. Suppose that
$K\to L$ is a homomorphism of commutative Fr\'echet algebras, and that $A$ is a Fr\'echet $K$-algebra.
It is easy to show that the Fr\'echet $L$-module $L\ptens{K} A$ is a Fr\'echet $L$-algebra
with multiplication uniquely determined by
\[
(\alpha\otimes a)(\beta\otimes b)=\alpha\beta\otimes ab
\qquad (\alpha,\beta\in L,\; a,b\in A).
\]

Consider now the Fr\'echet algebra homomorphism
\[
\lambda\colon\cO(\CC^\times)\to\CC[[h]], \quad
z\mapsto e^{ih},
\]
where $z\in\cO(\CC^\times)$ is the complex coordinate.
We let
\[
\tilde A(\DD^n_r)=\CC[[h]]\ptens{\cO(\CC^\times)}\cO_\defo(\DD^n_r),\quad
\tilde A(\BB^n_r)=\CC[[h]]\ptens{\cO(\CC^\times)}\cO_\defo(\BB^n_r).
\]
It follows from the above discussion that $\tilde A(\DD^n_r)$ and $\tilde A(\BB^n_r)$
are Fr\'echet $\CC[[h]]$-algebras in a canonical way.

Our next goal is to construct Fr\'echet $\CC[[h]]$-algebra isomorphisms
$\tilde A(\DD^n_r)\cong\cO_\fdef(\DD^n_r)$ and $\tilde A(\BB^n_r)\cong\cO_\fdef(\BB^n_r)$.
The technical part of the construction will be contained in the following two lemmas.

\begin{lemma}
\label{lemma:free_to_formal}
There exist Fr\'echet algebra homomorphisms
$\cF(\DD^n_r)\to\cO_\fdef(\DD^n_r)$ and
$\cF(\BB^n_r)\to\cO_\fdef(\BB^n_r)$
uniquely determined by
$\zeta_i\mapsto x_i\; (i=1,\ldots ,n)$.
\end{lemma}
\begin{proof}
Let us start with the (more difficult) case of the homomorphism
$\cF(\BB^n_r)\to\cO_\fdef(\BB^n_r)$.
Suppose that $u=\sum_\alpha c_\alpha\zeta_\alpha\in\cF(\BB^n_r)$.
We claim that the family $\sum_\alpha c_\alpha x_\alpha$ is summable in $\cO_\fdef(\BB^n_r)$.
To see this, consider the defining family $\{\|\cdot\|_N : N\in\Z_+\}$ of seminorms on $\CC[[h]]$
given by
\begin{equation*}
\| f\|_N=\sum_{p=0}^N |c_p|\qquad \Bigl(f=\sum_{p\in\Z_+} c_p h^p\in\CC[[h]]\Bigr)
\end{equation*}
and the defining family
$\{\|\cdot\|_{\BB,\rho} : \rho\in (0,r)\}$ of norms on $\cO(\BB^n_r)$
given by~\eqref{ball_pow_rep}. The projective tensor product of the seminorms
$\|\cdot\|_N$ and $\|\cdot\|_{\BB,\rho}$ will be denoted by $\|\cdot\|_{N,\rho}$.
We have
\begin{align}
\sum_{k\in\Z_+^n}\, \biggl\|\sum_{\alpha\in\rp^{-1}(k)} c_\alpha x_\alpha\biggr\|_{N,\rho}
&=\sum_{k\in\Z_+^n}\, \biggl\|\sum_{\alpha\in\rp^{-1}(k)}
c_\alpha e^{-i\rrm(\alpha)h}  x^k\biggr\|_{N,\rho}\notag\\
&=\sum_{k\in\Z_+^n}\, \biggl\|\sum_{\alpha\in\rp^{-1}(k)}
c_\alpha e^{-i\rrm(\alpha)h}\biggr\|_N \biggl(\frac{k!}{|k|!}\biggr)^{1/2} \rho^{|k|}\notag\\
&=\sum_{k\in\Z_+^n}\sum_{j=0}^N\, \biggl| \sum_{\alpha\in\rp^{-1}(k)} c_\alpha
\frac{(-i\rrm(\alpha))^j}{j!}\,\biggr| \biggl(\frac{k!}{|k|!}\biggr)^{1/2} \rho^{|k|}\notag\\
&\le \sum_{k\in\Z_+^n}\sum_{\alpha\in\rp^{-1}(k)} |c_\alpha|
\biggl(\sum_{j=0}^N \frac{\rrm(\alpha)^j}{j!}\biggr)
\biggl(\frac{k!}{|k|!}\biggr)^{1/2} \rho^{|k|}\notag\\
\label{F_to_formal_1}
&\le \sum_{k\in\Z_+^n}\sum_{\alpha\in\rp^{-1}(k)} |c_\alpha|
\biggl(\sum_{j=0}^N \frac{|\alpha|^{2j}}{j!}\biggr) \biggl(\frac{k!}{|k|!}\biggr)^{1/2} \rho^{|k|}.
\end{align}
Here we have used the obvious inequality $\rrm(\alpha)\le |\alpha|^2$.
Now choose $\rho_1\in (\rho,r)$, and find $C_1,C_2>0$ such that
\begin{gather*}
\sum_{j=0}^N\frac{t^{2j}}{j!}\le C_1(1+t^{2N}) \qquad (t\ge 0);\\
C_1(1+t^{2N})\rho^t \le C_2\rho_1^t \qquad (t\ge 0).
\end{gather*}
Since $|\alpha|=|k|$ whenever $\alpha\in\rp^{-1}(k)$, we can continue \eqref{F_to_formal_1}
as follows:
\begin{align}
\sum_{k\in\Z_+^n}\sum_{\alpha\in\rp^{-1}(k)} |c_\alpha|
\biggl(\sum_{j=0}^N \frac{|\alpha|^{2j}}{j!}\biggr) \biggl(\frac{k!}{|k|!}\biggr)^{1/2} \rho^{|k|}
&\le C_1 \sum_{k\in\Z_+^n}\sum_{\alpha\in\rp^{-1}(k)} |c_\alpha|
(1+|k|^{2N}) \biggl(\frac{k!}{|k|!}\biggr)^{1/2} \rho^{|k|}\notag\\
\label{F_to_formal_2}
&\le C_2 \sum_{k\in\Z_+^n}\sum_{\alpha\in\rp^{-1}(k)} |c_\alpha|
\biggl(\frac{k!}{|k|!}\biggr)^{1/2} \rho_1^{|k|}.
\end{align}
Finally, applying the Cauchy-Bunyakowsky-Schwarz inequality together with~\eqref{card_p^{-1}},
we continue~\eqref{F_to_formal_2} as follows:
\begin{equation}
\label{F_to_formal_3}
\begin{split}
C_2 \sum_{k\in\Z_+^n}\sum_{\alpha\in\rp^{-1}(k)} |c_\alpha|
\biggl(\frac{k!}{|k|!}\biggr)^{1/2} \rho_1^{|k|}
&\le C_2 \sum_{k\in\Z_+^n}\biggl(\sum_{\alpha\in\rp^{-1}(k)} |c_\alpha|^2\biggr)^{1/2}
\rho_1^{|k|}\\
&=C_2 \biggl\|\sum_\alpha c_\alpha \zeta_\alpha\biggr\|_{\rho_1}^\circ.
\end{split}
\end{equation}
Thus we see that for each $u=\sum_\alpha c_\alpha\zeta_\alpha\in\cF(\BB^n_r)$
the family $\sum_\alpha c_\alpha x_\alpha$ is summable in $\cO_\fdef(\BB^n_r)$.
Letting $\varphi(u)=\sum_\alpha c_\alpha x_\alpha$, we obtain a linear map
$\varphi\colon\cF(\BB^n_r)\to\cO_\fdef(\BB^n_r)$ such that $\varphi(\zeta_i)=x_i$
for all $i=1,\ldots ,n$. Moreover, \eqref{F_to_formal_1}--\eqref{F_to_formal_3}
imply that $\|\varphi(u)\|_{N,\rho}\le C_2\| u\|_{\rho_1}^\circ$, so that $\varphi$ is continuous.
Since the restriction of $\varphi$ to the dense subalgebra $F_n\subset\cF(\BB^n_r)$
is clearly an algebra homomorphism, we conclude that $\varphi$ is a homomorphism.

The homomorphism $\cF(\DD^n_r)\to\cO_\fdef(\DD^n_r)$ is constructed similarly.
Specifically, we replace $\|\cdot\|_{\BB,\rho}$ by $\|\cdot\|_{\DD,\rho}$ in the above
argument, remove the factor $\bigl(\frac{k!}{|k|!}\bigr)^{1/2}$ from~\eqref{F_to_formal_1}
and~\eqref{F_to_formal_2}, and observe that the last expression in~\eqref{F_to_formal_2}
will be exactly $C_2\|\sum_\alpha c_\alpha\zeta_\alpha\|_{\rho_1,1}$, where
the norm $\|\cdot\|_{\rho_1,1}$ is given by~\eqref{F_poly_expl}.
Thus~\eqref{F_to_formal_3} is not needed in this case. The last step of the construction
 is identical to the case of $\BB^n_r$.
\end{proof}

\begin{remark}
It follows from the above proof that Lemma~\ref{lemma:free_to_formal}
holds with $\cF(\DD^n_r)$ replaced by $\cF^\rT(\DD^n_r)$.
\end{remark}

\begin{remark}
If we knew that $\cO_\fdef(\DD^n_r)$ is an Arens-Michael algebra, then the construction
of the homomorphism $\cF(\DD^n_r)\to\cO_\fdef(\DD^n_r)$ could easily be
deduced from the universal property of $\cF(\DD^n_r)$. As a matter of fact,
$\cO_\fdef(D)$ is indeed an Arens-Michael algebra provided that $D$
is a complete bounded Reinhardt
domain, but the direct proof of this result is rather technical,
so we omit it. Anyway, the fact that $\cO_\fdef(\DD^n_r)$ and $\cO_\fdef(\BB^n_r)$
are Arens-Michael algebras will be immediate from Theorem~\ref{thm:formal}.
However, we do not know whether $\cO_\fdef(U)$ is an Arens-Michael algebra
for every open set $U\subset\CC^n$. The question seems to be open already
in the case where $U$ is a polydisk centered at some point $p\ne 0$.
\end{remark}

From now on, the image of $x_j\in\cO_\defo(\DD^n_r)\; (j=1,\ldots ,n)$ under the map
\[
\cO_\defo(\DD^n_r)\to \tilde A(\DD^n_r)=\CC[[h]]\ptens{\cO(\CC^\times)}\cO_\defo(\DD^n_r),
\quad u\mapsto 1_{\CC[[h]]}\otimes u,
\]
will be denoted by the same symbol $x_j$. This will not lead to a confusion.
The same convention applies to $\cO_\defo(\BB^n_r)$ and $\tilde A(\BB^n_r)$.

\begin{lemma}
\label{lemma:func_to_tilde}
There exist continuous linear maps $\cO(\DD^n_r)\to\tilde A(\DD^n_r)$
and $\cO(\BB^n_r)\to\tilde A(\BB^n_r)$ uniquely determined by $x^k\mapsto x^k\; (k\in\Z_+^n)$.
\end{lemma}
\begin{proof}
As in Lemma~\ref{lemma:free_to_formal}, we start by constructing the map
$\cO(\BB^n_r)\to\tilde A(\BB^n_r)$ (as we shall see, the case of $\DD^n_r$ is much easier).
Applying the functor $\CC[[h]]\ptens{\cO(\CC^\times)}(-)$ to the quotient map
\[
\cO(\CC^\times)\Ptens\cF(\BB^n_r)\to\cO_\defo(\BB^n_r),
\]
we obtain a surjective homomorphism
\[
\tilde\pi_{\BB}\colon\CC[[h]]\Ptens\cF(\BB^n_r)=\CC[[h;\cF(\BB^n_r)]]\to\tilde A(\BB^n_r)
\]
of Fr\'echet $\CC[[h]]$-algebras.
For each $s\in\Z_+$ and each $\rho\in (0,r)$, define a seminorm $\|\cdot\|_{s,\rho}$
on $\CC[[h;\cF(\BB^n_r)]]$ by
\[
\biggl\| \sum_j f_j h^j\biggr\|_{s,\rho}=\| f_s\|_\rho^\circ \qquad (f_j\in\cF(\BB^n_r)).
\]
Clearly, $\{\|\cdot\|_{s,\rho} : s\in\Z_+,\; \rho\in (0,r)\}$ is a (nondirected) defining family of
seminorms on $\CC[[h;\cF(\BB^n_r)]]$.

Given $k\in\Z_+^n$, let
\[
u_k=\frac{k!}{|k|!}\sum_{\alpha\in \rp^{-1}(k)} e^{i\rrm(\alpha)h}\zeta_\alpha\in\CC[[h,\cF(\BB^n_r)]].
\]
Observe that
\begin{equation}
\label{tilde_pi_u}
\tilde\pi_{\BB}(u_k)
=\frac{k!}{|k|!}\sum_{\alpha\in \rp^{-1}(k)} e^{i\rrm(\alpha)h}x_\alpha
=\frac{k!}{|k|!}\sum_{\alpha\in \rp^{-1}(k)}x^k
=x^k.
\end{equation}
Take $s\in\Z_+$ and $\rho\in (0,r)$, choose any $\rho_1\in (\rho,r)$, and find $C>0$
such that
\[
t^{2s}\rho^t\le C\rho_1^t\qquad (t\ge 0).
\]
Using the fact that $\rrm(\alpha)\le |\alpha|^2$, we obtain
\[
\begin{split}
\| u_k\|_{s,\rho}
&=\frac{k!}{|k|!}\,\biggl\|\sum_{\alpha\in\rp^{-1}(k)}
\frac{(i\rrm(\alpha))^s}{s!}\,\zeta_\alpha\biggr\|_\rho^\circ
=\frac{k!}{|k|!}\biggl(\sum_{\alpha\in\rp^{-1}(k)}
\frac{\rrm(\alpha)^{2s}}{(s!)^2}\biggr)^{1/2}\rho^{|k|}\\
&\le |k|^{2s} \left(\frac{k!}{|k|!}\right)^{1/2} \rho^{|k|}
\le C \left(\frac{k!}{|k|!}\right)^{1/2} \rho_1^{|k|}
=C\| x^k\|_{\BB,\rho_1}.
\end{split}
\]
This implies that for each $f=\sum_k c_k x^k\in\cO(\BB^n_r)$ the series
$\sum_k c_k u_k$ absolutely converges in $\CC[[h;\cF(\BB^n_r)]]$.
Moreover, we have
\[
\biggl\| \sum_k c_k u_k\biggr\|_{s,\rho}\le C\| f\|_{\BB,\rho_1}.
\]
Hence we obtain a continuous linear map
\[
\psi\colon\cO(\BB^n_r)\to\CC[[h;\cF(\BB^n_r)]],\quad x^k\mapsto u_k\qquad (k\in\Z_+^n).
\]
Taking into account \eqref{tilde_pi_u}, we conclude that
$\tilde\pi_{\BB}\psi\colon\cO(\BB^n_r)\to\tilde A(\BB^n_r)$ is the map we are looking for.

To construct the map $\cO(\DD^n_r)\to\tilde A(\DD^n_r)$,
observe that for each $k\in\Z_+^n$, each $\rho\in (0,r)$, and each $\tau\ge 1$
we have
\[
\|\zeta^k\|_{\rho,\tau}\le\tau^n\rho^{|k|}=\tau^n\| x^k\|_{\DD,\rho}
\]
(where the norm $\|\cdot\|_{\rho,\tau}$ on $\cF(\DD^n_r)$ is given by~\eqref{F_poly_expl}).
This implies that for each $f=\sum_k c_k x^k\in\cO(\DD^n_r)$ the series
$\sum_k c_k \zeta^k$ absolutely converges in $\cF(\DD^n_r)$ to an element $j(f)$, and that
$j\colon\cO(\DD^n_r)\to\cF(\DD^n_r)$ is a continuous linear map.
Clearly, the composition
\[
\cO(\DD^n_r)\xra{j}\cF(\DD^n_r)\hookrightarrow\CC[[h;\cF(\DD^n_r)]]
\xra{\tilde\pi_{\DD}}\tilde A(\DD^n_r)
\]
(where $\tilde\pi_{\DD}$ is defined similarly to $\tilde\pi_{\BB}$)
is the map we are looking for.
\end{proof}

Here is the main result of this section.

\begin{theorem}
\label{thm:formal}
There exist Fr\'echet $\CC[[h]]$-algebra isomorphisms
$\tilde A(\DD^n_r)\to\cO_\fdef(\DD^n_r)$ and $\tilde A(\BB^n_r)\to\cO_\fdef(\BB^n_r)$
uniquely determined by $x_i\mapsto x_i\; (i=1,\ldots ,n)$.
\end{theorem}
\begin{proof}
Let $D=\DD^n_r$ or $D=\BB^n_r$, and let
\[
\varphi_0\colon\cF(D)\to\cO_\fdef(D),\quad \zeta_i\mapsto x_i\qquad (i=1,\ldots ,n)
\]
be the Fr\'echet algebra homomorphism constructed in Lemma~\ref{lemma:free_to_formal}.
Consider the map
\[
\varphi_1\colon\cO(\CC^\times)\Ptens\cF(D)\to\cO_\fdef(D),\quad
f\otimes u\mapsto\lambda(f)\varphi_0(u).
\]
Clearly, $\varphi_1$ is a Fr\'echet $\cO(\CC^\times)$-algebra homomorphism,
and $\varphi_1$ vanishes on $I_D$. Hence $\varphi_1$ induces a Fr\'echet $\cO(\CC^\times)$-algebra
homomorphism
\[
\varphi_2\colon\cO_\defo(D)\to\cO_\fdef(D),
\quad f\otimes u+I_D\mapsto\lambda(f)\varphi_0(u).
\]
Consider now the map
\[
\varphi\colon\tilde A(D)=\CC[[h]]\ptens{\cO(\CC^\times)}\cO_\defo(D)\to\cO_\fdef(D),
\quad g\otimes v\mapsto g\varphi_2(v).
\]
Clearly, $\varphi$ takes $x_i$ to $x_i$ and is a Fr\'echet $\CC[[h]]$-algebra homomorphism.
We claim that $\varphi$ is an isomorphism. To see this, let
\[
\psi_0\colon\cO(D)\to\tilde A(D),\quad x^k\mapsto x^k
\qquad (k\in\Z_+^n)
\]
be the linear map constructed in Lemma~\ref{lemma:func_to_tilde}. We have a Fr\'echet
$\CC[[h]]$-module morphism
\[
\psi\colon\cO_\fdef(D)=\CC[[h]]\Ptens\cO(D)\to\tilde A(D),
\quad f\otimes g\mapsto f\psi_0(g).
\]
Clearly, $(\varphi\psi)(x^k)=x^k$ and $(\psi\varphi)(x^k)=x^k$ for all $k\in\Z_+^n$.
Since the $\CC[[h]]$-submodule generated by $\{ x^k : k\in\Z_+^n\}$ is dense both in
$\cO_\fdef(D)$ and in $\tilde A(D)$, we conclude that $\varphi\psi=\id_{\cO_\fdef(D)}$
and $\psi\varphi=\id_{\tilde A(D)}$. Hence $\varphi$ is a Fr\'echet $\CC[[h]]$-algebra
isomorphism.
\end{proof}

\appendix

\section{Locally convex bundles}

In this Appendix we collect some facts on bundles of locally convex spaces and algebras.
Most definitions below are modifications of those contained in \cite{Gierz}
(cf. also \cite{DH,Varela95}).
The principal difference between our approach and the approach of \cite{Gierz}
is that we do not endow the total space $E$ of a locally convex bundle with a family
of seminorms. Instead, we introduce a coarser ``locally convex uniform
vector structure'' on $E$ compatible with the topology on $E$
(see Definitions~\ref{def:loc_conv_unif_vect_str} and~\ref{def:LCbun}).
The reason is that we need a functor from locally convex $K$-modules
(where $K$ is a subalgebra of $C(X)$) to locally convex bundles over $X$
(see Theorems~\ref{thm:Kmod-Bnd} and~\ref{thm:Kalg-Bnd}).
It seems that the approach of \cite{Gierz} is not appropriate for this purpose.

By a {\em family of vector spaces} over a set $X$ we mean a pair $(E,p)$, where
$E$ is a set and $p\colon E\to X$ is a surjective map, together with a vector space structure
on each fiber $E_x=p^{-1}(x)\; (x\in X)$. As usual, we let
\[
E\times_X E=\{ (u,v)\in E\times E : p(u)=p(v)\}.
\]
For a subset $V\subset X$, a map $s\colon V\to E$ is a {\em section} of $(E,p)$ over $V$ if
$ps=\id_V$. It will be convenient to denote the value of $s$ at $x\in V$ by $s_x$.
The vector space of all sections of $(E,p)$ over $V$ will be denoted by $S(V,E)$.
In other words, $S(V,E)=\prod_{x\in V} E_x$.

\begin{definition}
Let $X$ be a topological space. By a {\em prebundle of topological vector spaces}
over $X$ we mean a family $(E,p)$ of vector spaces over $X$ together with a topology on $E$
such that $p$ is continuous and open, the zero section $0\colon X\to E$ is continuous,
and the operations
\begin{equation*}
\begin{split}
E\times_X E&\to E, \quad (u,v)\mapsto u+v,\\
\CC\times E&\to E,\quad (\lambda, v)\mapsto \lambda v,
\end{split}
\end{equation*}
are also continuous.
\end{definition}

\begin{remark}
\label{rem:fiber_TVS}
If $(E,p)$ is a prebundle of topological vector spaces, then each fiber $E_x$
is clearly a topological vector space with respect to the topology inherited from $E$.
\end{remark}

If $(E,p)$ is a prebundle of topological vector spaces over $X$, then
the set of all continuous sections of $(E,p)$ over $V\subset X$
will be denoted by $\Gamma(V,E)$. The above axioms readily imply that
$\Gamma(V,E)$ is a vector subspace of $S(V,E)$.

Prebundles of topological vector spaces are too general objects for our purposes.
First, we would like that the topology on each fiber $E_x$ inherited from $E$ be locally convex.
What is more important, we would like to consider prebundles endowed with an additional structure that would
enable us to ``compare'' $0$-neighborhoods in different fibers. This can be achieved as follows.

Let $(E,p)$ be a family of vector spaces over a set $X$.
We say that a subset $S\subset E$ is {\em absolutely convex}
(respectively, {\em absorbing}) if $S\cap E_x$ is absolutely convex
(respectively, absorbing) in $E_x$, for each $x\in X$.

\begin{definition}
\label{def:loc_conv_unif_vect_str}
Let $X$ be a set, and let $(E,p)$ be a family of vector spaces over $X$.
By a {\em locally convex uniform vector structure} on $E$ we mean a family
$\cU$ of subsets of $E$ satisfying the following conditions:
\begin{mycompactenum}
\item[$\mathrm{(U1)}$]
Each $U\in\cU$ is absorbing and absolutely convex.
\item[$\mathrm{(U2)}$]
If $U\in\cU$ and $\lambda\in\CC\setminus\{ 0\}$, then $\lambda U\in\cU$.
\item[$\mathrm{(U3)}$]
If $U,V\in\cU$, then $U\cap V\in\cU$.
\item[$\mathrm{(U4)}$]
If $U\in\cU$ and $V\supset U$ is an absolutely convex subset of $E$, then $V\in\cU$.
\end{mycompactenum}
\end{definition}

\begin{example}
If $X$ is a single point and $E$ is just a vector space, then there is a $1$-$1$ correspondence
between locally convex topologies on $E$ (i.e., topologies on $E$ making $E$ into a locally convex
topological vector space) and locally convex uniform vector structures on $E$.
Indeed, if $E$ is endowed with a locally convex topology, then the family of all absolutely
convex $0$-neighborhoods in $E$ is a locally convex uniform vector structure on $E$.
Moreover, each locally convex uniform vector structure arises in this way
(see, e.g., \cite[\S18.1]{Kothe_I}).
\end{example}

\begin{remark}
\label{rem:uniform_str}
A locally convex uniform vector structure $\cU$ on $E$
determines a uniform structure $\tilde\cU$
(in Weil's sense; see, e.g., \cite[Chap. 8]{Engelking})
on $E$ whose basis consists of all sets of the form
\[
\tilde U=\bigl\{ (u,v)\in E\times_X E : u-v\in U\bigr\}
\qquad (U\in\cU).
\]
\end{remark}

\begin{definition}
Let $X$ be a set, $(E,p)$ be a family of vector spaces over $X$, and $\cU$
be a locally convex uniform vector structure on $E$. We say that a subfamily $\cB\subset\cU$
is a {\em base} of $\cU$ if for each $U\in\cU$ there exists $B\in\cB$ such that $B\subset U$.
\end{definition}

Clearly, a filter base $\cB\subset 2^E$ consisting of absorbing, absolutely convex subsets of $E$
is a base of a (necessarily unique) locally convex uniform vector structure if and only if
$\cB$ satisfies the condition
\begin{mycompactenum}
\item[$\mathrm{(BU)}$]
For each $B\in\cB$ there exists $B'\in\cB$ such that $B'\subset (1/2)B$.
\end{mycompactenum}

It is a standard fact that each locally convex topology on a vector space
is generated by a family of seminorms. This result can easily be extended to
locally convex uniform vector structures on a family of vector spaces.
Let $(E,p)$ be a family of vector spaces over a set $X$.
By definition, a function $\|\cdot\|\colon E\to [0,+\infty)$ is a {\em seminorm}
if the restriction of $\|\cdot\|$ to each fiber is a seminorm in the usual sense.
Suppose that $\cN=\{ \|\cdot\|_\lambda : \lambda\in\Lambda\}$
is a family of seminorms on $E$.
We assume that $\cN$ is {\em directed}, that is, for each
$\lambda,\mu\in\Lambda$ there exist $C>0$ and $\nu\in\Lambda$
such that $\|\cdot\|_\lambda\le C\|\cdot\|_\nu$ and $\|\cdot\|_\mu\le C\|\cdot\|_\nu$.
Given $\lambda\in\Lambda$ and $\eps>0$, let $\sU_{\lambda,\eps}=\{ u\in E : \| u\|_\lambda<\eps\}$.

\begin{prop}
\label{prop:unifstruct_seminorms}
Let $X$ be a set, and let $(E,p)$ be a family of vector spaces over $X$.
\begin{mycompactenum}
\item For each directed family $\cN=\{ \|\cdot\|_\lambda : \lambda\in\Lambda\}$
of seminorms on $E$, the family $\cB_\cN=\{ \sU_{\lambda,\eps} : \lambda\in\Lambda,\; \eps>0\}$
is a base of a locally convex uniform vector structure $\cU_\cN$ on $E$.
\item Conversely, for each locally convex uniform vector structure $\cU$ on $E$ there exists
a directed family $\cN$ of seminorms on $E$ such that $\cU=\cU_\cN$.
Specifically, we can take $\cN=\{ p_B : B\in\cB\}$, where $\cB$ is any base of $\cU$
and $p_B$ is the Minkowski functional of $B$.
\end{mycompactenum}
\end{prop}
\begin{proof}
(i) Clearly, $\cB_\cN$ is a filter base on $E$, each set belonging to $\cB_\cN$ is absorbing
and absolutely convex, and $\cB_\cN$ satisfies $\mathrm{(BU)}$.

(ii) Given $B_1,B_2\in\cB$, find $B_3\in \cB$ such that $B_3\subset B_1\cap B_2$.
We clearly have $\max\{ p_{B_1},p_{B_2}\}\le p_{B_3}$, which implies that $\cN$ is directed.
It is a standard fact that for each absorbing, absolutely convex set $B$ we have
\begin{equation*}
\{ u\in E : p_B(u)<1\} \subset B \subset \{ u\in E : p_B(u)\le 1\}.
\end{equation*}
Thus for each $B\in\cB$ we have $\sU_{p_B,1}\subset B$,
whence $\cU\subset\cU_\cN$. On the other hand, for each $B\in\cB$ and each $\eps>0$
we have $(\eps/2)B\subset\sU_{p_B,\eps}$,
whence $\cU_\cN\subset\cU$.
\end{proof}

\begin{definition}
If $\cN=\{ \|\cdot\|_\lambda : \lambda\in\Lambda\}$ and
$\cN'=\{ \|\cdot\|_\mu : \mu\in\Lambda'\}$ are two directed families of seminorms on $E$,
then we say that $\cN$ is {\em dominated} by $\cN'$ (and write $\cN\prec\cN'$)
if for each $\lambda\in\Lambda$ there exist $C>0$ and $\mu\in\Lambda'$
such that $\|\cdot\|_\lambda\le C\|\cdot\|_\mu$.
If $\cN\prec\cN'$ and $\cN'\prec\cN$, then we say that $\cN$ and $\cN'$ are
{\em equivalent} and write $\cN\sim\cN'$.
\end{definition}

A standard argument shows that $\cN\prec\cN'$ if and only if
$\cU_{\cN}\subset\cU_{\cN'}$, and hence
$\cN\sim\cN'$ if and only if $\cU_{\cN}=\cU_{\cN'}$.

Given a family $(E,p)$ of vector spaces over $X$, let
\[
\add\colon E\times_X E \to E, \quad (u,v)\mapsto u+v.
\]
For each pair $S,T$ of subsets of $E$, we let $S+T=\add((S\times T)\cap (E\times_X E))$.

\begin{lemma}
\label{lemma:open_sum}
Let $X$ be a topological space, and let $(E,p)$ be a prebundle of topological vector spaces over $X$.
For each open set $V\subset X$, each $s\in\Gamma(V,E)$, and each open set $U\subset E$,
the set $s(V)+U$ is open in $E$.
\end{lemma}
\begin{proof}
Clearly, the map $p^{-1}(V)\to E,\; v\mapsto v-s_{p(v)}$, is continuous.
Now the result follows from the equality
\[
s(V)+U=\bigl\{ v\in p^{-1}(V) : v-s_{p(v)}\in U\bigr\}. \qedhere
\]
\end{proof}

\begin{definition}
\label{def:LCbun}
Let $X$ be a topological space. A {\em bundle of locally convex spaces}
(or a {\em locally convex bundle}) over $X$ is a triple $(E,p,\cU)$, where $(E,p)$
is a prebundle of topological vector spaces over $X$ and $\cU$ is a locally convex
uniform vector structure on $E$ satisfying the following compatibility axioms:
\begin{mycompactenum}
\item[$\mathrm{(B0)}$]
For each $U\in\cU$ there exists an open subset $U_0\subset U$, $U_0\in\cU$.
\item[$\mathrm{(B1)}$]
The family
\[
\bigl\{ s(V)+U : V\subset X\text{ open, } s\in \Gamma(V,E),\; U\in\cU,\; U \text{ is open}\bigr\}
\]
is a base for the topology on $E$.
\item[$\mathrm{(B2)}$]
For each $x\in X$, the set
\begin{equation*}
\bigl\{ s_x : s\in\Gamma(V,E),\; V\text{ is an open neighborhood of $x$}\bigr\}
\end{equation*}
is dense in $E_x$ with respect to the topology inherited from $E$.
\end{mycompactenum}
If, in addition, each fiber $E_x$ is a Fr\'echet space with respect to the topology inherited from $E$,
then we say that $(E,p,\cU)$ is a {\em Fr\'echet space bundle}.
\end{definition}

\begin{remark}
We would like to stress that the topology on $E$ does not coincide with the topology
$\tau(\cU)$ determined by the uniform structure $\tilde\cU$ (see Remark~\ref{rem:uniform_str}).
In fact, each fiber $E_x$ is $\tau(\cU)$-open, which is not the case for the original
topology on $E$ (unless $X$ is discrete). Axiom $\mathrm{(B1)}$ implies that $\tau(\cU)$ is stronger than
the original topology on $E$.
\end{remark}

\begin{lemma}
\label{lemma:bas_unif}
Let $X$ be a topological space, and let $(E,p,\cU)$ be a locally convex bundle over $X$.
Suppose that $\cB$ is a base of $\cU$ consisting of open sets. Then
\begin{mycompactenum}
\item
the family
\[
\bigl\{ s(V)+B : V\subset X\text{ open, } s\in \Gamma(V,E),\; B\in\cB\bigr\}
\]
is a base for the topology on $E$;
\item
for each open set $X'\subset X$, each $s\in\Gamma(X',E)$, and each $x\in X'$, the family
\[
\bigl\{ s(V)+B : V\subset X'\text{ is an open neighborhood of $x$, } B\in\cB\bigr\}
\]
is a base of open neighborhoods of $s_x$.
\end{mycompactenum}
\end{lemma}
\begin{proof}
Let $W$ be an open subset of $E$, and let $u\in W$. Without loss of generality, we may
assume that $W=s(V)+U$ for some open set $V\subset X$, $s\in\Gamma(V,E)$,
and an open set $U\in\cU$. Letting $x=p(u)$, we see that $u-s_x\in U$.
Since $U$ is open, there exists $\delta\in (0,1)$ such that $u-s_x\in (1-2\delta)U$.
Choose $B\in\cB$ such that $B\subset\delta U$. By Remark~\ref{rem:fiber_TVS},
$u+(B\cap E_x)$ is a neighborhood of $u$ in $E_x$.
Hence $\mathrm{(B2)}$ implies that there exist an open neighborhood $V'\subset V$ of $x$
and $t\in\Gamma(V',E)$ such that $t_x\in u+B$. We have
\[
t_x-s_x=(t_x-u)+(u-s_x)\in \delta U + (1-2\delta)U=(1-\delta)U.
\]
Hence there exists an open neighborhood $V''\subset V'$ of $x$ such that $(t-s)(V'')\subset (1-\delta)U$.
We clearly have $u\in t(V'')+B$, and
\[
t(V'')+B\subset s(V'')+(t-s)(V'')+B\subset s(V)+(1-\delta)U+\delta U=s(V)+U.
\]
This completes the proof of (i). The proof of (ii) is similar and is therefore omitted
(cf. also \cite{Gierz}).
\end{proof}

Let us now study relations between locally convex bundles in the sense of Definition~\ref{def:LCbun}
and locally convex bundles in the sense of \cite{Gierz}.
Let $(E,p)$ be a family of vector spaces over a set $X$, and let
$\cN=\{ \|\cdot\|_\lambda : \lambda\in\Lambda\}$
be a directed family of seminorms on $E$.
Given $\lambda\in\Lambda$, $\eps>0$, a set $V\subset X$, and a section $s\colon V\to E$,
define the ``$\eps$-tube'' around $s$ by
\[
\sT(V,s,\lambda,\eps)=\bigl\{ v\in E : p(v)\in V,\; \| v-s_{p(v)}\|_\lambda<\eps\bigr\}.
\]
Observe that $\sT(X,0,\lambda,\eps)=\sU_{\lambda,\eps}$ and
\begin{equation}
\label{tube=tube}
\sT(V,s,\lambda,\eps)=s(V)+\sU_{\lambda,\eps}.
\end{equation}

\begin{definition}
Let $X$ be a topological space, $(E,p)$ be a prebundle of topological vector spaces over $X$,
and $\cN=\{ \|\cdot\|_\lambda : \lambda\in\Lambda\}$ be a directed family of seminorms on $E$.
We say that $\cN$ is {\em admissible} if the following conditions hold
(cf. \cite{Hof_Kei,Varela84,Gierz,Ara_Mat}):
\begin{mycompactenum}
\item[$\mathrm{(Ad1)}$]
The family
\[
\cB(\cN)=\bigl\{ \sT(V,s,\lambda,\eps) : V\subset X \text{ open},\;
s\in\Gamma(V,E),\;  \lambda\in\Lambda,\; \eps>0\bigr\}
\]
consists of open sets and is a base for the topology on $E$.
\item[$\mathrm{(Ad2)}$]
For each $x\in X$, the set
\begin{equation*}
\bigl\{ s_x : s\in\Gamma(V,E),\; V\text{ is an open neighborhood of $x$}\bigr\}
\end{equation*}
is dense in $E_x$ with respect to the topology generated by the restrictions of the seminorms
$\|\cdot\|_\lambda \; (\lambda\in\Lambda)$ to $E_x$.
\end{mycompactenum}
\end{definition}

\begin{remark}
\label{rem:up_semi_open}
The condition that all sets belonging to $\cB(\cN)$ are open is equivalent to
the upper semicontinuity of all seminorms from $\cN$.
Indeed, $\|\cdot\|_\lambda$ is upper semicontinuous if and only if
$\sU_{\lambda,\eps}=\sT(X,0,\lambda,\eps)$ is open for all $\eps>0$.
By \eqref{tube=tube} and Lemma~\ref{lemma:open_sum},
this implies that each $\sT(V,s,\lambda,\eps)\in\cB(\cN)$ is open.
\end{remark}

\begin{remark}
\label{rem:strong_admiss}
Some authors (e.g., \cite{Gierz,Ara_Mat}) use a stronger form of
$\mathrm{(Ad1)}$ and $\mathrm{(Ad2)}$ in which $\Gamma(V,E)$ is replaced
by the space $\Gamma_{\cN}(V,E)$ of {\em $\cN$-bounded} sections
(i.e., those $s\in\Gamma(V,E)$ for which the function $x\mapsto\| s(x)\|_\lambda$
is bounded for every $\lambda\in\Lambda$). This restriction is not needed
for our purposes. Anyway, the stronger form of $\mathrm{(Ad1)}$ and $\mathrm{(Ad2)}$
is clearly equivalent to ours in the case where $X$ is locally compact.
\end{remark}

\begin{lemma}
\label{lemma:eqv_fam_1}
Let $(E,p)$ be a prebundle of topological vector spaces over $X$, and let $\cN$ and $\cN'$ be
directed families of upper semicontinuous seminorms on $E$.
\begin{mycompactenum}
\item
If $\cN\prec\cN'$ and $\cN'$ satisfies $\mathrm{(Ad2)}$, then
each set from $\cB(\cN)$ is a union of sets belonging to $\cB(\cN')$.
\item
If $\cN\sim\cN'$ and $\cN$ is admissible, then so is $\cN'$.
\end{mycompactenum}
\end{lemma}
\begin{proof}
(i) Let $\cN=\{ \|\cdot\|_\lambda : \lambda\in\Lambda\}$ and
$\cN'=\{ \|\cdot\|_\mu : \mu\in\Lambda'\}$.
Take any $\sT(V,s,\lambda,\eps)\in\cB(\cN)$, and let $u\in\sT(V,s,\lambda,\eps)$.
Choose $\mu\in\Lambda'$ and $C>0$ such that $\|\cdot\|_\lambda\le C\|\cdot\|_\mu$.
Let $x=p(u)$, and find $\delta>0$ such that
\[
2C\delta+\| u-s_x\|_\lambda<\eps.
\]
Since $\cN'$ satisfies $\mathrm{(Ad2)}$, there exist an open neighborhood $W$ of $x$
and $t\in\Gamma(W,E)$ such that
$\| t_x-u\|_\mu<\delta$.
We have
\[
\| t_x-s_x\|_\lambda\le \| t_x-u\|_\lambda+\| u-s_x\|_\lambda
\le C\| t_x-u\|_\mu+\| u-s_x\|_\lambda
<C\delta+\| u-s_x\|_\lambda.
\]
Since $\|\cdot\|_\lambda$ is upper semicontinuous, there exists an open
neighborhood $W'$ of $x$, $W'\subset V\cap W$, such that
\[
\| t_y-s_y\|_\lambda<C\delta+\| u-s_x\|_\lambda \quad (y\in W').
\]
Clearly, $u\in\sT(W',t,\mu,\delta)$. We claim that
\begin{equation}
\label{W_subset_U}
\sT(W',t,\mu,\delta)\subset\sT(V,s,\lambda,\eps).
\end{equation}
Indeed, let $v\in\sT(W',t,\mu,\delta)$, and let $y=p(v)$.
We have
\[
\begin{split}
\| v-s_y\|_\lambda
&\le\| v-t_y\|_\lambda+\| t_y-s_y\|_\lambda\\
&\le C\| v-t_y\|_\mu+\| t_y-s_y\|_\lambda
<C\delta+C\delta+\| u-s_x\|_\lambda<\eps.
\end{split}
\]
This implies \eqref{W_subset_U} and completes the proof of (i). Part (ii) is immediate from (i).
\end{proof}

\begin{lemma}
\label{lemma:two_top_fiber}
Let $X$ be a topological space, $(E,p)$ be a prebundle of topological vector spaces over $X$,
and $\cN=\{ \|\cdot\|_\lambda : \lambda\in\Lambda\}$
be a directed family of seminorms on $E$ satisfying $\mathrm{(Ad1)}$.
Then for each $x\in X$ the topology on $E_x$ inherited from $E$ coincides with
the topology generated by the restrictions of the seminorms
$\|\cdot\|_\lambda \; (\lambda\in\Lambda)$ to $E_x$.
\end{lemma}
\begin{proof}
Let $\tau$ denote the topology on $E_x$ inherited from $E$, and let
$\tau'$ denote the topology on $E_x$ generated by the restrictions of the seminorms
$\|\cdot\|_\lambda \; (\lambda\in\Lambda)$.
The standard base for $\tau'$ consists of all sets of the form
\[
B_{\lambda,\eps}(v)=\{ u\in E_x : \| u-v\|_\lambda<\eps\}
\qquad (v\in E_x,\; \lambda\in\Lambda,\; \eps>0).
\]
Since $\sT(V,s,\lambda,\eps)\cap E_x=B_{\lambda,\eps}(s_x)$,
$\mathrm{(Ad1)}$ implies that $\tau\subset\tau'$.
On the other hand, we have
\begin{equation*}
B_{\lambda,\eps}(v)=v+B_{\lambda,\eps}(0_x)=
v+\sT(X,0,\lambda,\eps)\cap E_x.
\end{equation*}
Together with Remark~\ref{rem:fiber_TVS}, this implies that
$B_{\lambda,\eps}(v)$ is $\tau$-open. Thus $\tau=\tau'$.
\end{proof}

\begin{prop}
\label{prop:bundles_via_seminorms}
Let $X$ be a topological space, and let $(E,p)$ be a prebundle of topological vector
spaces over $X$. A locally convex uniform vector structure $\cU$ on $E$
satisfies $\mathrm{(B0)}$--$\mathrm{(B2)}$ if and only if
$\cU=\cU_\cN$ for an admissible directed family of seminorms on $E$.
Specifically, given $\cU$, we can take $\cN=\{ p_B : B\in\cB\}$, where $\cB$
is any base of $\cU$ consisting of open sets.
\end{prop}
\begin{proof}
Let $\cN$ be an admissible directed family of seminorms on $E$.
The set $\sU_{\lambda,\eps}=\sT(X,0,\lambda,\eps)$ is open by $\mathrm{(Ad1)}$,
and so $\cU_\cN$ satisfies $\mathrm{(B0)}$. It is also immediate from
$\mathrm{(Ad1)}$ and~\eqref{tube=tube} that $\cU_\cN$ satisfies $\mathrm{(B1)}$.
Finally, $\mathrm{(B2)}$ follows from~$\mathrm{(Ad2)}$
and Lemma~\ref{lemma:two_top_fiber}.

Conversely, let $\cU$ be a locally convex uniform vector structure on $E$
satisfying $\mathrm{(B0)}$--$\mathrm{(B2)}$, and let
$\cB$ be a base of $\cU$ consisting of open sets.
Applying Proposition~\ref{prop:unifstruct_seminorms}, we see that
$\cU=\cU_\cN$, where $\cN=\{ p_B : B\in\cB\}$.
We claim that $\cN$ is admissible. Indeed, a standard argument (see, e.g., \cite[I.4]{RR})
shows that for each $B\in\cB$ we have $B=\sU_{p_B,1}$.
Together with \eqref{tube=tube}, this implies that for each open set $V\subset X$,
each $s\in\Gamma(V,E)$, and each $\eps>0$ we have
\begin{equation}
\label{tube=tube_2}
\sT(V,s,p_B,\eps)=s(V)+\sU_{p_B,\eps}=s(V)+\eps\sU_{p_B,1}=s(V)+\eps B.
\end{equation}
Now $\mathrm{(Ad1)}$ follows from
\eqref{tube=tube_2} and Lemma~\ref{lemma:bas_unif} (i).
Finally, $\mathrm{(Ad2)}$ is immediate from $\mathrm{(B2)}$
and Lemma~\ref{lemma:two_top_fiber}.
\end{proof}

\begin{remark}
According to Gierz \cite{Gierz}, a locally convex bundle over $X$ is
a triple $(E,p,\cN)$, where $(E,p)$ is a prebundle of topological vector spaces
over $X$, and $\cN$ is an admissible (in the strong sense, see Remark~\ref{rem:strong_admiss})
directed family of seminorms on $E$.
For our purposes, however, the locally convex uniform vector structure determined by $\cN$
is more important than $\cN$ itself.
Thus our point of view is closer to that of \cite{DH} and \cite{Varela95}.
In some sense, the difference between our bundles and bundles in the
sense of \cite{Gierz} is the same as between locally convex spaces and
polynormed spaces (i.e., vector spaces endowed with distinguished families of
seminorms, see \cite{X2}).
However, in contrast to the case of topological vector spaces,
neither the locally convex uniform vector structure determines the topology of $E$,
nor vice versa.
\end{remark}

\begin{definition}
Let $(E,p,\cU)$ and $(E',p',\cU')$ be locally convex bundles over $X$. A continuous map
$f\colon E\to E'$ is a {\em bundle morphism} if the following holds:
\begin{mycompactenum}
\item[$\mathrm{(BM1)}$]
$p'f=p$;
\item[$\mathrm{(BM2)}$]
the restriction of
$f$ to each fiber $E_x\; (x\in X)$ is a linear map from $E_x$ to $E'_x$;
\item[$\mathrm{(BM3)}$]
for each $U'\in\cU'$, we have $f^{-1}(U')\in\cU$.
\end{mycompactenum}
\end{definition}

\begin{remark}
Clearly, (BM3) is equivalent to the uniform continuity of $f$
with respect to the uniform
structures $\tilde\cU$ and $\tilde\cU'$ (see Remark~\ref{rem:uniform_str}).
\end{remark}

The category of all locally convex bundles and bundle morphisms over $X$
will be denoted by $\Bnd(X)$.

The following result (see \cite[5.2--5.7]{Gierz}) is a locally convex version of
Fell's theorem \cite{Fell_Mackey} (cf. also \cite[II.13.18]{FD},
\cite[C.25]{Williams}, \cite{Varela84}, \cite{Varela95}). It provides
a useful way of constructing locally convex bundles out of their fibers,
seminorms, and sections.

\begin{prop}
\label{prop:sec_bnd}
Let $X$ be a topological space, $(E,p)$ be a family of vector spaces over $X$,
and $\cN=\{\|\cdot\|_\lambda : \lambda\in\Lambda\}$ be a directed family of seminorms on $E$.
Suppose that $\Gamma$ is a vector subspace of $S(X,E)$
satisfying the following conditions:
\begin{mycompactenum}
\item[$\mathrm{(S1)}$]
For each $x\in X$, the set $\{ s_x : s\in \Gamma\}$ is dense in $E_x$
with respect to the topology generated by the restrictions of the seminorms
$\|\cdot\|_\lambda\; (\lambda\in\Lambda)$ to $E_x$;
\item[$\mathrm{(S2)}$]
For each $s\in \Gamma$ and each $\lambda\in\Lambda$, the map $X\to\R,\; x\mapsto \| s_x\|_\lambda$,
is upper semicontinuous.
\end{mycompactenum}
Then there exists a unique topology on $E$ such that
$(E,p,\cU_\cN)$ is a locally convex bundle
and such that $\Gamma\subset\Gamma(X,E)$. Moreover, the family
\begin{equation}
\label{base}
\bigl\{ \sT(V,s,\lambda,\eps) : V\subset X \text{ open},\;
s\in \Gamma,\; \lambda\in\Lambda,\;  \eps>0\bigr\}
\end{equation}
is a base for the topology on $E$.
\end{prop}

\begin{remark}
The uniqueness part of Proposition~\ref{prop:sec_bnd} is not proved in \cite{Gierz},
so let us explain why the topology $\cT$ with the above properties is unique, i.e.,
why~\eqref{base} must be a base of $\cT$.
Let $W\subset E$ be an open set, and let $u\in W$. Without loss of generality,
we may assume that $W=\sT(V,s,\lambda,\eps)\in\cB(\cN)$.
Let $x=p(u)$. Since $\| u-s_x\|_\lambda<\eps$, we can choose $\delta>0$ such that
\[
2\delta+\| u-s_x\|_\lambda<\eps.
\]
By $\mathrm{(S1)}$, there exists $t\in \Gamma$ with $\| t_x-u\|_\lambda<\delta$.
Then $\| t_x-s_x\|_\lambda<\delta+\| u-s_x\|_\lambda$. By the upper semicontinuity
of $\|\cdot\|_\lambda$, there exists an open neighborhood $V'$ of $x$ such that $V'\subset V$ and
\[
\| t_y-s_y\|_\lambda<\delta+\| u-s_x\|_\lambda\qquad (y\in V').
\]
By construction, $u\in\sT(V',t,\lambda,\delta)$. We claim that
\begin{equation}
\label{uniq_top}
\sT(V',t,\lambda,\delta) \subset \sT(V,s,\lambda,\eps).
\end{equation}
Indeed, let $v\in\sT(V',t,\lambda,\delta)$, and let $y=p(v)$. We have
\[
\| v-s_y\|_\lambda<\delta+\| t_y-s_y\|_\lambda<2\delta+\| u-s_x\|_\lambda<\eps.
\]
This proves \eqref{uniq_top} and implies that \eqref{base} is a base of $\cT$.

Alternatively, the uniqueness of $\cT$ can be proved by adapting Fell's original argument
(see \cite{Fell_Mackey}, \cite[II.13.18]{FD},
or \cite[C.25]{Williams}) to the locally convex setting.
\end{remark}

\begin{remark}
\label{rem:sec_bnd_equiv_fam}
Under the conditions of Proposition~\ref{prop:sec_bnd}, we can replace $\cN$ by
any directed subfamily equivalent to $\cN$. By the uniqueness
part of Proposition~\ref{prop:sec_bnd}, this will not affect the topology and the
locally convex uniform vector structure on $E$.
\end{remark}

Let us now describe a situation where conditions $\mathrm{(S1)}$ and $\mathrm{(S2)}$
of Proposition~\ref{prop:sec_bnd}
are satisfied automatically. Let $K$ be a commutative algebra.
By a {\em locally convex $K$-module} we mean a $K$-module $M$ together
with a locally convex topology such that for each $a\in K$ the map $M\to M,\; x\mapsto ax$,
is continuous.
The category of all locally convex $K$-modules and continuous $K$-module
morphisms will be denoted by $K\lmod$.
Suppose that $X$ is a topological space and $\gamma\colon K\to C(X)$
is an algebra homomorphism. Given $a\in K$ and $x\in X$, we write $a(x)$ for $\gamma(a)(x)$.
Define $\eps_x\colon K\to\CC$ by $\eps_x(a)=a(x)$, and let $\fm_x=\Ker\eps_x$.
Given a locally convex $K$-module $M$ and $u\in M$, let
\begin{equation}
\label{fiber}
M_x=M/\ol{\fm_x M},\qquad
u_x=u+\ol{\fm_x M}\in M_x.
\end{equation}
We say that $M_x$ is the {\em fiber} of $M$ over $x\in X$.
If $\|\cdot\|$ is a continuous seminorm on $M$, then the respective quotient
seminorm on $M_x$ will be denoted by the same symbol $\|\cdot\|$;
this will not lead to confusion.

The following lemma is a locally convex version of \cite[Proposition 1.2]{Rf_contfld}.

\begin{lemma}
\label{lemma:semicont}
For each $u\in M$, the function $X\to\R,\; x\mapsto \| u_x\|$, is upper semicontinuous.
\end{lemma}
\begin{proof}
Let $x\in X$, and suppose that $\| u_x\|<C$. We need to show that
$\| u_y\|<C$ as soon as $y$ is close enough to $x$.
We have
\[
\inf\{ \| u+v\| : v\in \fm_x M\}
=\inf\{ \| u+v\| : v\in \ol{\fm_x M}\}
=\| u_x\|<C,
\]
and so there exist $a_1,\ldots ,a_n\in \fm_x$ and $v_1,\ldots ,v_n\in M$ such that
\begin{equation}
\label{norm<C}
\Bigl\| u+\sum_{i=1}^n a_i v_i\Bigr\|<C.
\end{equation}
Observe that for each $a\in K$ and each $y\in X$ we have $a-a(y)\in\fm_y$. Hence
\begin{equation}
\label{u_y_y}
\| u_y\|
\le \Bigl\| u+\sum_{i=1}^n (a_i-a_i(y)) v_i\Bigr\|
\le \Bigl\| u+\sum_{i=1}^n a_i v_i\Bigr\|+\sum_{i=1}^n |a_i(y)| \| v_i\|.
\end{equation}
Since each $a_i$ is continuous and vanishes at $x$, \eqref{norm<C} and \eqref{u_y_y}
together imply that there exists a neighborhood
$V$ of $x$ such that $\| u_y\|<C$ for all $y\in V$. This completes the proof.
\end{proof}

We are now in a position to construct a fiber-preserving functor from $K\lmod$
to $\Bnd(X)$.
Given a locally convex $K$-module $M$, let $\sE(M)=\bigsqcup_{x\in X} M_x$,
and let $p_M\colon\sE(M)\to X$ be given by $p_M(M_x)=\{ x\}$.
Thus $(\sE(M),p_M)$ is a family of vector spaces over $X$.
Let $\cC_M=\{ \|\cdot\|_\lambda : \lambda\in\Lambda\}$ denote the family of all
continuous seminorms on $M$. For each $\lambda\in\Lambda$ and each $x\in X$,
the quotient seminorm of $\|\cdot\|_\lambda$ on $M_x$ will be denoted
by the same symbol $\|\cdot\|_\lambda$ (see discussion before Lemma~\ref{lemma:semicont}).
Thus we obtain a directed family $\cN_M=\{\|\cdot\|_\lambda : \lambda\in\Lambda\}$ of seminorms
on $\sE(M)$. The locally convex uniform vector structure
on $\sE(M)$ determined by $\cN_M$ will be denoted by $\cU_M$.
For each $u\in M$, the function $\tilde u\colon X\to\sE(M),\; x\mapsto u_x$, is clearly
a section of $(\sE(M),p_M)$. Let $\Gamma_M=\{ \tilde u : u\in M\}$.
For each $x\in X$, we obviously have $\{ \tilde u_x : \tilde u\in\Gamma_M\}=M_x$,
and so $\mathrm{(S1)}$ holds. Lemma~\ref{lemma:semicont} implies that
$\mathrm{(S2)}$ holds as well. Applying Proposition~\ref{prop:sec_bnd},
we see that $(\sE(M),p_M,\cU_M)$ is a locally convex bundle over $X$.
For brevity, we will denote every basic open set $\sT(V,\tilde u,\lambda,\eps)$ in $\sE(M)$
(where $V\subset X$ is an open set,
$u\in M$, $\lambda\in\Lambda$, and $\eps>0$) simply by $\sT(V,u,\lambda,\eps)$.

\begin{remark}
\label{rem:mod_bnd_equiv_fam}
In the above construction, we can let $\cC_M$ be any directed defining family
of seminorms on $M$. By Remark~\ref{rem:sec_bnd_equiv_fam},
this will not affect the topology and the locally convex uniform vector structure on $\sE(M)$.
\end{remark}

Suppose now that $f\colon M\to N$ is a morphism in $K\lmod$. For each $x\in X$
we clearly have $f(\ol{\fm_x M})\subset \ol{\fm_x N}$. Hence $f$
induces a continuous linear map $f_x\colon M_x\to N_x,\; u_x\mapsto f(u)_x$.
We let $\sE(f)\colon\sE(M)\to\sE(N)$ denote the map whose restriction to each fiber $M_x$
is $f_x$.

\begin{lemma}
$\sE(f)\colon \sE(M)\to\sE(N)$ is a bundle morphism.
\end{lemma}
\begin{proof}
Clearly, $\sE(f)$ satisfies $\mathrm{(BM1)}$ and $\mathrm{(BM2)}$.
Let $u\in M$ and $x\in X$, and let us prove the continuity of $\sE(f)$ at $u_x\in M_x$.
Let $\{ \|\cdot\|_\lambda : \lambda\in\Lambda\}$ (respectively,
$\{ \|\cdot\|_\mu : \mu\in\Lambda'\}$) denote the family of all continuous seminorms
on $M$ (respectively, $N$). By Lemma~\ref{lemma:bas_unif} (ii), a basic neighborhood of
$\sE(f)(u_x)=f(u)_x$ has the form $\sT(V,f(u),\mu,\eps)$,
where $V\subset X$ is an open neighborhood of $x$, $\mu\in\Lambda'$, and $\eps>0$.
Since $f$ is continuous, there exists $\lambda\in\Lambda$ such that for each $v\in M$ we have
$\| f(v)\|_\mu=\| v\|_\lambda$. By passing to the quotients, we see that
$\| f_y(v_y)\|_\mu\le \| v_y\|_\lambda$ ($v\in M$, $y\in X$).
We claim that
\begin{equation}
\label{Ef_cont}
\sE(f)\bigl(\sT(V,u,\lambda,\eps)\bigr)\subset \sT(V,f(u),\mu,\eps).
\end{equation}
Indeed, for each $v_y\in \sT(V,u,\lambda,\eps)$, where $v\in M$ and $y\in V$,
we have
\[
\| \sE(f)(v_y)-f(u)_y\|_\mu
=\| f_y(v_y-u_y)\|_\mu
\le \| v_y-u_y\|_\lambda
<\eps.
\]
This implies \eqref{Ef_cont} and shows that $\sE(f)$ is continuous.
Finally, letting $u=0$ and $V=X$ in \eqref{Ef_cont}, we conclude that $\sE(f)$
satisfies $\mathrm{(BM3)}$.
\end{proof}

Summarizing, we obtain the following.

\begin{theorem}
\label{thm:Kmod-Bnd}
There exists a functor $\sE\colon K\lmod\to\Bnd(X)$ uniquely determined by the following
properties:
\begin{mycompactenum}
\item
for each $M\in K\lmod$ and each $x\in X$, we have $\sE(M)_x=M_x$;
\item
the locally convex uniform vector structure on $\sE(M)$ is determined by $\cN_M$;
\item
for each $M\in K\lmod$ and each $u\in M$, the section $\tilde u\colon X\to\sE(M),\; x\mapsto u_x$,
is continuous;
\item
for each morphism $f\colon M\to N$ in $K\lmod$ and each $x\in X$, we have
$\sE(f)_x=f_x$.
\end{mycompactenum}
\end{theorem}

For the purposes of Section~\ref{sect:deforms}, we need locally convex bundles with an additional
algebraic structure.
Let $X$ be a set, and let $(A,p)$ be a family of vector spaces over $X$. Suppose that
each fiber $A_x\; (x\in X)$ is endowed with an algebra structure, and let
\begin{equation}
\label{mult}
\mult\colon A\times_X A \to A, \quad (u,v)\mapsto uv.
\end{equation}
For each pair $S,T$ of subsets of $A$, we let $S\cdot T=\mult((S\times T)\cap (A\times_X A))$.

\begin{definition}
Let $X$ be a topological space. By a {\em locally convex algebra bundle} over $X$ we mean
a locally convex bundle $(A,p,\cU)$ over $X$ together with an algebra structure on each fiber
$A_x\; (x\in X)$ such that
\begin{mycompactenum}
\item[$\mathrm{(B3)}$]
the multiplication \eqref{mult} is continuous;
\item[$\mathrm{(B4)}$]
for each $U\in\cU$ there exists $V\in\cU$ such that $V\cdot V\subset U$.
\end{mycompactenum}
If, in addition, each fiber $A_x$ is a Fr\'echet algebra with respect to the topology inherited from $A$,
then we say that $(A,p,\cU)$ is a {\em Fr\'echet algebra bundle}.
If $(A,p,\cU)$ and $(A',p',\cU')$ are locally convex algebra bundles over $X$,
then a bundle morphism $f\colon A\to A'$ is an {\em algebra bundle morphism}
if the restriction of
$f$ to each fiber $A_x\; (x\in X)$ is an algebra homomorphism from $A_x$ to $A'_x$.
\end{definition}

The category of all locally convex algebra bundles and algebra bundle morphisms over $X$
will be denoted by $\AlgBnd(X)$.

We need the following modification of Proposition~\ref{prop:sec_bnd}.

\begin{prop}
\label{prop:sec_alg_bnd}
Under the conditions of Proposition~{\upshape\ref{prop:sec_bnd}}, suppose that
each $E_x$ is endowed with an algebra structure and that the following
holds:
\begin{mycompactenum}
\item[$\mathrm{(S3)}$]
For each $\lambda\in\Lambda$ there exist $\mu\in\Lambda$ and $C>0$ such that
for each $x\in X$ and each $u,v\in E_x$ we have
$\| uv\|_\lambda\le C\| u\|_\mu \| v\|_\mu$;
\item[$\mathrm{(S4)}$]
$\Gamma$ is a subalgebra of $S(X,E)$.
\end{mycompactenum}
Then the bundle $(E,p,\cU_\cN)$ constructed in Proposition~{\upshape\ref{prop:sec_bnd}}
is a locally convex algebra bundle.
\end{prop}
\begin{proof}
Condition $\mathrm{(S3)}$ implies that for each $\eps>0$ we have
$\sU_{\mu,\delta}\cdot \sU_{\mu,\delta}\subset \sU_{\lambda,\eps}$,
where $\delta=\min\{\eps/C,1\}$. Thus $\mathrm{(B4)}$ holds.
To prove $\mathrm{(B3)}$,
let $u,v\in E\times_X E$, and let $\sT(V,c,\lambda,\eps)$ be a basic
neighborhood of $uv$, where $V\subset X$ is an open set, $c\in \Gamma$,
$\lambda\in\Lambda$, and $\eps>0$.
Find $\mu\in\Lambda$ and $C>0$ satisfying $\mathrm{(S3)}$.
Fix any $\eps'\in (0,\eps)$ such that $uv\in \sT(V,c,\lambda,\eps')$, and choose
$\delta>0$ such that
\begin{equation}
\label{delta_cond}
C\delta(\| u\|_\mu+\| v\|_\mu+5\delta)<\eps-\eps'.
\end{equation}
Let $x=p(u)=p(v)$. By $\mathrm{(S3)}$, the multiplication on $E_x$ is
continuous, so $\mathrm{(S1)}$ implies that there exist $a,b\in \Gamma$ such that
\begin{gather}
\label{ab_cond_1}
\| u-a_x\|_\mu<\delta,\quad
\| v-b_x\|_\mu<\delta,\\
\label{ab_cond_2}
a_x b_x\in \sT(V,c,\lambda,\eps').
\end{gather}
By $\mathrm{(S4)}$ and \eqref{ab_cond_2},
there exists an open neighborhood $W$ of $x$ such that $W\subset V$ and
\begin{equation}
\label{V_cond_1}
(ab)(W)\subset \sT(V,c,\lambda,\eps').
\end{equation}
By shrinking $W$ if necessary and by using $\mathrm{(S2)}$, we can also assume that
\begin{equation}
\label{V_cond_2}
\| a_y\|_\mu<\| a_x\|_\mu+\delta, \quad
\| b_y\|_\mu<\| b_x\|_\mu+\delta \qquad (y\in W).
\end{equation}
By \eqref{ab_cond_1}, $\sT(W,a,\mu,\delta)$ (respectively, $\sT(W,b,\mu,\delta)$)
is a neighborhood of $u$ (respectively, $v$). We claim that
\begin{equation}
\label{m_E_cont}
\sT(W,a,\mu,\delta)\cdot \sT(W,b,\mu,\delta)\subset \sT(V,c,\lambda,\eps).
\end{equation}
Indeed, let $y\in W$, and let $u',v'\in E_y$ be such that
$\| u'-a_y\|_\mu<\delta$ and $\| v'-b_y\|_\mu<\delta$. We have
\begin{align*}
\| u'v'-c_y\|_\lambda
&\le \| u'v'-a_y b_y\|_\lambda+\| a_y b_y-c_y\|_\lambda\\
&\le \| (u'-a_y)v'\|_\lambda+\| a_y(v'-b_y)\|_\lambda+\eps'
&& \text{(by \eqref{V_cond_1})} \\
&\le C \| u'-a_y\|_\mu \| v'\|_\mu+C\| a_y\|_\mu\| v'-b_y\|_\mu+\eps'
&& \text{(by $\mathrm{(S3)}$)} \\
&<C\delta(\| a_y\|_\mu+\| v'\|_\mu)+\eps'\\
&<C\delta(\| a_y\|_\mu+\| b_y\|_\mu+\delta)+\eps'\\
&<C\delta(\| a_x\|_\mu+\| b_x\|_\mu+3\delta)+\eps'
&& \text{(by \eqref{V_cond_2})} \\
&<C\delta(\| u\|_\mu+\| v\|_\mu+5\delta)+\eps'
&& \text{(by \eqref{ab_cond_1})} \\
&<(\eps-\eps')+\eps'=\eps
&& \text{(by \eqref{delta_cond})}.
\end{align*}
Thus \eqref{m_E_cont} holds, which implies $\mathrm{(B3)}$ and completes the proof.
\end{proof}

Let $K$ be a commutative algebra. By a {\em locally convex $K$-algebra} we mean
a locally convex $K$-module $A$ together with a continuous $K$-bilinear associative
multiplication ${A\times A}\to A$.
Morphisms of locally convex $K$-algebras are defined
in the obvious way. The category of locally convex $K$-algebras will be denoted by $K\lalg$.
As above, let $X$ be a topological space, and let
$\gamma\colon K\to C(X)$ be an algebra homomorphism.
Observe that for each locally convex $K$-algebra $A$ and each
$x\in X$ the subspace $\ol{\fm_x A}$ is a two-sided ideal of $A$.
Thus the fiber $A_x=A/\ol{\fm_x A}$ of $A$ over $x$ is a locally convex algebra
in a natural way. Let $\cC_A=\{ \|\cdot\|_\lambda : \lambda\in\Lambda\}$ denote the family of all
continuous seminorms on $A$.
Since the multiplication on $A$ is continuous, it follows that for each $\lambda\in\Lambda$
there exist $\mu\in\Lambda$ and $C>0$ such that for all $a,b\in A$ we have
$\| ab\|_\lambda\le C\| a\|_\mu \| b\|_\mu$. By passing to the quotient seminorms
on the fibers $A_x\; (x\in X)$, we see that
$\mathrm{(S3)}$ holds (with $E_x=A_x$).
Clearly, $\Gamma_A$ satisfies $\mathrm{(S4)}$.
Applying Proposition~\ref{prop:sec_alg_bnd},
we conclude that $\sE(A)$ is a locally convex algebra bundle.
Moreover, if $f\colon A\to B$ is a locally convex $K$-algebra morphism, then
$\sE(f)\colon\sE(A)\to\sE(B)$ is an algebra bundle morphism. Thus we have the following
analog of Theorem~\ref{thm:Kmod-Bnd}.

\begin{theorem}
\label{thm:Kalg-Bnd}
There exists a functor $\sE\colon K\lalg\to\AlgBnd(X)$ uniquely determined by
the following properties:
\begin{mycompactenum}
\item
for each $A\in K\lalg$ and each $x\in X$, we have $\sE(A)_x=A_x$;
\item
the locally convex uniform vector structure on $\sE(A)$ is determined by $\cN_A$;
\item
for each $A\in K\lalg$ and each $u\in A$, the section $\tilde u\colon X\to\sE(A),\; x\mapsto u_x$,
is continuous;
\item
for each morphism $f\colon A\to B$ in $K\lalg$ and each $x\in X$, we have
$\sE(f)_x=f_x$.
\end{mycompactenum}
\end{theorem}

In conclusion, let us discuss the notion of continuity for locally convex bundles.
Let $(E,p,\cU)$ be a locally convex bundle over a topological space $X$. Recall that
we always have $\cU=\cU_{\cN}$, where $\cN$ is an admissible directed family of seminorms on $E$
(see Proposition~\ref{prop:bundles_via_seminorms}).
By Remark~\ref{rem:up_semi_open}, each seminorm belonging to $\cN$
is upper semicontinuous.
In the theory of Banach bundles (see, e.g., \cite{FD}), it is
customary to consider only those Banach bundles whose norm is a continuous
function on $E$. This leads naturally to the following definition.

\begin{definition}
\label{def:cont_bnd}
We say that a locally convex bundle $(E,p,\cU)$ is {\em continuous} if there exists
an admissible directed family $\cN$ of continuous seminorms on $E$ such that
$\cU=\cU_{\cN}$.
\end{definition}

The following result gives a convenient way of proving the
continuity of locally convex bundles.

\begin{prop}
\label{prop:cont_semi}
Let $(E,p,\cU)$ be a locally convex bundle over a topological space $X$, and let
$\Gamma$ be a vector subspace of $\Gamma(X,E)$ such that for each $x\in X$
the set $\{ s_x : s\in \Gamma\}$ is dense in $E_x$. Suppose that $\|\cdot\|$
is an upper semicontinuous seminorm on $E$ such that for every $s\in\Gamma$
the map $X\to\R,\; x\mapsto \| s_x\|$, is continuous. Then $\|\cdot\|$ is continuous.
\end{prop}
\begin{proof}[Proof {\upshape (cf. \cite{Williams}, the last step of the proof of C.25)}]
We have to show that $\|\cdot\|$ is lower semicontinuous.
Let $u\in E$, let $x=p(u)$, and suppose that $\| u\|>c>0$. Choose $\delta>0$ such that
$\| u\|>c+2\delta$. Since $\|\cdot\|$ is upper semicontinuous and the topology on $E_x$
is translation invariant, it follows that $\{ v\in E_x : \| v-u\|<\delta\}$ is open in $E_x$.
Hence there exists $s\in\Gamma$ such that $\| u-s_x\|<\delta$.
In particular, $\| s_x\|>c+\delta$. By assumption, this implies that there exists
an open neighborhood $V$ of $x$ such that $\| s_y\|>c+\delta$ for all $y\in V$.
By Remark~\ref{rem:up_semi_open}, $\sT(V,s,\|\cdot\|,\delta)$ is an open
neighborhood of $u$. If now $v\in \sT(V,s,\|\cdot\|,\delta)$ and $y=p(v)$, then
\[
\| v\|\ge \| s_y\|-\| s_y-v\|>c+\delta-\delta=c.
\]
This implies that $\|\cdot\|$ is lower semicontinuous and completes the proof.
\end{proof}

\end{document}